\documentclass[%
DIV=10, 
english, 
titlepage=false, 
fontsize=11pt,
]{scrartcl}

\usepackage[utf8]{inputenc}
\usepackage[T1]{fontenc}
\usepackage{amsmath}
\usepackage{amssymb}
\usepackage{amsthm}
\usepackage{amsfonts}
\usepackage{babel}
\usepackage{lmodern}
\usepackage{csquotes}
\usepackage{enumitem}
\usepackage{array}
\usepackage{amsfonts}
\usepackage{mathtools}
\usepackage{esint}
\usepackage{array}
\usepackage{authblk}
\usepackage{nccmath}
\usepackage{bold-extra}
\usepackage{gensymb}
\usepackage{cite}
\usepackage{tikz}
\usepackage{placeins}
\usetikzlibrary{hobby}
\usetikzlibrary{backgrounds}

\usepackage{pgf}
\usepackage{pgfplots}
\pgfplotsset{compat=newest}
\usepackage{xcolor}
\usetikzlibrary{plotmarks,backgrounds,patterns}
\usetikzlibrary{external}
\tikzset{external/only named=true}

\usepackage{color}
\definecolor{LinkColor}{rgb}{0,0,1}
\definecolor{LinkColor2}{rgb}{0,0.5,0}
\definecolor{lg}{rgb}{.5,.5,.5}

\usepackage[%
left=1.25in,
right=1.25in,
bottom=1.2in,
top=1in,
footskip=0.5in,
]{geometry}

\usepackage[%
pdftitle={Titel},%
pdfauthor={Autor},%
pdfcreator={LaTeX, LaTeX with hyperref and KOMA-Script},
pdfsubject={Betreff}, 
pdfkeywords={Keywords}
]{hyperref} 
\hypersetup{%
    colorlinks	=true,
    linkcolor	=LinkColor,%
    citecolor	=LinkColor2,%
    urlcolor	=LinkColor,%
}

\setkomafont{section}{\centering\Large\rmfamily\normalfont\scshape}
\setkomafont{subsection}{\large\rmfamily\normalfont\bfseries}
\setkomafont{subsubsection}{\rmfamily\normalfont\bfseries}
\rmfamily
\numberwithin{equation}{section}

\makeatletter
\renewcommand{\@seccntformat}[1]{\csname the#1\endcsname.\hspace{1ex}}
\makeatother

\newtheorem*{Abs}{Abstract}
\newtheorem{Thm}{Theorem}[section]
\newtheorem{Lem}[Thm]{Lemma}

\newtheorem{Cor}[Thm]{Corollary}
\newtheorem{Def}[Thm]{Definition}

\theoremstyle{definition}
\newtheorem{Rem}[Thm]{Remark}

\makeatletter
\renewenvironment{proof}[1][\proofname]{%
    \par\pushQED{\qed}\normalfont%
    \topsep6\p@\@plus6\p@\relax
    \trivlist\item[\hskip\labelsep\bfseries#1\@addpunct{.}]%
    \ignorespaces
}{%
    \popQED\endtrivlist\@endpefalse
}
\makeatother

\makeatletter
\renewcommand\paragraph{\@startsection{paragraph}{4}{\z@}%
    {1ex \@plus1ex \@minus.2ex}%
    {-1em}%
    {\normalfont\normalsize\bfseries}}
\renewcommand\subparagraph{\@startsection{paragraph}{4}{\z@}%
    {1ex \@plus1ex \@minus.2ex}%
    {-1em}%
    {\normalfont\normalsize\itshape}}
\makeatother

\newcommand{\abs}[1]{\left\vert#1\right\vert}

\newcommand{\bigabs}[1]{\big\vert#1\big\vert}

\newcommand{\ap}[3]{\left(#1,#2\right)_{a_{\varepsilon}(#3)}}

\newcommand{\bp}[3]{\left(#1,#2\right)_{b_{\varepsilon}(#3)}}

\newcommand{\B}[1]{\boldsymbol{#1}}

\newcommand{\cp}[3]{\left(#1,#2\right)_{c_{\varepsilon}(#3)}}

\newcommand{\ee}[1]{#1^{\varepsilon}}

\newcommand{\ie}[1]{#1_{\varepsilon}}

\newcommand{\io}[1]{\int_{\Omega}#1\textup{\,d}x}

\newcommand{\ito}[1]{\int_{\tOmega}#1\textup{\,d}x}

\newcommand{\HO}{H^1(\Omega)}

\newcommand{\HOO}{H^1_{(0),\varphi}(\Omega)}

\newcommand{\Hz}{H^1_0(\Omega)}

\newcommand{\limsupse}{\underset{\varepsilon\searrow 0}{\limsup\,}}

\newcommand{\liminfse}{\underset{\varepsilon\searrow 0}{\liminf\,}}

\newcommand{\limse}{\underset{\varepsilon\searrow 0}{\lim\,}}

\newcommand{\Li}{L^{\infty}(\Omega)}

\newcommand{\Lu}{L^{\infty}(\Omega)}

\newcommand{\Lzc}{L^2_{\varphi}(\Omega)}

\newcommand{\norm}[1]{\left\Vert #1\right\Vert}

\newcommand{\tOmega}{\tilde{\Omega}}

\newcommand{\tve}{\tilde{\varphi}_{\eps}}

\newcommand{\dLL}{\;\mathrm d x}

\newcommand{\R}{\mathbb{R}}

\newcommand{\N}{\mathbb{N}}

\newcommand{\eps}{\varepsilon}

\newcommand{\bx}{\B{x}}

\usepackage{selinput}

\begin{document} 

\begin{titlepage}
	\begin{addmargin}{0.5 cm}
		\centering
		\LARGE{\scshape Phase-Field Methods for Spectral\\ Shape and Topology Optimization}\\
		\rmfamily\mdseries
		\vspace{0.02\paperheight} 
		\normalsize
        \today\\
        \vspace{0.02\paperheight} 
		\textsc{Harald Garcke\footnote{Fakul\"at f\"ur Mathematik, Universit\"at Regensburg, 93053 Regensburg, Germany. \href{mailto:harald.garcke@ur.de}{harald.garcke@ur.de},
        \href{mailto:paul.huettl@ur.de}{paul.huettl@ur.de},
        \href{mailto:patrik.knopf@ur.de}{patrik.knopf@ur.de}}, Paul H\"uttl\footnotemark[1], 
        Christian Kahle\footnote{Mathematisches Institut, Universität Koblenz, 56070 Koblenz, Germany. \href{mailto:christian.kahle@uni-koblenz.de}{christian.kahle@uni-koblenz.de}},
        Patrik Knopf\footnotemark[1]
        \newline
        and Tim Laux\footnote{Hausdorff Center for Mathematics, 
        University of Bonn, 53115 Bonn, Germany. \href{mailto:tim.laux@hcm.uni-bonn.de}{tim.laux@hcm.uni-bonn.de}}
        }
        \\[1ex]
		\smallskip 
		\begin{center}
			\footnotesize
			{
				\textit{This is a preprint version of the paper. Please cite as:} \\  
				H. Garcke, P. H\"uttl, C. Kahle, P. Knopf and Tim Laux,\\
                \textit{ESAIM Control Optim. Calc. Var.} \textbf{29} (2023), Paper No. 10 \\ 
				\url{https://doi.org/10.1051/cocv/2022090}
			}
		\end{center}
		\smallskip
		\small
		\begin{Abs}
		\normalfont
		We optimize a selection of eigenvalues of the Laplace operator with Dirichlet or Neumann boundary conditions by adjusting the shape of the domain on which the eigenvalue problem is considered. Here, a phase-field function is used to represent the shapes over which we minimize.
		The idea behind this method is to modify the Laplace operator by introducing phase-field dependent coefficients in order to 
		extend the eigenvalue problem on a fixed design domain containing all admissible shapes. The resulting shape and topology optimization problem can then be formulated as an optimal control problem with PDE constraints in which the phase-field function acts as the control.
		For this optimal control problem, we establish first-order necessary optimality conditions and we rigorously derive its sharp interface limit. 
		Eventually, we present and discuss several numerical simulations for our optimization problem.\\[1ex]
		\end{Abs}
		\\
	\end{addmargin}
\end{titlepage}


\bigskip
\normalsize
\setlength{\parskip}{1ex}
\setlength{\parindent}{0ex}
\allowdisplaybreaks


\normalsize

\section{Introduction}\label{SEC:Intro}
Optimization linking the shape and topology of a domain to the eigenvalues of an elliptic operator is a fascinating field leading to attractive mathematical problems. At the same time this field has many applications as for instance the mechanical stability of vibrating objects, thermic properties of bodies and wave propagation. However, as stated by Henrot in  \cite{Henrot}, ``\dots, they are very simple to state and generally hard to solve!’’

There is a rich literature on classical shape optimization in the sense that the shape itself is controlling the problem, i.e., the quantities that are varied along the optimization process are themselves domains, which are referred to as \textit{shapes}, see \cite{Allaire2,Buttazzo,Bucur,Henrot,Henrot2}.

A fundamental problem is to optimize the eigenvalues of the Laplacian with either homogeneous Dirichlet or homogeneous Neumann boundary condition by adjusting the shape $D$ on which the eigenvalue problem is considered. To be precise, the overall goal is to find a shape $D$ such that a selection of eigenvalues of the problem
\begin{align}\tag{CL}\label{EWP}
	\left\{
    \begin{aligned}
        -\Delta w &=\lambda w &&\quad\text{in }D,\\
        w\vert_{\partial D} &= 0 &&\quad\text{on }\partial D,
    \end{aligned}
	\right.
	\qquad\text{or}\qquad
	\left\{
	\begin{aligned}
		-\Delta w &=\mu w &&\quad\text{in }D,\\
		\partial_{\B{n}} w &= 0 &&\quad\text{on }\partial D,
	\end{aligned}
	\right.
\end{align}
is optimal in a certain sense. In the following, we refer to \eqref{EWP} as the \textit{classical eigenvalue problem}.

In this paper, we consider a phase-field approximation to optimize a selection of eigenvalues of the classical problem \eqref{EWP}. Our approach shares similarities with the one in \cite{BogoselOudet} but allows for more general penalizing terms in the approximate eigenvalue problem and is able to deal with additional volume constraints and pointwise constraints. 
Furthermore, our analysis does not rely on any further knowledge of sets minimizing the sharp interface problem such as boundedness or openness, as discussed in \cite{DePhilippisVelichkov}. We will now briefly explain the general strategy.

Instead of directly varying the shape $D$, we formulate an approximate eigenvalue problem on a fixed domain $\Omega$, which is called the \textit{design domain}. This domain $\Omega$ comprises any admissible shape, which is now implicitly represented by a phase-field. 
More precisely, instead of interpreting the shape as the unknown quantity in the optimization process, we describe it via the phase-field $\varphi_{\eps}$ which attains the value $+1$ in most parts of the shape and the value $-1$ in most parts of its complement (with respect to the design domain $\Omega$). Close to the boundary of the shape, the phase-field exhibits a diffuse interfacial layer whose thickness is proportional to the interface parameter $\varepsilon>0$. Here a transition between the values $+1$ and $-1$ takes place.

The mentioned approximation is carried out by 
considering an elliptic eigenvalue problem on the whole design domain whose differential operator exhibits phase-field dependent coefficient functions. 
These coefficients are chosen in such a way that the corresponding classical eigenvalue problem \eqref{EWP} is approximately fulfilled on the shape $D_{\eps}=\{\ie{\varphi}\ge 0\}.$ For a more detailed discussion we refer to 
the next section, and especially to Subsection~\ref{Sub:SIAInt}, for explicit choices of the coefficients which ensure that \eqref{EWP} holds on the shape $D=\{\varphi_0=1\}$ in the sharp interface limit $\eps\to 0$.
For an overview of analytical results in the sharp interface case concerning the stability of classical eigenvalue problems under variation of the shape, e.g., continuity of eigenvalues under domain perturbation and shape differentiability, we refer to the books \cite{AttouchButtazzo, Bucur,Henrot, HenrotPierre}. There, the concept of $\gamma$-convergence, which is exactly the notion of convergence of sets under which eigenvalues are continuous, is a key tool. This concept is also needed in our framework as we need to approximate sets of finite perimeter with suitable smooth sets in order to construct a recovery sequence in the $\Gamma$-limit. Here, we will proceed in a similar fashion as in \cite{BogoselOudet}, see Section \ref{Sec:RigSIA}.

Before we present this approach in detail let us review some properties of eigenvalue optimization problems for the Laplace operator.
Let us consider a rectangular shape $D=D_{L,l}=(0,L)\times(0,l)$ with $L,l>0$. Then, the eigenvalues of \eqref{EWP}
are given as
\begin{align}\label{EV2d}
    \pi^2\left(\frac{m^2}{L^2}+\frac{k^2}{l^2}\right).
\end{align}
The difference is that in the Dirichlet case, $m$ and $k$ range over all positive integers, whereas in the Neumann case, $m=0$ and $k=0$ are also taken into account. 
The corresponding $L^2(D)$-normalized eigenfunctions are
\begin{align*}
        w^{D}_{m,k}(x,y)
    =
        \frac{2}{\sqrt{Ll}}
        \sin\left(\frac{m\pi x}{L}\right)
        \sin\left(\frac{k\pi y}{l}\right) 
   \quad m,k\ge 1,
\end{align*}
in the Dirichlet case,
and 
\begin{align*}
        w^{N}_{m,k}(x,y)
    =
        \frac{2}{\sqrt{Ll}}
        \cos\left(\frac{m\pi x}{L}\right)
        \cos\left(\frac{k\pi y}{l}\right)
   \quad m,k\ge 0,
\end{align*}
in the Neumann case. Considering a class of rectangles with fixed area, we hence observe that the smallest non-trivial Neumann eigenvalue approaches zero for very thin and long rectangles, which is not the case for eigenvalues of the Dirichlet problem. This indicates that maximizing the smallest non-trivial Neumann eigenvalue and minimizing the smallest Dirichlet eigenvalue are meaningful problems.

As explained above, classical variational problems vary the shape $D$ in order to optimize the behavior of the smallest non-trivial Dirichlet or Neumann eigenvalue.
In the Dirichlet case, the theorem of Faber--Krahn (see, e.g., \cite[Thm.~3.2.1]{Henrot}) states that for any fixed constant $c>0$, we have
\begin{align}\label{DirOpt}
        \lambda_1(B)
    =
        \min\left\{
            \lambda_1(D)\left| \,D \text{ open subset of }\mathbb{R}^n \text{ with} \abs{D}=c\right.
                \right\},
\end{align}
where $\abs{D}$ denotes the Lebesgue-measure of $D$ and $B$ is any ball in $\mathbb{R}^N$ with volume $c$. 
The analogon in the Neumann case is the theorem of Szeg\H{o}\, and Weinberger (see, e.g., \cite[Thm.~7.1.1]{Henrot}) which states that the first non-trivial eigenvalue satisfies
\begin{align}\label{NeuOpt}
        \mu_1(B)
    =
        \max\left\{
            \mu_1(D)
                \left| 
                    \begin{aligned}
                    &\,D \text{ open subset of }\mathbb{R}^n\text{ with}\\ 
                    &\text{ Lipschitz boundary and }
                    \abs{D}=c
                    \end{aligned}
                \right.
            \right\}.
\end{align}


Now, we briefly want to explain how the optimization problems in this paper are formulated. 
Let us fix an arbitrary $\eps>0$. For any phase-field $\varphi=\varphi_{\eps}$ belonging to some feasible set,
let $\left(\lambda^{\varepsilon,\varphi}_k\right)_{k\in\mathbb{N}}$ and $\left(\mu^{\varepsilon,\varphi}_k\right)_{k\in\mathbb{N}}$ denote the eigenvalues of the approximate equations on the design domain with Dirichlet and Neumann boundary data, respectively. In the Dirichlet case, we minimize the functional
\begin{align}\label{minD}
        J_{\varepsilon}^D(\varphi)
    &=
        \Psi(\lambda_{i_1}^{\varepsilon,\varphi},\dots,\lambda_{i_l}^{\varepsilon,\varphi})
        +\gamma E_{\text{GL}}^{\varepsilon}(\varphi),
\end{align}
and in the Neumann case, we minimize the functional
\begin{align}\label{minN}        
         J_{\varepsilon}^N(\varphi)
     &=
        \Psi(\mu_{i_1}^{\varepsilon,\varphi},\dots,\mu_{i_l}^{\varepsilon,\varphi})
                +\gamma E_{\text{GL}}^{\varepsilon}(\varphi),
\end{align}
where the indices $i_1,\dots,i_l\in \mathbb{N}$ pick eigenvalues from the above sequences and $\gamma>0$ is a given constant.
Here, $E_{\text{GL}}^{\varepsilon}$ stands for the \textit{Ginzburg--Landau energy} which is defined as
\begin{align*}
        E_{\text{GL}}^{\varepsilon}(\varphi)
    =
        \int_{\Omega}
            \left(
                \frac{\varepsilon}{2}\abs{\nabla\varphi}^2
                +\frac{1}{\varepsilon}\psi(\varphi)
            \right)
        \text{\,d}x,
\end{align*}
where $\varepsilon>0$ corresponds to the thickness of the diffuse interface and $\psi$ is a bulk potential that usually has a double-obstacle structure (cf. Section~\ref{GinLan}). In the framework of $\Gamma$-limits, this regularizing term can be understood as an approximation of the perimeter functional penalizing the free boundary of the shape
(see, e.g., \cite{Ambrosio}) and is needed for the optimization problems to be well-posed.
The function $\Psi:(\mathbb{R}_{>0})^l\to \mathbb{R}$ is used to formulate quite different extremal problems involving the above selection of eigenvalues,  such as linear combinations, see Section~\ref{SOPT}.
For a rigorous investigation of the $\Gamma$-limit, a componentwise monotonicity assumption on $\Psi$ will be crucial in order to preserve the monotonicity of eigenvalues with respect to set inclusion.

The goal of the paper at hand is to analyze the optimization problems \eqref{minD} and \eqref{minN}, to rigorously study the sharp interface limit for the Dirichlet problem, and to present and discuss several numerical simulations illustrating this theoretical discussion.

We point out that various methods are proposed in the literature to deal with this kind of shape and topology optimization problems on the sharp interface level. The most common one is the method of boundary variation (see \cite{Murat-Simon, Simon, Sokolowski}).
However, this approach does not allow for topological changes.
This can be overcome by the homogenization method (see \cite{AllaireBook}) or the more special SIMP method (see \cite{Bendsoe}).

A fruitful approach studied also in the context of spectral optimization is the level-set method, see e.g., \cite{Allaire1, Oudet}. Here the evolution of the shape is governed by a Hamilton--Jacobi equation and velocities to evolve the shape are computed in the framework of shape derivatives. This method has moderate numerical costs and also allows for topological changes but the creation of new holes can be challenging, although it can be handled by incorporating topological derivatives, see \cite{Burger}.

We want to mention that a further development of the level-set method from \cite{Oudet} is given in \cite{Antunes, AntunesOudet-NumericalResultsExtremalProblemsEigenvaluesLaplace} which evolves the shape by a gradient descent. The optimization in \cite{Antunes} is restricted to star-shaped domains whereas \cite{AntunesOudet-NumericalResultsExtremalProblemsEigenvaluesLaplace} also allows for multiple connected components. The works \cite{Oudet, Antunes, AntunesOudet-NumericalResultsExtremalProblemsEigenvaluesLaplace} have in common that they make use of a so called genetic algorithm (see \cite{OudetPhd}) in order to obtain a suitable initial guess for the shape.

In the present paper, we combine our phase-field approach with the \texttt{VMPT} method (see \cite{BlankRupprecht}), which is a generalization of the projected gradient method. The main advantage concerning numerical implementation is that it naturally allows for topology changes and especially the nucleation of holes.  
Our approach is capable of reproducing the results of \cite{AntunesOudet-NumericalResultsExtremalProblemsEigenvaluesLaplace}
but is also able to tackle a large variety of spectral shape optimization problems that go far beyond the classical ones in the literature. For instance, our method allows for cost functionals involving linear combinations of eigenvalues and even (at least numerically) for simultaneous maximization and minimization of a selection of eigenvalues. Furthermore, the phase-field method enables us to fix the shape in certain regions but also to place obstacles for the shape. This means we can initially prescribe regions in the design domain that must or must not be covered by the shape, respectively, see Section~\ref{sec:num}.
The phase-field approach was also used numerically in \cite{BogoselOudet} where no additional volume constraint is imposed.

Our paper is structured as follows. In Section~\ref{Sec:Form}, we give a mathematically precise formulation of the model and the optimization problems. Afterwards, in Section~\ref{SEC:Analysis}, we study the diffuse interface problem which allows us to state first-order necessary optimality conditions.  Here, the assumption that the considered eigenvalues are simple will be essential, as otherwise Fr\'echet-differentiability cannot be guaranteed. 
In Section~\ref{SEC:SIDir}, we perform an in-depth analysis of the sharp interface problem in the Dirichlet case.
In particular, we show the continuity of spectral quantities when passing to the limit $\eps\to 0$ under a certain rate condition.
Eventually, we prove an unconditional $\Gamma$-convergence result for the involved cost functionals. In the first part of our proof, we proceed in the spirit of \cite{BogoselOudet}, which contains many highly valuable ideas and key steps. Nevertheless, we present a detailed proof in order to avoid some inconsistencies of \cite{BogoselOudet} concerning the usage of the perimeter.
Moreover, the inclusion of an additional volume constraint in our framework requires a further in-depth analysis. We point out that we cannot prove a rigorous $\Gamma$-convergence result in the Neumann case as a certain coefficient in the approximating phase-field model degenerates in the limit $\eps\to 0$.
Based on the results of Section~\ref{SEC:Analysis}, we present several numerical simulations in Section~\ref{sec:num}, which exemplify optimal shapes in concrete situations.

\section{Formulation of the problem}\label{Sec:Form} 
In this section we introduce the mathematical model and the optimization problems. 

\subsection{The design domain and the phase-field variable}

We fix a bounded design domain $\Omega\subset \R^{d}$, $d\in \N$ with Lipschitz boundary.
For any suitable open shape $D\subseteq \Omega$, we consider the classical eigenvalue problem \eqref{EWP}, i.e., the eigenvalue equation for the Laplacian with either a homogeneous Dirichlet or a homogeneous Neumann boundary condition. 
The goal of our shape and topology optimization problems is to minimize a cost functional involving a selection of eigenvalues by adjusting the shape $D$ in an optimal way.

To approximate the shape $D$, we introduce a phase-field function $\ie{\varphi}: \Omega \to [-1,1]$ which attains the value $+1$ in most parts of $D$ and the value $-1$ in most parts of the relative complement $\Omega\backslash D$. The free boundary $\partial D$ is approximated by a thin interfacial layer in which $\ie{\varphi}$ continuously changes its values between $-1$ and $+1$. The thickness of this so-called diffuse interface is proportional to a small parameter $\eps>0$.

In the sharp interface limit $\varepsilon\to 0$, where the thickness of the diffuse interface is sent to zero, $\left\{\varphi_0= 1\right\}$ exactly represents the shape $D$ on which the eigenvalue problem \eqref{EWP} is satisfied. In this case, the set $\left\{\varphi_0= -1\right\}$ represents the relative complement $\Omega\backslash D$ which is not involved in the eigenvalue problem \eqref{EWP}, and $\partial D$ can be expressed as $\Omega\cap \partial \left\{\varphi_0= -1\right\} = \Omega\cap \partial \left\{\varphi_0= 1\right\}$. When considering the problem for a fixed $\eps>0$, we will often omit the index $\eps$ in the notation for the phase-field, i.e., we will just write $\varphi$ instead of $\varphi_\eps$ .

We further want to prescribe regions within the design domain $\Omega$ where the pure phases are prescribed. 
Mathematically speaking, we fix two disjoint measurable sets
$S_0,S_1\subset \Omega$ such that $\tOmega\coloneqq \Omega\backslash (S_0\cup S_1)$ is a domain with Lipschitz boundary.
This is, for example, the case if we choose $S_0$ and $S_1$ as closed balls which keep a fixed, positive distance between themselves and towards the boundary $\partial\Omega$ of the design domain.
Thus, the set representing these constraints on $\varphi$ is given as
\begin{align}\label{UCons}
        \mathcal{U}
    \coloneqq
        \left\{
            \left.
                \varphi\in L^1(\Omega)\right|
                \varphi=-1\text{ a.e.~on } S_0\text{ and }\varphi=1\text{ a.e.~on } S_1
        \right\}.
\end{align}
As the values of $\varphi\in \mathcal{U}$ are fixed on $S_0\cup S_1$, the relevant set in our optimization process is given as $\tOmega$ which will play also an important role in the definition of the Ginzburg--Landau energy.

We additionally prescribe general bounds on the mean value of $\varphi$ in order to take volume constraints into account. To this end, we impose the general constraint 
\begin{align*}
     \beta_1\bigabs{\tOmega}
     \le \int_{\tOmega}\varphi\text{\,d}x
     \le \beta_2\bigabs{\tOmega},
\end{align*}
with $\beta_1,\beta_2\in\mathbb{R}$, $\beta_1\le\beta_2$ and $\beta_1,\beta_2\in (-1,1)$. This condition $\beta_1,\beta_2\in (-1,1)$ ensures that in the sharp interface case, where $\varphi\in BV(\tOmega,\left\{\pm 1\right\})$, the sets $\left\{\varphi=1\right\}\cap\tOmega$ and $\left\{\varphi=-1\right\}\cap\tOmega$ both have strictly positive measure. Hence, the trivial cases are excluded.
Note that for $\beta_1=\beta_2$ we obtain an equality constraint for the mean value.

Furthermore, we require sufficient regularity of the phase-field, namely $H^1(\tOmega)$, in order for the Ginzburg--Landau energy (that will be introduced in the next subsection) to be well-defined.
All these constraints are summarized in the set
\begin{align}\label{GDef}
        \mathcal{G}^{\beta}
    =
        \left\{
            \varphi \in H^1(\tOmega)
                \left|\,
                    \abs{\varphi}\le 1, 
                    \beta_1\bigabs{\tOmega}
                    \le \int_{\tOmega}\varphi\text{\,d} x
                    \le \beta_2\bigabs{\tOmega}
                \right.
        \right\}.
\end{align}
We point out that for the upcoming analysis in Section~\ref{SEC:Analysis}, we could also include a constraint preventing the shape to touch the boundary, i.e., $\varphi=-1$ on $\partial\Omega$, which is used in the numerical simulations presented in Section~\ref{sec:num}. In order for this constraint to be well-defined in the trace sense, and not to interfere with the one formulated via $\mathcal{U}$ in \eqref{UCons}, we would also need to demand $S_1\subset\joinrel\subset \Omega$ to be compactly contained. Nevertheless, we do not include this constraint in the discussion of the sharp interface limit in Section~\ref{SEC:SIDir} as this would produce an additional term in the $\Gamma$-limit of the Ginzburg-Landau energy as explained in the following subsection.

\subsection{The Ginzburg--Landau energy}\label{GinLan}

For the definition of the objective functional and especially for the well-posedness of the minimization problem the so-called \textit{Ginzburg--Landau energy}
\begin{align*}
        E_{\text{GL}}^{\eps}(\varphi)
    =
        \int_{\tOmega}
            \left(
                \frac{\eps}{2}\abs{\nabla\varphi}^2+\frac{1}{\eps}\psi(\varphi)
            \right)\text{\,d}x, 
        \quad \eps>0,
\end{align*}
is crucial.
Here, the potential $\psi: \R\to \R\cup \left\{+\infty \right\}$ is assumed to have exactly the two global minimum points $-1$~and~$+1$ with 
\begin{align*}
    \underset{\R}{\min}\,\psi = \psi(\pm1) = 0,
\end{align*}
and a local maximum point in between (usually at zero).
Furthermore, $\psi$ is assumed to exhibit the decomposition $\psi(\varphi)=\psi_0(\varphi)+I_{[-1,1]}(\varphi)$ 
with $\psi_0\in C_{\text{loc}}^{1,1}(\R,\R)$ 
and $I_{[-1,1]}:\R\to \R\cup \left\{+\infty \right\}$ being the indicator functional
\begin{align*}
        I_{[-1,1]}(\xi)
     =
        \begin{cases}
            0&\text{if } \xi\in [-1,1],\\
            +\infty& \text{otherwise}.
        \end{cases}
\end{align*}
Common choices for $\psi_0$ would be the quadratic potential $\psi_0(\varphi)=\frac{1}{2}(1-\varphi^2)$ or the quartic potential $\psi_0(\varphi)=\frac{1}{2}(1-\varphi^2)^2$. For more details we refer to \cite{Blank,Garcke}.

In our optimization problems, we impose the phase-field constraint $\varphi\in \mathcal{G}^{\beta}\cap\mathcal{U}$. As in \cite{Blank,Garcke}, it thus suffices to include the regular part
\begin{align}\label{hatE}
        E^{\varepsilon}(\varphi)
    :=
        \int_{\tOmega}
            \left(
                \frac{\eps}{2}\abs{\nabla\varphi}^2+\frac{1}{\eps}\psi_0(\varphi)
            \right) \,\mathrm dx,                
\end{align}
of the Ginzburg--Landau energy in the cost functional, since $E^{\eps}(\varphi) = E_\text{GL}^{\eps}(\varphi)$ for all $\varphi\in \mathcal{G}^{\beta}$. 

Note that in the Ginzburg--Landau energy, we merely integrate over the domain $\tilde{\Omega}\subset\Omega$. If we demanded $\varphi\in H^1(\Omega)\cap \mathcal{U}$, we would obtain Dirichlet conditions
\begin{alignat*}{3}
    \varphi&=-1&&\quad\text{on }\partial S_0,\\
    \varphi&=1&&\quad\text{on }\partial S_1.
\end{alignat*}
This would produce an additional contact energy term in the $\Gamma$-limit of $E^{\varepsilon}$ as $\varepsilon\to 0$, see \cite{Owen}. In order to avoid this phenomenon and to obtain the classical $\Gamma$-limit as studied in \cite{Modica}, the energy and the $H^1$-regularity are restricted to the subset $\tOmega\subset \Omega$.

\subsection{The approximate eigenvalue problems}\label{Sub:coeff}
For any $\varepsilon>0$, we now introduce approximate eigenvalue problems with Dirichlet boundary condition and Neumann boundary condition, respectively. They will act as the state equation in the forthcoming optimization problems.
We either consider
\begin{alignat}{3}
    \left\{
        \begin{aligned}\label{statesD}
                -\nabla\cdot\left[a_{\varepsilon}(\varphi)\nabla w\right]
                +b_{\varepsilon}(\varphi)w
           &=
                \lambda^{\eps,\varphi}c_{\varepsilon}(\varphi)w
           &&
                \quad\text{in }\Omega,\\      
                w&=0&&\quad\text{on }\partial\Omega, 
        \end{aligned}             
    \right.
\end{alignat}
or
\begin{alignat}{3}
    \left\{
        \begin{aligned}\label{statesN}
                -\nabla\cdot\left[a_{\varepsilon}(\varphi)\nabla w\right]
           &=
                \mu^{\eps,\varphi}c_{\varepsilon}(\varphi)w
           &&
                \quad\text{in }\Omega,\\      
                \frac{\partial w}{\partial\B{\nu}}&=0&&\quad\text{on }\partial\Omega.
        \end{aligned}             
    \right.
\end{alignat}
Here $\B{\nu}$ is the outer unit normal vector on $\partial\Omega$, and $a_{\varepsilon},b_{\varepsilon},c_{\varepsilon}:\R\to \R$ are coefficient functions which depend on the phase-field $\varphi$ and the interface parameter $\eps>0$. 

For fixed $\eps>0$, we demand that $a_{\varepsilon}, c_{\varepsilon}>C_{\varepsilon}>0$ and $\ie{b}\ge 0$ in $\mathbb{R}$ in order to avoid degeneration. We further assume $a_{\varepsilon}, b_{\varepsilon}, c_{\varepsilon}\in C^{1,1}_{loc}(\mathbb{R})$.
These properties allow us to define the following scalar products on $L^2(\Omega)$ depending on the phase-field $\varphi\in L^{\infty}(\Omega)$:
 \begin{align*}
        \ap{u}{\eta}{\varphi}
    \coloneqq 
        \int_{\Omega}a_{\varepsilon}(\varphi)u\eta\text{\,d}x,\quad
        \cp{u}{\eta}{\varphi}
    \coloneqq 
        \int_{\Omega}c_{\varepsilon}(\varphi)u\eta\text{\,d}x,
       \quad u,\eta \in L^2(\Omega).
 \end{align*}
The induced norms on $L^2(\Omega)$ are
\begin{align}
	\label{DEF:SCP}
        \norm{u}_{a_{\varepsilon}(\varphi)}
    =
        \ap{u}{u}{\varphi}^{\frac{1}{2}},
        \quad
        \norm{u}_{c_{\varepsilon}(\varphi)}
    =
        \cp{u}{u}{\varphi}^{\frac{1}{2}}.
\end{align}
In the following, we use the notation $\Lzc$ to indicate that $L^2(\Omega)$ is equipped with the $\varphi$-dependent scalar product $\cp{\cdot}{\cdot}{\varphi}$.
Similarly, we equip the spaces $\Hz$ and 
\begin{align*}
    \HOO=
        \left\{
            w\in\HO\left|
            \int_{\Omega}
                c_{\varepsilon}(\varphi)w
            \text{\,d}x=0\right.
        \right\}
\end{align*}
with the scalar product $\left(\nabla\cdot,\nabla\cdot\right)_{\ie{a}(\varphi)}$. 
For the purpose of a clearer presentation we further define a positive semi-definite bilinear form $\bp{\cdot}{\cdot}{\varphi}$ in the same fashion as for the coefficient functions $a_{\varepsilon},c_{\varepsilon}$. However, this bilinear form does not define a scalar product as it possibly degenerates.

In the subsequent analysis, we will work with the weak formulations of the approximate problems \eqref{statesD} and \eqref{statesN} which are given as
\begin{align}\label{statewD}
    \ap{\nabla w}{\nabla \eta}{\varphi}
    +\bp{w}{\eta}{\varphi}
    =\lambda^{\eps,\varphi}\cp{w}{\eta}{\varphi}
    \quad\text{for all }\eta\in \Hz,
\end{align}
and
\begin{align}\label{statewN}
    \ap{\nabla w}{\nabla \eta}{\varphi}
    =\mu^{\eps,\varphi}\cp{w}{\eta}{\varphi}
    \quad\text{for all }\eta\in \HO,
\end{align}
respectively.
In Theorem~\ref{EEW} we will see that for any $\varphi\in \Lu$, all eigenvalues in either the Dirichlet or the Neumann case can be written as a sequence
\begin{align*}
    0<\lambda_1^{\eps,\varphi}
    \le\lambda_2^{\eps,\varphi}
    \le\lambda_3^{\eps,\varphi}
    \le\dots\to \infty,
\end{align*}
or
\begin{align*}
    0=\mu_0^{\eps,\varphi}
    <\mu_1^{\eps,\varphi}
    \le\mu_2^{\eps,\varphi}
    \le\mu_3^{\eps,\varphi}
    \le\dots\to \infty,
\end{align*}
respectively.

\subsection{The sharp interface limit}\label{Sub:SIAInt}
Before we formulate the optimization problems in which \eqref{statewD} and \eqref{statewN} serve as the state equations, we formally discuss their behavior when taking the limit $\eps\to 0$.

In both cases \eqref{statesD} and \eqref{statesN} we want to ensure that the boundary condition is not only fulfilled on the fixed boundary $\partial\Omega$
but also on the free boundary obtained in the sharp interface limit $\varepsilon\to 0$. By our diffuse interface approach we want to approximate this behavior.
 
Although the analytical results for $\varepsilon>0$ are independent of the following considerations as they can be carried out under the general assumptions on the coefficient functions made in Section~\ref{Sub:coeff}, we want to formally discuss how the coefficient functions need to be chosen explicitly to obtain the desired properties in the sharp interface limit.
 
For phase-field functions $\ie{\varphi}$ 
that are expected to converge to $\varphi_0$ in the sharp interface limit $\eps\to 0$  (with $\varphi_0$ attaining only the values $-1$ and $+1$), we define 
\begin{align*}
    \ee{\Omega}_{+}&
\coloneqq
    \left\{\left.\bx\in \Omega\right| \ie{\varphi}(\bx)\ge 0\right\},\\
    \ee{\Omega}_{-}&
\coloneqq
    \left\{
        \left.\bx\in \Omega\right| \ie{\varphi}(\bx)<0
    \right\},
    \\
    \Omega_{\pm}&
\coloneqq
    \left\{
        \left.\bx\in \Omega\right| \varphi_0(\bx)=\pm 1
    \right\}, 
    \\
    \Gamma&
\coloneqq
    \partial\Omega_{+}\cap\Omega,        
\end{align*}
where $\B{n}_{\varepsilon}$ is the outer unit normal vector field on $\partial\Omega_+^{\eps}\cap \Omega$ and $\B{n}$ is the outer unit normal vector field on $\Gamma$.
An illustration of the diffuse interface and the sharp interface limit can be found in Figure \ref{FIG:Setting}.

\begin{figure}
	\hspace{0.1\linewidth}
    \begin{minipage}[c]{\linewidth}
        \parbox{0.65\linewidth}{
            \begin{center}
                \begin{tikzpicture}[
                    scale=1.2,
                    background rectangle/.style={draw=black,line width=.3ex},
                    show background rectangle]
                    \draw[
                    	fill=gray!70,
                    	draw=white, 
                    	line width=1ex]
                    		(0.5,0.5) to 
                    		[closed, curve through = 
                    		{(1,0.5) (1.5,0.5) (1.7,1) (2,1) (3,1.5) (4,2) (4.1,2) (5,2.5)}]
                    	(0,3); 
                    \draw[
                        gray!30,  
                        line width=2ex]
                            (0.5,0.5) to 
                            [closed, curve through = 
                                {(1,0.5) (1.5,0.5) (1.7,1) (2,1) (3,1.5) (4,2) (4.1,2) (5,2.5)}]
                            (0,3);  
                    \draw[
                        black, 
                        line width=.1ex]
                            (0.5,0.5) to 
                            [closed, curve through = 
                                {(1,0.5) (1.5,0.5) (1.7,1) (2,1) (3,1.5) (4,2) (4.1,2) (5,2.5)}]
                            (0,3); 
                    \node (M) at (5,5.5) {\large $\Omega_{-}$};                                    
                    \node (P) at (1,2.5) {\large $\Omega_{+}$};
                    \node (G) at (2.3,1.3) {\large $\Gamma$};
                    \node (E) at (3.5,3.3) {%
                    \small
                        $\begin{aligned}
                            \begin{drcases}
                                -\Delta w = \lambda w\\
                                -\Delta w = \mu w
                            \end{drcases}
                        	\text{in $\Omega_+$}
                        \end{aligned}
                        $
                    }; 
                    \node (E) at (4.5,1.3) {%
                	\small
                	$\begin{aligned}
                		\begin{drcases}
                			w = 0\\
                			\partial_{\B{n}} w = 0
                		\end{drcases}
                		\text{on $\Gamma$}
                	\end{aligned}
                	$
                };   
                \end{tikzpicture}
            \end{center}
        }
        \parbox{0.1\linewidth}{ 
            \begin{align*}
               	\small
                    \begin{drcases}
                         w = 0\\
                         \partial_{\B{\nu}} w = 0
                    \end{drcases}
                \text{on $\partial\Omega$}
            \end{align*}
         }
    \end{minipage}
    \caption{The classical eigenvalue problems on $D=\Omega_+$ approximated by the diffuse interface approach. The diffuse interface is represented by the light gray surrounding of $\Gamma$. }\label{FIG:Setting}
\end{figure}
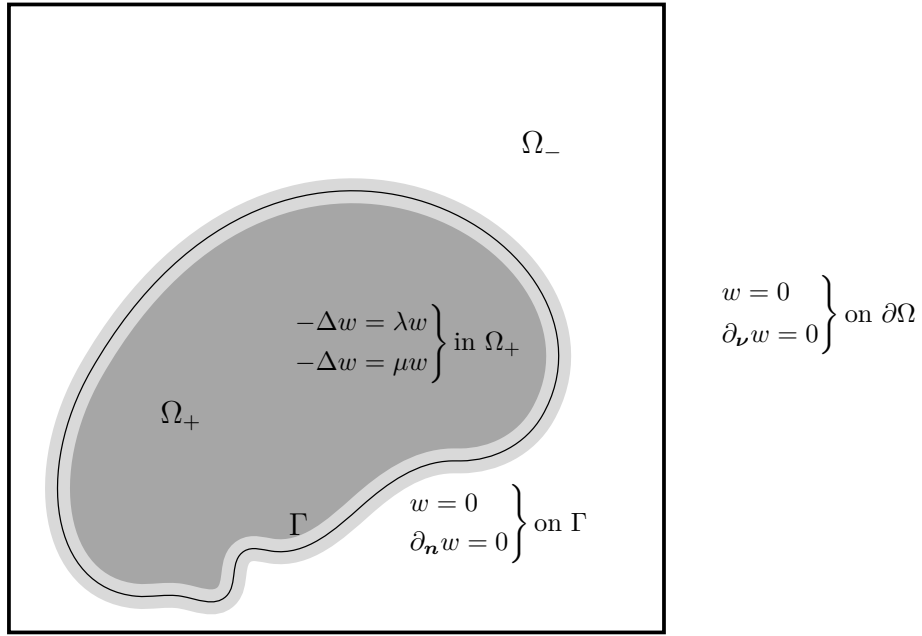

Now, the coefficient functions are to be chosen in such a way that they enforce the boundary condition%
\begin{align}\label{Dir}
w=0 \quad\text{on }\Gamma,
\end{align}
in the Dirichlet case and the boundary condition
\begin{align}\label{Neu}
\frac{\partial w}{\partial \B{n}}=0 \quad\text{on }\Gamma,
\end{align}
in the Neumann case.
We now present suitable choices for the coefficients $a_\eps$, $b_\eps$ and $c_\eps$ and we formally discuss how
the boundary conditions \eqref{Dir} and \eqref{Neu} are obtained in the sharp interface limit. 

In the Neumann case, we choose
 \begin{align}\label{ac}
     a_{\varepsilon}(1)=c_{\varepsilon}(1)=1,\qquad
     a_{\varepsilon}(-1)=a\varepsilon,\qquad
     c_{\varepsilon}(-1)=c\varepsilon,
 \end{align} 
with constants $a,c>0$. 
Then, condition \eqref{Neu} will be implicitly enforced in the following sense. 
The weak formulation of \eqref{statesN} is given by
 \begin{align}\label{WeakNeum}
        \int_{\Omega}
            a_{\varepsilon}(\ie{\varphi})\nabla w^{\eps,\ie{\varphi}}\cdot \nabla \eta
        \text{\,d}x
    =
        \mu^{\eps,\ie{\varphi}}
        \int_{\Omega}
            c_{\varepsilon}(\ie{\varphi})w^{\eps,\ie{\varphi}}\eta
        \text{\,d}x\quad\text{for all}\; \eta\in \HO.
 \end{align}
Assuming that the convergence $\ie{\varphi}\to\varphi_0$ implies the convergence of all appearing $\eps$-dependent quantities, we use \eqref{ac} to recover
\begin{align*}
          \int_{\Omega_+}
            \nabla w^{\varphi_0}\cdot \nabla \eta
        \text{\,d}x
    =
        \mu^{\varphi_0}
        \int_{\Omega_+}
            w^{\varphi_0}\eta
        \text{\,d}x\quad\text{for all}\; \eta\in \HO,  
\end{align*}
that is the weak formulation of the classical eigenvalue problem with Neumann boundary data
\begin{alignat*}{3}
        -\Delta w^{\varphi_0}
    &=
        \mu^{\varphi_0}w^{\varphi_0}
    &&\quad\text{in } \phantom{\partial}D &&=\Omega_{+},\\
        \frac{\partial w^{\varphi_0}}{\partial \B{n}}&=0
    &&\quad \text{on }\partial D &&=\Gamma,  
\end{alignat*}
by formally sending $\varepsilon\to 0$.

For the Dirichlet case let us fix $a_{\varepsilon}=c_{\varepsilon}\equiv 1$.
Then condition \eqref{Dir} will be ensured by the coefficient function $b_{\varepsilon}$ appearing in the state equation \eqref{statesD} by prescribing the following properties
  \begin{align}\label{bZI}
      b_{\varepsilon}(1)=0,\qquad
      \underset{\varepsilon\searrow 0}{\lim}\, b_{\varepsilon}(-1)=\infty.
  \end{align}
The idea of adding such a coefficient function comes from the porous medium approach that is used to model fluid dynamics phenomena, see e.g., \cite{GarHe,GarHeHin}.

Assuming again that the convergence $\ie{\varphi}\to\varphi_0$ implies the convergence of all appearing $\eps$-dependent quantities, we infer that
 \begin{align*}
    \int_{\ee{\Omega}_{-}}b_{\varepsilon}(\ie{\varphi})\abs{w^{\eps,\ie{\varphi}}}^2\text{\,d}x
 \end{align*}
stays bounded for $\varepsilon\to 0$ which can only be the case if $w^{\varphi_0}=0$ almost everywhere on the set $\Omega\backslash \Omega_{+}=\left\{\varphi_0=-1\right\}$. 
Proceeding as in the Neumann case and formally passing to the limit $\eps\to 0$ in the weak formulation, we conclude that $w^{\varphi_0}$ is a solution to the classical eigenvalue problem on $\Omega_{+}$ with Dirichlet boundary data, that is
\begin{align*}
\left\{
\begin{aligned}
        -\Delta w^{\varphi_0}
    &=
        \lambda^{\varphi_0}w^{\varphi_0}
    &&\quad\text{in }\phantom{\partial}D&&=\Omega_{+},\\
        w^{\varphi_0}&=0
    &&\quad \text{on }\partial D&&=\Gamma.     
\end{aligned}
\right.
\end{align*}
For a detailed rigorous analysis of the sharp interface limit in the Dirichlet case we refer to Section~\ref{SEC:SIDir}.
However, a rigorous analysis of the Neumann problem in our framework is not possible as the coefficient function $a_{\varepsilon}$ chosen in \eqref{ac} degenerates outside the prescribed shape. 
More explicitly, testing the weak formulation of the Neumann problem \eqref{WeakNeum} with the eigenfunction $w^{\eps,\varphi_{\eps}}$ yields
\begin{align*}
	\int_{\Omega}a_{\eps}(\varphi_{\eps})\abs{\nabla w^{\eps,\varphi_{\eps}}}^2\text{\,d}x=\mu^{\eps,\varphi_{\eps}},
\end{align*}
as we can assume the eigenfunction to be normalized with respect to $\norm{\,\cdot\,}_{c_{\eps}(\varphi)}$. To apply classical compactness results, we need to control the Dirichlet energy of the eigenfunctions, but $a_{\eps}$ as chosen in \eqref{ac} degenerates in the phase $\left\{\varphi_{\eps}=-1\right\}$ as $\eps\to 0$, i.e., in the left-hand side we obtain the term
\begin{align*}
	a\eps \int_{\left\{\varphi_{\eps}=-1\right\}}\abs{\nabla w^{\eps,\varphi_{\eps}}}^2\text{\,d}x.
\end{align*}
In other words, knowing that the sequence of eigenvalues $\left(\mu^{\eps,\varphi_{\eps}}\right)_{\eps>0}$ is bounded, does not imply that, on the whole of $\Omega$, also the Dirichlet energy is bounded.	

We are now in a position to introduce the optimization problems for $\eps>0$ in which \eqref{statesD} and \eqref{statesN}, respectively, can be regarded as the state equation.

\subsection{The optimization problems}\label{SOPT}
For any fixed $l\in\mathbb{N}$ and indices $i_1,\dots,i_l\in\mathbb{N}$ with $1\le i_1<i_2<\dots<i_l\,$, we include the eigenvalues $\lambda_{i_1},\dots,\lambda_{i_l}$ or $\mu_{i_1},\dots,\mu_{i_l}$, respectively, in the cost functional via the function
\begin{align*}
    \Psi: (\mathbb{R}_{>0})^{l}\to \mathbb{R},
\end{align*}
which is assumed to be of class $C^1$.
As mentioned above, the Ginzburg--Landau energy also needs to be included in the cost functional. Hence for $\eps>0$, we define the objective functional as
\begin{align}
    \label{OBJD}
            J_{l}^{D,\eps}(\varphi)
        \coloneqq
            \Psi(\lambda_{i_1}^{\eps,\varphi},\dots, \lambda_{i_l}^{\eps,\varphi})
        +
            \gamma E_{{\mathrm{GL}}}^{\eps}(\varphi),
\end{align}
in the Dirichlet case, and
\begin{align}\label{OBJN}
            J_{l}^{N,\eps}(\varphi)
        \coloneqq
            \Psi(\mu_{i_1}^{\eps,\varphi},\dots, \mu_{i_l}^{\eps,\varphi})
        +
            \gamma E_{{\mathrm{GL}}}^{\eps}(\varphi),
\end{align}
in the Neumann case, where $\gamma>0$ is a weighting parameter. 
Recalling \eqref{UCons} and \eqref{GDef}, we further define the set of admissible phase-fields as
\begin{align*}
    \Phi_{\text{ad}}\coloneqq\mathcal{G}^{\beta}\cap \mathcal{U}.
\end{align*}
Now, the optimization problem reads as
\begin{align}\label{PD} \tag{$\mathcal{P}^{D,\eps}_{l}$}
    \left\{
        \begin{aligned}
            &\text{ min} 
            &&J^{D,\eps}_l(\varphi),\\
            &\text{ s.t.} 
            &&\varphi\in \Phi_{\text{ad}},\\
            &
            &&\lambda^{\eps,\varphi}_{i_1},\dots, \lambda^{\eps,\varphi}_{i_l} 
            \;\text{are eigenvalues of } \eqref{statewD}
        \end{aligned}
    \right.
\end{align}
in the Dirichlet case, and
\begin{align}\label{PN} \tag{$\mathcal{P}^{N,\eps}_{l}$}
    \left\{
        \begin{aligned}
            &\text{ min} 
            &&J^{N,\eps}_l(\varphi),\\
            &\text{ s.t.} 
            &&\varphi\in \Phi_{\text{ad}},\\
            &
            &&\mu^{\eps,\varphi}_{i_1},\dots, \mu^{\eps,\varphi}_{i_l} 
            \;\text{are eigenvalues of } \eqref{statewN}
        \end{aligned}
    \right.
\end{align}
in the Neumann case.

Note that we do not need an additional assumption on the function $\Psi$ itself to be bounded from below in order for the minimization problem to possess a minimizer, as we can show that any eigenvalue of our approximate problem is bounded by the corresponding eigenvalue of the classical eigenvalue problem up to multiplicative constants, see Lemma~\ref{Lem:Bound}.
This allows us to cover a large variety of optimization problems.
For example, the problems \eqref{DirOpt} and \eqref{NeuOpt} could be formulated within our framework by choosing
\begin{align*}
    \Psi(\lambda_1^{\eps,\varphi})=\lambda_1^{\eps,\varphi},\quad
    \Psi(\mu_1^{\eps,\varphi})=-\mu_1^{\eps,\varphi},
\end{align*}
in \eqref{PD} and \eqref{PN},
respectively. 
 In Section~\ref{sec:num}, we will further demonstrate that the optimization of linear combinations of eigenvalues
\begin{align*}
        \Psi(\lambda_{i_1}^{\varepsilon,\varphi},\dots,\lambda_{i_l}^{\varepsilon,\varphi})
    =
        \sum_{j=1}^{l}\alpha_j\lambda_{i_j},\quad
        \Psi(\mu_{i_1}^{\varepsilon,\varphi},\dots,\mu_{i_l}^{\varepsilon,\varphi})
    =
        \sum_{j=1}^{l}\alpha_j\mu_{i_j},
\end{align*}
(with coefficients $\alpha_j\in\R$) can also be handled at least numerically.

Recall that $E^\eps(\varphi) = E^\eps_\text{GL}(\varphi)$ for all $\varphi\in \mathcal{G}^{\beta} \cap \mathcal{U}$.
Therefore, it suffices to merely include the regular part $E^\eps$ of the Ginzburg--Landau energy $E^\eps_\text{GL}$ in the cost functional.
This is important for the analysis as it allows us to compute directional derivatives 
\begin{align*}
    \frac{\text{d}}{\text{d}t}E^{\varepsilon}(\varphi+t(\vartheta-\varphi))\big\vert_{t=0}=
    \int_{\tOmega}
        \varepsilon\nabla\varphi\cdot\nabla(\vartheta-\varphi)
    \text{\,d}x
    +
    \int_{\tOmega}
        \frac{1}{\varepsilon}\psi_0^{\prime}(\varphi)(\vartheta-\varphi)
    \text{\,d}x,
\end{align*}
in \textit{every} direction $\vartheta-\varphi$ with $\vartheta\in \mathcal{G}^{\beta}\cap \mathcal{U}\,$, which would not be possible for $E^\eps_\text{GL}$.

To investigate these optimal control problems, we first have to establish the existence of eigenvalues and associated eigenfunctions. Furthermore, we need to analyze the continuity and differentiability properties of these quantities with respect to the phase-field variable. 

\section{Analysis of the diffuse interface problem}\label{SEC:Analysis}
\subsection{The state equations and their properties}
In this section we fix $\varepsilon>0$ and therefore, as mentioned above, we will just write $\varphi$ instead of $\varphi_\eps$. For a cleaner presentation, we also omit the superscript $\varepsilon$ when denoting eigenvalues and eigenfunctions as the $\eps$-dependency is indicated by the coefficient functions.
\begin{Def}[Definition of eigenvalues and eigenfunctions]\label{DEF:EEW}
	Let $\varphi\in \Li$ be arbitrary. 
    \begin{enumerate}[label=\textnormal{(\alph*)},leftmargin=*]
    \item $\lambda^\varphi$ is called a \textit{Dirichlet eigenvalue} of the state equation \eqref{statewD} if there exists a nontrivial weak solution $w^{D,\varphi}$ to \eqref{statewD}, i.e., $0\neq w^{D,\varphi} \in \Hz$, and it holds that
	\begin{align}\label{WEstateD}
	       \ap{\nabla w^{D,\varphi}}{\nabla \eta}{\varphi}
           +\bp{w^{D,\varphi}}{\eta}{\varphi}
	   =
	       \lambda^{\varphi}
	       \cp{w^{D,\varphi}}{\eta}{\varphi}
	\quad \text{for all $\eta\in \Hz$}.
	\end{align}
	In this case, the function $w^{D,\varphi}$ is called an \textit{eigenfunction} to the eigenvalue $\lambda^\varphi$.
    \item $\mu^\varphi$ is called a \textit{Neumann eigenvalue} of the state equation \eqref{statewN} if there exists a nontrivial weak solution $w^{N,\varphi}$ to \eqref{statewN}, i.e., $0\neq w^{N,\varphi} \in \HO$, and it holds that
   	\begin{align}\label{WEstateN}
    	  \ap{\nabla w^{N,\varphi}}{\nabla \eta}{\varphi}
    	   =
    	     \mu^{\varphi}
    	     \cp{w^{N,\varphi}}{\eta}{\varphi}
    	\quad \text{for all $\eta\in \HO$}.
    	\end{align}
    	In this case, the function $w^{N,\varphi}$ is called an \textit{eigenfunction} to the eigenvalue $\mu^\varphi$.
    \end{enumerate}
\end{Def}
The properties and assumptions of the previous section allow us to prove two classical functional analytic results concerning the eigenvalues and eigenfunctions.
\begin{Thm}[Existence and properties of eigenvalues and eigenfunctions]\label{EEW}
	$\,$\newline Let $\varphi\in \Li$ be arbitrary.
	\begin{enumerate}[label=\textnormal{(\alph*)},leftmargin=*]
		\item
		There exist sequences
		\begin{align*}
    		  \left(
    		      w^{D,\varphi}_k,\lambda_k^{\varphi}
    		  \right)_{k\in \N}
    		\subset\Hz\times \R,\quad
                \left(
            		      w^{N,\varphi}_k,\mu_k^{\varphi}
                \right)_{k\in \N_0}
            		\subset\HO\times \R,
		\end{align*}
		possessing the following properties:
		\begin{itemize}
			\item For all $k\in\N$, $w^{D,\varphi}_k$ is an eigenfunction to the eigenvalue $\lambda^\varphi_k$ and $w^{N,\varphi}_k$ is an eigenfunction to the eigenvalue $\mu^\varphi_k$ in the sense of Definition~\ref{DEF:EEW}.
			\item The eigenvalues $\lambda_k^{\varphi},\mu_k^{\varphi}$ (which are repeated according to their multiplicity) can be ordered in the following way:
			\begin{align*}
    			0<\lambda_1^{\varphi}
    			&\le\lambda_2^{\varphi}
    			\le\lambda_3^{\varphi}
    			\le\cdots,\\
                0=\mu_0^{\varphi}<\mu_1^{\varphi}
                &\le\mu_2^{\varphi}
        	    \le\mu_3^{\varphi}
                \le\cdots.
			\end{align*}
			Moreover, it holds that $\lambda_k^{\varphi},\mu_k^{\varphi}\to \infty$ as $k\to\infty$, and there exist no further  eigenvalues of the state equation $\eqref{WEstateD}$ and $\eqref{WEstateN}$.
			\item Both the Dirichlet eigenfunctions $\{w^{D,\varphi}_1, w^{D,\varphi}_2,\dots\} \subset \Hz$ and the Neumann eigenfunctions $\{w^{N,\varphi}_0, w^{N,\varphi}_1,\dots\} \subset \HO$ form an orthonormal basis of the space $\Lzc$. Furthermore, the eigenfunctions $\{w^{N,\varphi}_1, w^{N,\varphi}_2,\dots\}$ belong to the space $\HOO$
        and form an $\Lzc$-orthonormal basis of the space
        \begin{align*}
            L^2_{(0),\varphi}(\Omega)=\left\{
                            w\in \Lzc\left|\int_{\Omega}c_{\varepsilon}(\varphi)w\textup{\,d}x=0
                            \right.
                        \right\}.
        \end{align*}
        In particular, this implies that any eigenfunction to a non-trivial eigenvalue belongs to the space $\HOO$.
        \end{itemize}
		\item
		For $k\in\N$, we have the Courant--Fischer characterizations
		\begin{align}\label{CFD}
                \lambda_k^{\varphi}
            =
                \underset{V\in\mathcal{S}_{k-1}}{\max}\min
                    \left\{
                        \left.
                            \frac{
                                \int_{\Omega}
                                    \ie{a}(\varphi)\abs{\nabla v}^2
                                \text{\,d}x
                              +\int_{\Omega}
                                    \ie{b}(\varphi)\abs{v}^2
                                 \textup{\,d}x
                             }
                             {
                                \int_{\Omega}\ie{c}(\varphi)\abs{v}^2\textup{\,d}x
                             }
                      \right|
                        \begin{aligned}
                        	&v\in V^{\perp,\Lzc}\cap \Hz, \\
                            &v\neq 0
                        \end{aligned}
                \right\},
        \end{align}
        and
		\begin{align}\label{CFN}
                \mu_k^{\varphi}
            =
                \underset{V\in\mathcal{S}_{k-1}}{\max}\min
                \left\{
                    \left.
                        \frac{
                            \int_{\Omega}
                                \ie{a}(\varphi)\abs{\nabla v}^2
                            \textup{\,d}x
                         }
                         {
                             \int_{\Omega}\ie{c}(\varphi)\abs{v}^2\textup{\,d}x
                         }
                    \right|
                    \begin{aligned}
                    	&v\in V^{\perp,\Lzc}\cap \HOO, \\
                        &v\neq 0
                    \end{aligned}
                \right\}.
        \end{align}        
 		Here, $\mathcal{S}_{k-1}$ denotes the collection of all $(k-1)$-dimensional subspaces of 
 		$\Lzc$.
 		
		The set $V^{\perp,\Lzc}$ denotes the orthogonal complement of $V\subset L^2(\Omega)$ with respect to the scalar product $\cp{\cdot}{\cdot}{\varphi}$ on $\Lzc$.
		
		Moreover, the maximum is attained at the subspace
		\begin{align*}
		      V
          =
              \langle                   
                w^{D,\varphi}_1,\dots,w^{D,\varphi}_{k-1}
              \rangle_{\textup{span}}\quad\text{and}\quad
             V
         =
            \langle 
                w^{N,\varphi}_1,\dots,w^{N,\varphi}_{k-1}
            \rangle_{\textup{span}},
		\end{align*}
		respectively.
        \item \label{l1s}
        We can choose the eigenfunction $w_1^{D,\varphi}$ such that it is positive almost everywhere in $\Omega$. Furthermore, every solution $w\in \Hz$ of
        \begin{align*}
        	       \ap{\nabla w}{\nabla\eta}{\varphi}
                   +\bp{w}{\eta}{\varphi}
        	   =
        	       \lambda_1^{\varphi}
        	       \cp{w}{\eta}{\varphi}
        	\quad \text{for all $\eta\in \Hz$},
        \end{align*}
        is a multiple of $w_1^{D,\varphi}$, i.e., there is a constant $\xi\in \mathbb{R}$ such that $w=\xi w_1^{D,\varphi}$ almost everywhere in $\Omega$. This means that the eigenspace to $\lambda_1^{\varphi}$ is simple.
	\end{enumerate}
\end{Thm}

\begin{proof}[Proof of Theorem~\ref{EEW}]
(a) The assertion is a direct consequence of the spectral theorem for compact self-adjoint operators (see e.g., \cite[Sect.~12.12]{Alt}).

(b) The claim is established in the same fashion as in \cite[Thm.~3.2]{Garcke}.

(c) The assertion follows directly from \cite[Thm.~8.38]{Gilbarg}.
\end{proof}

\begin{Rem}
    In the following, we impose weaker assumptions on our phase field $\varphi$ compared to \cite{Garcke}, namely we only consider $\varphi\in H^1(\tOmega)\cap L^{\infty}(\Omega)$ instead of $\varphi\in H^1(\Omega)\cap L^{\infty}(\Omega)$, where $\tOmega=\Omega\backslash(S_0\cap S_1)$ with $S_0$ and $S_1$ being the sets appearing in the pointwise constraint \eqref{UCons}. As explained in Section~\ref{GinLan}, we consider this reduction of $H^1$-regularity in order to avoid additional Dirichlet boundary conditions for the phase-field on the boundaries of $S_0$ and $S_1$ which would complicate the sharp interface limit $\varepsilon\to 0$ that is discussed in Section~\ref{SEC:SIDir}.
    
    Nevertheless, we can still formulate and prove all the continuity and differentiability results established in \cite{Garcke} as they only rely on the pointwise almost everywhere convergence of phase-field sequences $\varphi_k\to \varphi$ in $\Omega$ and the boundedness in $L^{\infty}(\Omega)$ in order to apply Lebesgue's dominated convergence theorem.
\end{Rem}

\medskip

Following this remark, we only display the most important results of \cite{Garcke} adapted to our setting, namely the continuity of eigenvalues and eigenfunctions as well as the Fréchet-differentiability of simple eigenvalues.
\begin{Thm}[Continuity properties for the eigenvalues and their eigenfunctions]\label{slw}
	Let $j\in \N$ be arbitrary and
	let $\left(\varphi_k\right)_{k\in\N}\subset \Li$ be a bounded sequence with
	\begin{align*}
	   \varphi_k\to \varphi\quad \text{a.e.~in $\Omega$ as $k\to\infty$}.
	\end{align*} 
	Moreover, let $(u_j^{D,\varphi_k})_{k\in\N}\subset \Hz$ and $(u_j^{N,\varphi_k})_{k\in\N}\subset \HOO$ be sequences of $L_{\varphi_k}^2(\Omega)$-normalized eigenfunctions to the eigenvalues $(\lambda_{j}^{\varphi_k})_{k\in\N}$ and $(\mu_{j}^{\varphi_k})_{k\in\N}$ respectively.
	
	Then it holds that
	\begin{align*}
	   \lambda^{\varphi_k}_j\to \lambda^{\varphi}_j 
	   \quad\text{and}\quad
       \mu^{\varphi_k}_j\to \mu^{\varphi}_j
       \quad \text{as $k\to\infty$}.
	\end{align*}
	
	Furthermore, there exist $\Lzc$-normalized eigenfunctions $\overline{u}_{j}^D\in \Hz,\,\overline{u}_{j}^N\in \HOO$ to the eigenvalue $\lambda_j^\varphi,\mu_j^{\varphi}$ respectively, such that for $\zeta\in \left\{D,N\right\}$,
	\begin{align*}
    	   u^{\zeta,\varphi_k}_{j}\rightharpoonup\overline{u}^{\zeta}_j
           \quad \text{in }\HO,
    	\quad\text{and}\quad
    	   u^{\zeta,\varphi_k}_j\to \overline{u}^{\zeta}_j
           \quad \text{in }\Lzc,
	\end{align*}
	as $k\to \infty$ along a non-relabeled subsequence.
\end{Thm}
\begin{proof}
    The assertion can be established inductively by proceeding as in \cite[Thm.~4.4]{Garcke}, using the Courant--Fischer characterization from Theorem~\ref{EEW}, and the Banach--Alaoglu theorem applied to the sequence of eigenfunctions.
\end{proof}

In the remaining analysis of this section, namely the Fréchet-differentiability of eigenvalues, the assumption that the considered eigenvalues are simple is crucial. For instance, it already becomes clear in finite dimension (see \cite[Sec.~2.5]{Henrot}) that multiple eigenvalues are in general  not differentiable as two different eigenvalues can cross in a non-differentiable way. To overcome this issue, one can switch to a weaker notion of differentiability such as directional differentiability or semi-differentiability, cf.~\cite{Henrot, Garcke}. Nevertheless, our numerical method needs the first order conditions to be formulated in the framework of classical Fréchet-differentiability.
However multiple eigenvalues can to some extent still be handled numerically by adaptively reformulating the cost functional as we will see in Section~\ref{sec:num}.

Let us first recall the \textit{sign condition} for simple eigenvalues that was introduced in \cite{Garcke}.
\begin{Lem} \label{LEM:SIGN}
	Let $i\in \N$ and $\varphi\in \Li$ be arbitrary. We suppose that the eigenvalue $\lambda_i^{\varphi}$ $(\text{or }\mu_i^{\varphi})$ is simple.
	Let $w_i^{D,\varphi}\in \Hz$ $\big(\text{or }w_i^{N,\varphi}\in \HOO\big)$ be an $\Lzc$-normalized eigenfunction to the eigenvalue $\lambda_i^\varphi$ $($or $\mu_i^{\varphi})$.
	
	Then, for all $0<\sigma<1$, there exists $\delta>0$ such that for all
	\begin{align*}
	   h\in \Li
       \quad\text{with}
       \quad \norm{h}_{\Li}<\delta,
	\end{align*}
	there exist a unique $L^2_{\varphi+h}(\Omega)$-normalized eigenfunction $w_i^{D,\varphi+h}\in \Hz$ $\big($or $w_i^{N,\varphi+h}\in \HOO\big)$ to the eigenvalue $\lambda_i^{\varphi+h}$ $($or $\mu_i^{\varphi+h})$ satisfying the condition 
	\begin{align}
	 \label{SC}
	   0<\sigma<\cp{w^{D,\varphi+h}_i}{w^{D,\varphi}_i}{\varphi}
	   \qquad\Bigg(\;\text{or}\;\;
	   0<\sigma<\cp{w^{N,\varphi+h}_i}{w^{N,\varphi}_i}{\varphi}\;
	   \Bigg).
	\end{align}
    
	In particular, the eigenvalue $\lambda_i^{\varphi+h}$ $($or $\mu_i^{\varphi+h})$ is simple.
\end{Lem}
In the following, for the derivatives of the coefficient functions, we will use also the notation
\begin{align*}
        \left(u,v\right)_{\ie{a}^{\prime}(\varphi)h}
    \coloneqq
        \int_{\Omega}\ie{a}^{\prime}(\varphi)h\, uv\text{\,d}x
\end{align*}
to provide a clearer presentation.

\begin{Thm}[Fr\'echet-differentiability of simple eigenvalues and their eigenfunctions]\label{difewv}
	Let $\varphi\in \Li$ be arbitrary and suppose that for $i\in\N$, the eigenvalue $\lambda_i^\varphi$ $(\text{or } \mu_i^\varphi)$ is simple. 
	We further fix a corresponding $\Lzc$-normalized eigenfunction $w_i^{D,\varphi}$ $(\text{or }w_i^{N,\varphi})$.
	
	Then there exist constants $\delta_i^\varphi$, $r_i^\varphi>0$ such that the operator
	\begin{align*}
	&\left\{
	\begin{aligned}
        	S^{D,\varphi}_i: 
        	B_{\delta_i^\varphi}(\varphi)\subset \Li
    	&\to 
        	B_{r_i^\varphi}
        	   \big(
        	       (w^{D,\varphi}_i,\lambda_i^{\varphi})
        	   \big)
        	\subset\Hz\times \R\\
        	\vartheta
    	&\mapsto 
    	   \big(w^{D,\vartheta}_i,\lambda^{\vartheta}_i\big)
    \end{aligned}
    \right.	
    \\[1ex]
    \Bigg(\text{or}\quad
    &\left\{\begin{aligned}
        	S^{N,\varphi}_i: 
        	B_{\delta_i^\varphi}(\varphi)\subset \Li
    	&\to 
        	B_{r_i^\varphi}
        	   \big(
        	       (w^{N,\varphi}_i,\mu_i^{\varphi})
        	   \big)
        	\subset\HOO\times \R,\\
        	\vartheta
    	&\mapsto 
    	   \big(w^{N,\vartheta}_i,\mu^{\vartheta}_i\big)
    \end{aligned}
    \right.
    \Bigg),
	\end{align*}
	is well-defined and continuously Fr\'echet differentiable.
    
    Here, for any $\vartheta\in B_{\delta_i^\varphi}(\varphi)$, $w_i^{D,\vartheta}$ $(\text{or }w_i^{N,\vartheta})$ denotes the unique $L^2_\vartheta(\Omega)$-normalized eigenfunction to the eigenvalue $\lambda_i^\vartheta$ $(\text{or }\mu_i^\vartheta)$ satisfying the sign condition \eqref{SC} written for $h=\vartheta-\varphi$.
	    
	Moreover, for any $h\in \Li$, the Fr\'echet derivative $(\lambda^{\vartheta}_i)^{\prime}h$ of the Dirichlet eigenvalue $\lambda^{\vartheta}_i$ at $\vartheta\in B_{\delta_i^\varphi}(\varphi)$ in the direction $h$ reads as
	\begin{align}\label{lhdefD}
        \begin{aligned}
        	   \big(\lambda^{\vartheta}_i\big)^{\prime}h
           	&= 	
        	   \left(
        	       S^{D,\varphi}_{i,2}(\vartheta)
        	    \right)^{\prime}h\\
        	&=
    	       \left(
                    \nabla w_i^{D,\vartheta},\nabla w_i^{D,\vartheta}
               \right)_{\ie{a}^{\prime}(\vartheta)h}
               +\left(
                    w_i^{D,\vartheta},w_i^{D,\vartheta}
               \right)_{\ie{b}^{\prime}(\vartheta)h}  
        	   -\lambda^{\vartheta}_i
                \left(
                    w_i^{D,\vartheta},w_i^{D,\vartheta}
                \right)_{\ie{c}^{\prime}(\vartheta)h}.
        \end{aligned}
	\end{align}
    On the other hand, for any $h\in \Li$, the Fr\'echet derivative $(\mu^{\vartheta}_i)^{\prime}h$ of the Neumann eigenvalue $\mu^{\vartheta}_i$ at $\vartheta\in B_{\delta_i^\varphi}(\varphi)$ in the direction $h$ reads as
    	\begin{align}\label{lhdefN}
        	   \big(\mu^{\vartheta}_i\big)^{\prime}h
        	= 	
        	   \left(
        	       S^{N,\varphi}_{i,2}(\vartheta)
        	    \right)^{\prime}h
        	=
        	       \left(
                       \nabla w_i^{N,\vartheta},\nabla w_i^{N,\vartheta}
                   \right)_{\ie{a}^{\prime}(\vartheta)h}
        	   -
        	       \mu^{\vartheta}_i
        	       \left(
                        w_i^{N,\vartheta},w_i^{N,\vartheta}
                    \right)_{\ie{c}^{\prime}(\vartheta)h}.
    	\end{align}
\end{Thm}
\subsection{First order optimality conditions}\label{SEC:OPT:1}
We now intend to apply the theory developed in Section~\ref{SEC:Analysis} to show that the optimization problems $\eqref{PD}$ and $\eqref{PN}$ (that were introduced in Section~\ref{SOPT}) possess a minimizer. However, in the Neumann case, we first need to establish an additional boundedness property.

Recall that one possible application of our model is to \textit{maximize} the first non-trivial Neumann eigenvalue. 
Since \eqref{PN} is formulated as a minimization problem, we thus allow for functions $\Psi$ that are not bounded from below. A possible choice to realize a maximization of the first Neumann eigenvalue would be $\Psi(\mu_1^{\varphi})=-\mu_1^{\varphi}$ (meaning that $\Psi(x)=-x$ for all $x\in\R$).
To apply the direct method in the calculus of variations, we need to show that the cost functional $J_1^{\eps}(\varphi)=\Psi(\mu_1^{\varphi})+\gamma E^{\eps}(\varphi)$ remains bounded from below on the admissible set $\Phi_{\text{ad}}$, even if $\Psi$ is not bounded from below.

Our goal is to show that any Dirichlet eigenvalue of \eqref{WEstateD} and any Neumann eigenvalue of \eqref{WEstateN}, is uniformly bounded by expressions involving the corresponding eigenvalue of the classical eigenvalue problem on the \textit{whole} design domain $\Omega$. This allows us to deduce that each Dirichlet and Neumann eigenvalue belongs to a compact subset of $\mathbb{R}_{>0}$. Hence, as the function $\Psi$ is assumed to be continuous on $(\mathbb{R}_{>0})^{l}$, it is bounded on such compact sets.

\begin{Lem}\label{Lem:Bound}
    Let $k\in\mathbb{N}$ and let $\lambda_k^{LD}$ denote the $k$-th eigenvalue of the classical Dirichlet eigenvalue problem, i.e.,
    \begin{align*}
        (\nabla w,\nabla \eta)=\lambda_k^{LD}(w,\eta)\quad \forall\eta \in \Hz,
    \end{align*}
    and let $\mu_k^{LD}$ denote the $k$-th eigenvalue of the classical Neumann eigenvalue problem, i.e.,
    \begin{align*}
        (\nabla w,\nabla \eta)=\mu_k^{LD}(w,\eta)\quad \forall\eta \in \HO,
    \end{align*} 
    where $(\cdot,\cdot)$ denotes the standard scalar product on $L^2(\Omega)$.
    Then there exist constants $C_{1,\eps},$ $C_{2,\eps}>0$ depending only on the choice of coefficient functions $a_{\eps}$, $b_{\eps}$, and $c_{\eps}$ such that
    \begin{alignat*}{2}
        C_{1,\eps}\,\lambda_k^{LD}
        &\le\lambda_k^{\varphi}
        &&\le C_{2,\eps}\,(\lambda_k^{LD}+1),\\
        C_{1,\eps}\,\mu_k^{LD}
        &\le\mu_k^{\varphi}
        &&\le C_{2,\eps}\,\mu_k^{LD},
    \end{alignat*}
    for all $\varphi\in \Phi_{\textup{ad}}$.
\end{Lem}
\begin{proof}
    Let us start with the Neumann case.
    We will work with a Courant--Fischer characterization which is equivalent to \eqref{CFN} namely
    \begin{align}\label{CFAlt}
        \mu_k^{\varphi}
            &=
                \underset{V\in\mathcal{\tilde{S}}_{k+1}}{\min}\max
                \left\{
                    \left.
                        \frac{
                            \int_{\Omega}
                               \ie{a}(\varphi)\abs{\nabla v}^2
                            \text{\,d}x
                         }
                         {
                             \int_{\Omega}\ie{c}(\varphi)\abs{v}^2\text{\,d}x
                         }
                    \right|
                    \begin{aligned}
                    	&v\in V, \\
                        &v\neq 0
                    \end{aligned}
                \right\},   
    \end{align}
    see \cite[Sec.~1.3.1]{Henrot}. Here, ${\tilde{S}}_{k+1}$ denotes the collection of all $(k+1)$-dimensional subspaces of $H^1(\Omega)$. Note that compared to \cite{Henrot} we have to consider dimension $(k+1)$ instead of $k$, as we start our labeling of Neumann eigenvalues with the index $0$ and not with $1$.
    Obviously, we obtain the characterization of the classical Neumann eigenvalue by setting $a_{\eps}\equiv c_{\eps}\equiv 1$, i.e.,
    \begin{align}\label{CFAltL}
        \mu_k^{LD}
            &=
                \underset{V\in\mathcal{\tilde{S}}_{k+1}}{\min}\max
                \left\{
                    \left.
                        \frac{
                            \int_{\Omega}
                               \abs{\nabla v}^2
                            \text{\,d}x
                         }
                         {
                             \int_{\Omega}\abs{v}^2\text{\,d}x
                         }
                    \right|
                    \begin{aligned}
                    	&v\in V, \\
                        &v\neq 0
                    \end{aligned}
                \right\}.    
    \end{align}
    We now recall the assumptions on the coefficient functions from Section~\ref{Sub:coeff}. In particular, we know that there is a constant $\tilde{C}_{\eps}>0$ such that $a_{\eps}(\varphi),c_{\eps}(\varphi)<\tilde{C}_{\eps}$ for all $\varphi\in \Phi_{\text{ad}}$, as $a_{\eps},c_{\eps}\in C^0(\mathbb{R})$ and $\abs{\varphi}\le 1$. Furthermore, we assumed $\ie{a},\ie{c}$ to be uniformly bounded from below, i.e., $a_{\eps},c_{\eps}\ge C_{\eps}$
    for some constant $C_{\eps}>0$.
    Thus, we deduce that there are constants $C_{1,\eps},C_{2,\eps}>0$ that only depend on the choice of the real functions $a_{\eps},c_{\eps}$, such that
    \begin{align*}
        C_{1,\eps}\,\frac{
            \int_{\Omega}
                \abs{\nabla v}^2
            \text{\,d}x
            }
            {
            \int_{\Omega}\abs{v}^2\text{\,d}x
            }    
        \le
        \frac{
            \int_{\Omega}
                a_{\eps}(\varphi)\abs{\nabla v}^2
            \text{\,d}x
            }
            {
            \int_{\Omega}c_{\eps}(\varphi)\abs{v}^2\text{\,d}x
            }
        \le    
        C_{2,\eps}\,\frac{
            \int_{\Omega}
                \abs{\nabla v}^2
            \text{\,d}x
            }
            {
            \int_{\Omega}\abs{v}^2\text{\,d}x
            },          
    \end{align*}
    for all $0\neq v\in H^1(\Omega)$ and $\varphi\in \Phi_{\text{ad}}.$
    Comparing \eqref{CFAlt} and \eqref{CFAltL}, this directly allows us to deduce the claim in the Neumann case.
    
    The Dirichlet case works completely analogously but one has to mind the additional term coming from the coefficient function $b_{\varepsilon}\in C^{1,1}_{loc}(\R)$ which is assumed to be non-negative. Thus we obtain here
    \begin{align*}
        C_{1,\eps}\,\frac{
            \int_{\Omega}
                \abs{\nabla v}^2
            \text{\,d}x
            }
            {
            \int_{\Omega}\abs{v}^2\text{\,d}x
            }    
        \le
        \frac{
            \int_{\Omega}
                a_{\eps}(\varphi)\abs{\nabla v}^2
                +b_{\eps}(\varphi)\abs{v}^2
            \text{\,d}x
            }
            {
            \int_{\Omega}c_{\eps}(\varphi)\abs{v}^2\text{\,d}x
            }
        \le    
        C_{2,\eps}\,\left(\frac{
            \int_{\Omega}
                \abs{\nabla v}^2
            \text{\,d}x
            }
            {
            \int_{\Omega}\abs{v}^2\text{\,d}x
            }+1\right),          
    \end{align*} 
    for all $0\neq v\in H^1_0(\Omega)$ and $\varphi\in \Phi_{\text{ad}}$.
    \end{proof}
\begin{Thm}[Existence of a minimizer to \eqref{PD} and \eqref{PN}] \label{Exm}
	 The problems $\eqref{PD}$ and \eqref{PN} possess a minimizer $\varphi^D\in \mathcal{G}^{\beta}\cap \mathcal{U}$ and $\varphi^N\in \mathcal{G}^{\beta}\cap \mathcal{U}$, respectively.
\end{Thm}
\begin{proof}
    We proceed by applying the direct method in the calculus of variations.
    First of all, the feasible set $\mathcal{G}^{\beta}\cap\mathcal{U}$ is non-empty since it contains the function which is identical to $\beta_1$ in $\tOmega$, $-1$ in $S_0$ and $1$ in $S_1$.
    By the previous discussion we already know that the objective functionals $J^{D,\varepsilon}_l$ and $J^{N,\varepsilon}_l$ are bounded from below.
    Since $\varphi\in \mathcal{G}^{\beta}\subset L^\infty(\Omega)$, the term
    \begin{align*}
        \gamma E^\eps(\varphi) =
         \gamma \int_{\tOmega}
            \left(
                \frac{\eps}{2}\abs{\nabla\varphi}^2+\frac{1}{\eps}\psi_0(\varphi)
            \right) \,\mathrm dx,
    \end{align*}
    in the cost functional can be used to control $\varphi$ in the $H^1(\tOmega)$-norm, but not in the whole $H^1(\Omega)$-norm. This means that for any minimizing sequence, we can only apply compactness on $\tOmega$ which implies strong convergence in $L^2(\tOmega)$. However, as the elements of this minimizing sequence are additionally contained in the feasible set $\mathcal{G}^{\beta}\cap \mathcal{U}$, their values on $S_0\cup S_1$ are a priori fixed, which gives us the desired pointwise almost everywhere convergence on the whole domain $\Omega$. This allows us to apply Theorem~\ref{slw} which provides the continuity of the eigenvalues with respect to the phase-field variable.
\end{proof}
Now, invoking the differentiability properties established in Section~\ref{SEC:Analysis}, we can derive a first-order necessary condition for local optimality. The following variational inequalities follow directly from the fact that for $\zeta\in \left\{D,N\right\}$ and $\vartheta\in \mathcal{G}^{\beta}\cap \mathcal{U}$
\begin{align*}
    	0\le 
            \frac{\mathrm d}{\mathrm d t} 
                \left[J^{\zeta,\varepsilon}_{l}\big(
                    \varphi^{\zeta}+t(\vartheta-\varphi^{\zeta}) 
                  \big)\right]_{\vert_{t=0}} 
          = 
                \left(J^{\zeta,\varepsilon}_l\right)'
                    \big(\varphi^{\zeta}\big)(\vartheta-\varphi^{\zeta}),
\end{align*}
as the convexity of $\mathcal{G}^{\beta}\cap \mathcal{U}$ implies that $\varphi^{\zeta}+t(\vartheta-\varphi^{\zeta})\in \mathcal{G}^{\beta}\cap \mathcal{U}$ for $t\in [0,1]$.

\begin{Thm}[The optimality system to \eqref{PD}] \label{VUD}
	Let $\varphi\in \mathcal{G}^{\beta}\cap \mathcal{U}$ be a local minimizer of the optimization problem $\eqref{PD}$, i.e., there exists $\delta>0$ such that 
	\begin{align*}
    	   J_l^{D,\varepsilon}(\vartheta)
       &\ge
            J_{l}^{D,\varepsilon}(\varphi)
       \quad\text{for all $\vartheta\in \mathcal{G}^{\beta}\cap \mathcal{U}$ with } \norm{\vartheta-\varphi}_{\Li}<\delta.
	\end{align*}
	Suppose that the eigenvalues $\lambda^{\varphi}_{i_1},\dots,\lambda^{\varphi}_{i_l}$ are simple and let us fix associated $\Lzc$-normalized eigenfunctions $w^{D,\varphi}_{i_1},\dots,w^{D,\varphi}_{i_l}\in \Hz$.
	
	Then the following optimality system is satisfied:
	\begin{itemize}
    		\item The state equations
        \begin{alignat}{3}\tag{$SD_j$}
            \left\{
                \begin{aligned}
                        -\nabla\cdot
                            \left[
                                a_{\varepsilon}(\varphi)\nabla w^{D,\varphi}_{i_j}
                            \right]
                        +b_{\varepsilon}(\varphi)w^{D,\varphi}_{i_j}
                   &=
                        \lambda_{i_j}^{\varphi}c_{\varepsilon}(\varphi)w^{D,\varphi}_{i_j}
                   &&
                        \quad\text{in }\Omega,\\      
                        w^{D,\varphi}_{i_j}&=0&&\quad\text{on }\partial\Omega
                \end{aligned}             
            \right.
        \end{alignat}
    		are fulfilled in the weak sense for all $j\in\{1,\dots,l\}$. 
    		\item 
    		The variational inequality
       		\begin{align}
       		\tag{${VD}$}
           		\begin{aligned}
             		\hspace{-3em}
             		0
               		&\le 
                   		\gamma\eps
                   		\int_{\tOmega}
                           \nabla\varphi\cdot\nabla(\vartheta-\varphi)\,
                         \textup{d}x
                   		+\frac{\gamma}{\eps}
                   		   \int_{\tOmega}
                               \psi_0^{\prime}(\varphi)(\vartheta-\varphi)
                            \textup{\,d}x \\[1ex]
               		&\; 
                   		+\sum_{j=1}^{l}
                   		\Bigg\{
                            [\partial_{\lambda_{i_j}}
                            \hspace{-0.7ex}\Psi]
                		          \big(
                                    \lambda^{\varphi}_{i_1},\dots,\lambda^{\varphi}_{i_l}
                                    \big)\;
                		              \Bigg(
                		                  \left(
                                            \nabla w^{D,\varphi}_{i_j},\nabla w^{D,\varphi}_{i_j}
                                           \right)_{\ie{a}^{\prime}(\varphi)(\vartheta-\varphi)}
                                           +\left(
                                                w^{D,\varphi}_{i_j},w^{D,\varphi}_{i_j}
                                           \right)_{\ie{b
                                           }^{\prime}(\varphi)(\vartheta-\varphi)}     \\[-1ex]
               		&\qquad\qquad\qquad\qquad\qquad\qquad\qquad
                		                  -\lambda^{\varphi}_{i_j}
                                            \left(
                                            w^{D,\varphi}_{i_j},w^{D,\varphi}_{i_j}
                                           \right)_{\ie{c}^{\prime}(\varphi)(\vartheta-\varphi)}
                		              \Bigg)
              		  \Bigg\}
           		\end{aligned}
       		\end{align}
    		is satisfied for all $\vartheta\in \mathcal{G}^{\beta}\cap \mathcal{U}$.
	\end{itemize}
\end{Thm}

\begin{Thm}[The optimality system to \eqref{PN}] \label{VUN}
	Let $\varphi\in \mathcal{G}^{\beta}\cap \mathcal{U}$ be a local minimizer of the optimization problem $\eqref{PN}$, i.e., there exists $\delta>0$ such that 
	\begin{align*}
           J_l^{N,\varepsilon}(\vartheta)
       &\ge 
           J_{l}^{N,\varepsilon}(\varphi)
       \quad 
           \text{for all $\vartheta\in \mathcal{G}^{\beta}\cap \mathcal{U}$ with } \norm{\vartheta-\varphi}_{\Li}<\delta.
	\end{align*}
	Suppose that the eigenvalues $\mu^{\varphi}_{i_1},\dots,\mu^{\varphi}_{i_l}$ are simple and let us fix associated $\Lzc$-normalized eigenfunctions  $w^{N,\varphi}_{i_1},\dots,w^{N,\varphi}_{i_l}\in \HOO$.
	
	Then the following optimality system is satisfied:
	\begin{itemize}
    		\item The state equations
        \begin{alignat}{3}\tag{$SN_j$}
            \left\{
            \begin{aligned}
                    -\nabla\cdot\left[a_{\varepsilon}(\varphi)\nabla w^{N,\varphi}_{i_j}\right]
               &=
                    \mu_{i_j}^{\varphi}c_{\varepsilon}(\varphi)w^{N,\varphi}_{i_j}
               &&
                    \quad\text{in }\Omega,\\      
                    \frac{\partial w^{N,\varphi}_{i_j}}{\partial \B{\nu}}&=0&&\quad\text{on }\partial\Omega
            \end{aligned}             
            \right.
        \end{alignat}
    		are fulfilled in the weak sense for all $j\in\{1,\dots,l\}$. 
    		\item 
    		The variational inequality
        		\begin{align}
        		\tag{${VN}$}
            		\begin{aligned}
                		  0
                		&\le 
                    		\gamma\eps
                    		\int_{\tOmega}
                                \nabla\varphi\cdot\nabla(\vartheta-\varphi)\,
                             \textup{d}x
                    		+\frac{\gamma}{\eps}
                    		\int_{\tOmega}
                                \psi_0^{\prime}(\varphi)(\vartheta-\varphi)
                             \textup{\,d}x \\[1ex]
                		&\qquad 
                    		+\sum_{j=1}^{l}
                    		\Bigg\{
                              [\partial_{\lambda_{i_j}}
                                \hspace{-0.7ex}\Psi]   \big(\mu^{\varphi}_{i_1},\dots,\mu^{\varphi}_{i_l}\big)\;
                    		              \Bigg(
                    		                  \left(
                                                w^{N,\varphi}_{i_j},w^{N,\varphi}_{i_j}
                                               \right)_{\ie{a}^{\prime}(\varphi)(\vartheta-\varphi)} \\[-1ex]
                		&\qquad\qquad\qquad\qquad\qquad\qquad\qquad
                    		                  -\mu^{\varphi}_{i_j}
                                          \left(
                                            w^{N,\varphi}_{i_j},w^{N,\varphi}_{i_j}
                                           \right)_{\ie{c}^{\prime}(\varphi)(\vartheta-\varphi)}
                		                  \Bigg)
                 		  \Bigg\}
            		\end{aligned}
        		\end{align}
    		is satisfied for all $\vartheta\in \mathcal{G}^{\beta}\cap \mathcal{U}$.
	\end{itemize}
\end{Thm}

\section{Sharp Interface Asymptotics for the Dirichlet Case}\label{SEC:SIDir}

In this section, we want to discuss the sharp interface asymptotics for the Dirichlet eigenvalue optimization problem \eqref{PD}, i.e., its behavior when $\varepsilon\to 0$.
For the sake of a rigorous discussion we need to make additional assumptions that are supposed to hold throughout the remainder of this section. 

    \begin{enumerate}[label=\textnormal{\bfseries{(A\arabic*)}}]
        \item \label{Asmpt:Dom} We assume that the design domain $\Omega$ is a bounded Lipschitz domain in $\R^d$ with $d\ge 2$.
        \item\label{Asmp:ac}
            We fix $a_{\varepsilon}=c_{\varepsilon}= 1$ on $[-1,1]$.
        \item\label{Asmp:b}
            Let 
            \begin{align}
            \label{BEPS}
                b_{\varepsilon}:[-1,1]\to [0,\overline{b}_{\varepsilon}], 
                \;\; \eps>0,
                \quad\text{and}\quad
                b_0:[-1,1]\to [0,+\infty]  ,
            \end{align}
            be functions with
            \begin{itemize}
                \item $b_{\eps}$ is decreasing, continuous and surjective
                \item $b_0$ is continuous at the point $1$   
                \item $b_0(0)<+\infty$,
                \item $b_{\varepsilon}\to b_0$ pointwise in $[-1,1]$ as $\eps\to 0$,
                \item and $b_{\delta} \ge b_{\varepsilon}$ on $[-1,1]$ for all $0\le\delta\le\varepsilon$.
            \end{itemize}
            Here, the interval $[0,+\infty]$ is to be understood as a subset of the extended real numbers $\overline\R = \R\cup\{\pm\infty\}$, on which we use the common conventions $\pm \infty \,\cdot\, 0 = 0$ and $0^{-1} = +\infty$.
            
            Moreover, the numbers $\overline{b}_\eps$ in \eqref{BEPS}
            are chosen such that
            \begin{align*}
                \underset{\varepsilon\to 0}{\lim}\;\overline{b}_{\varepsilon}=+\infty
                \quad\text{with}\quad
                \overline{b}_{\varepsilon}=o(\varepsilon^{-\kappa}) \quad\text{as $\eps\to 0$,}
            \end{align*}       
    where, depending on the dimension $d$,
    \begin{align*}
        \begin{cases}
            \kappa\in (0,1)
            &\text{if}\; d = 2, \\
            \kappa = \frac 2d 
            &\text{if}\; d\ge 3.
        \end{cases}
    \end{align*}
        \item \label{Asmp:NTV}
            In the following, we only consider elements $\varphi\in BV(\tOmega,\left\{\pm 1\right\})\cap \mathcal{U}$ such that the set
            \begin{align*}
                E^{\varphi}\coloneqq
                \left\{
                    x\in\Omega\left|\,
                    \varphi(x)=1\right.
                \right\}
            \end{align*}
            contains an open ball $B$. From this assumption we infer $C_0^{\infty}(B)\subset V^{\varphi}$ hence this excludes the pathological case that the space
            \begin{align*}
                V^{\varphi}\coloneqq
                \left\{
                    \eta\in\Hz\left|\,
                    \eta=0\text{ a.e.~in }\Omega\backslash E^{\varphi}\right.
                \right\}
            \end{align*} 
            is trivial or finite dimensional. 
            Recalling the definition of the set $\mathcal U$ in Section~\ref{SOPT}, this condition on $E^{\varphi}$ can be implemented in the constraint $\varphi\in \mathcal U$ by simply demanding $B\subset S_1$ for any prescribed open ball $B\subset \Omega$. 
             Later, in Subsection~\ref{Sec:LimProp}, we will discuss how $V^{\varphi}$ is related to ``Sobolev-like'' spaces in the context of quasi-open sets. 
            We further define the space
            \begin{align*}
                H^{\varphi}\coloneqq
                \left\{
                    \eta\in L^2(\Omega)\left|\,
                    \eta=0\text{ a.e.~in }\Omega\backslash E^{\varphi}\right.
                \right\}.
            \end{align*}
        \item In addition to the assumptions of Section~\ref{Sec:Form}, we demand that $S_0$ and $S_1$ are sets of finite perimeter in $\Omega$. Then \cite[Thm.~3.87]{AmbrosioFuscoPallara} yields that any $\varphi\in BV(\tOmega,\left\{\pm 1\right\})\cap \mathcal{U}$ is indeed an element of $BV(\Omega,\left\{\pm 1\right\}).$ 
        Hence, in particular, $E^{\varphi}$ is a set of finite perimeter in $\Omega$.
        Therefore, we consider $\varphi\in BV(\Omega,\left\{\pm 1\right\})\cap \mathcal{U}$ in the following.
        \item \label{Asmpt:Pot} For the potential $\psi$ appearing in the Ginzburg--Landau energy, we choose the double-obstacle potential whose regular part is given as $\psi_0(\varphi)=\frac{1}{2}(1-\varphi^2)$. This type of free energy is for example studied in \cite{BloweyElliott}.
    \end{enumerate}

\begin{Rem}~\label{RemLimit}
    \begin{enumerate}[label=\textnormal{(\alph*)},leftmargin=*]
    \item The case of dimension $d=1$ needs to be excluded as here some of the fundamental theorems about quasi-open sets and capacity theory are not true, see, e.g., \cite[Chap.~4]{Bucur}.
    \item 
        The growth condition $\ie{\overline{b}}=o(\varepsilon^{-\kappa})$ is chosen in order to obtain the desired convergence of the term involving $b_{\varepsilon}$ in Lemma~\ref{LEM:lam1Conv} and Theorem~\ref{Thm:Sconv} via a Hölder estimate in dependence of the dimension $d\ge2$, see \cite[Proof of Lem.~3, 2nd step]{GarHe} for the case $d=3$.
        As explained in \cite[Rem. 2]{GarHe}, for $d=2$, this growth condition can be weakened to $\ie{\overline{b}}=o(\varepsilon^{-\kappa})$ for any $\kappa\in (0,1)$.
        \item \label{itemCr}
        According to \cite[Lem.~3.2]{Hecht}, we find a crack free representative $E^{\varphi}_c$ of $E^{\varphi}$ that is a set of finite perimeter with $\textup{int}(E^{\varphi}_c)=\textup{int}(\overline{E^{\varphi}_c})$ (where int denotes the interior) and 
        \begin{align*}
            \abs{E^{\varphi}\bigtriangleup E_c^{\varphi}}=0, 
            \quad\text{where}\quad
            E^{\varphi}\bigtriangleup E_c^{\varphi} = (E^{\varphi} \cup E_c^{\varphi}) \backslash (E^{\varphi} \cap E_c^{\varphi}).
        \end{align*}
        We further point out that assumption \ref{Asmp:NTV} guarantees that $\textup{int}(E^{\varphi}_c)\neq \emptyset$ and we can thus apply \cite[Thm.~6.3]{DelZol} from which we infer
        \begin{align*}
                V^{\varphi}
            \subset
                    \left\{v\in H_0^1(\Omega)\left|\,
                    v=0\text{ a.e.~in }\Omega\backslash E^{\varphi}_c\right.
                    \right\}
            = 
                H_0^1(\textup{int}(E^{\varphi}_c)).
        \end{align*}
        This means any function in the abstract space $V^{\varphi}$ can be seen as an element of the restricted Sobolev space $H_0^1(\textup{int}(E^{\varphi}_c))$. This will help us to understand the limit problem in the remainder of this section.
        \item Instead of employing the potential as declared in \ref{Asmpt:Pot}, it would also be possible to use different choices.
        For instance, the quartic regular part $\psi_0(\varphi)=\frac{1}{2}(1-\varphi^2)^2$ could also be chosen. 
        This choice would only affect the choice of profiles used to construct a recovery sequence in \textit{Step~1} of the proof of Theorem~\ref{GamHelp} but our theory would still remain valid.
    \end{enumerate}
\end{Rem}

Under the above assumptions, we can prove that eigenfunctions of \eqref{WEstateD} for $\varepsilon>0$ converge to eigenfunctions of a limit problem as $\varepsilon\searrow 0$, and these limit functions have suitable properties.
To this end, we first want to develop a better understanding of the limit problem.

\subsection{The limit problem and its properties}\label{Sec:LimProp}

In the following, we discuss the limit eigenvalue problem and its most important properties. It is well known that due to well-posedness of the minimization problem on the sharp interface level, we need to consider the Dirichlet eigenvalue problem in its relaxed form using Borel measures as introduced, e.g., in \cite{MasoMurat,AttouchButtazzo, ButtazzoMaso91, Bucur, HenrotPierre}. In the spirit of \cite{BogoselOudet,DePhilippis,BucurHenrot} we only recall the key facts of this theory needed in our framework, the details can be found there and the references therein.

The following theory strongly relies on so called \textit{quasi-open} sets and the notion of capacity as it is closely related to the relaxed Dirichlet problem.
The \textit{capacity} of a measurable set $E\subset \R^d$ is defined as
\begin{align*}
	\textup{cap}(E)=\inf\left\{\left.\int_{\R^d}\abs{\nabla u}^2+u^2\text{\,d}x\right\vert u\in \mathcal{U}_{E}\right\},
\end{align*}	
where $\mathcal{U}_{E}$ is the set of functions $u\in H^1(\R^d)$ such that $u\ge1$ a.e.~in a neighborhood of  $E$. Here, the expression ``a.e.~in a neighborhood'' means that there exists an open set $\mathcal{O}$ containing $E$ such that $u\ge 1$ a.e.~in $\mathcal{O}$. We say that a relation holds \textit{quasi-everywhere} (short q.e.) if it holds up to a set of zero capacity.

A set $\omega \subset \R^d$ is called \textit{quasi-open} if for every $\delta>0$ there is an open set $\omega_{\delta}\subset \R^d$ such that $\textup{cap}(\omega \bigtriangleup \omega_{\delta})<\delta$. Here $\bigtriangleup$ denotes the symmetric difference of sets.\\
Above, we have introduced the space $V^{\varphi}$ based on a function $\varphi\in BV(\Omega,\left\{\pm 1\right\})$. More general, for any measurable set $E\subset \Omega$, we define the ``Sobolev-like'' space 
\begin{align*}
	\tilde{H}_0^1(E)=\left\{\left. u\in H^1_0(\Omega) \,\right\vert\, u=0 \text{ a.e.~in }\Omega\backslash E\right\}.
\end{align*}
It is clear that $\tilde{H}_0^1(E)\subset H_0^1(\Omega)$ is a closed subspace. In general, for open sets $E\subset \Omega$, the inclusion $H_0^1(E)\subset \tilde{H}_0^1(E)$ might be strict. However, if the open set $E$ additionally has a Lipschitz boundary, then equality holds.
A crucial property of the space $\tilde{H}_0^1(E)$ is that there exists a unique quasi-open set $\omega\subset \Omega$ such that
\begin{align}\label{MeasQuasi}
	\tilde{H}_0^1(E)=H_0^1(\omega) = \left\{\left. u\in H_0^1(\Omega) \,\right\vert\, \tilde{u}=0 \text{ q.e.~in } \Omega\backslash \omega\right\}.
\end{align}
This can be verified as in the proof of \cite[Prop.~3.3.44]{HenrotPierre}. Note that uniqueness is to be understood up to a set of zero capacity. Here, $\tilde{u}\in H_0^1(\Omega)$ denotes the quasi-everywhere unique quasi-continuous representative of $u\in H_0^1(\Omega)$ and in the following, we identify each function in $H_0^1(\Omega)$ with its quasi-continuous representative. We refer to \cite[Sec.~3.3.4]{HenrotPierre} for further details.
It is important to notice that, in general, it merely holds $\omega\subset E$. Indeed, the case $\abs{E\backslash \omega}>0$ may occur, see \cite{BucurHenrot}.

The relation \eqref{MeasQuasi} can now further be refined via the following Laplace equation.  Let $u_{E}\in \tilde{H}_0^1(E)$ be the unique solution of
\begin{align}\label{wE}
	\int_{\Omega}\nabla u_{E}\cdot \nabla v\text{\,d}x=\int_{\Omega} 1v\text{\,d}x \quad \forall v\in \tilde{H}_0^1(E)
\end{align}	
and $u_{\omega}\in H_0^1(\omega)$ be the unique solution of
\begin{align}\label{Lap1H}
	\int_{\Omega}\nabla u_{\omega}\cdot \nabla v\text{\,d}x=\int_\Omega 1v\text{\,d}x \quad \forall v\in {H}_0^1(\omega).
\end{align}
We thus infer $u_{E}=u_{\omega}$ in $H_0^1(\omega)$. 
We now associate the quasi-open set $\omega$ with the Borel measure
\begin{align*}
	\infty_{\Omega\backslash \omega}(B)=
	\begin{cases}
		0&\quad \text{if } \mathrm{cap}(B\cap\Omega\backslash\omega)=0,\\
		+\infty &\quad \text{else},
	\end{cases}	
\end{align*}
for any Borel-set $B$.
Using the notation $\mu_{\omega}\coloneqq \infty_{\Omega\backslash\omega}$, this allows us to define the so called \textit{relaxed Dirichlet problem} associated to \eqref{Lap1H}: Find $u_{\mu_{\omega}}\in X_{\mu_{\omega}}(\Omega)\coloneqq H_0^1(\Omega)\cap L^2_{\mu_{\omega}}(\Omega)$
\begin{align}\label{RelDir}
	\int_{\Omega}\nabla u_{\mu_{\omega}}\cdot \nabla v\text{\,d}x+\int_{\Omega} u_{\mu_{\omega}} v\text{\,d}\mu_{\omega}=
	\int_{\Omega}1v\text{\,d}x\quad \forall v\in X_{\mu_{\omega}}(\Omega).
\end{align}
Here $H_0^1(\Omega)\cap L^2_{\mu_{\omega}}(\Omega)$ denotes the space of functions $v\in H_0^1(\Omega)$ fulfilling
\begin{align*}
	\int_{\Omega}v^2\text{\,d}\mu_{\omega}<\infty.
\end{align*}
This problem is studied in depth in \cite{MasoMurat,AttouchButtazzo, ButtazzoMaso91, Bucur, HenrotPierre} as this is the canonical limiting problem when passing to the limit with a sequence of solutions of the classical Dirichlet--Laplace problem formulated on open sets $(\Omega_n)_{n\in\N}$. Due to \cite{ButtazzoMaso91, Bucur}, $X_{\mu_{\omega}}$ is a Hilbert space with the scalar product
\begin{align*}
	(u,v)_{X_{\mu_{\omega}}}\coloneqq \int_{\Omega}\nabla u\cdot \nabla v\text{\,d}x+\int_{\Omega}uv\text{\,d}\mu_{\omega}.
\end{align*}
Using \cite[Ex.~4.3.2]{Bucur} and the fact that there exists a finely open set $\mathcal{A}\subset \Omega$ and a set of zero capacity $\mathcal{N}$ such that $\omega=\mathcal{A}\cup \mathcal{N}$ (see \cite[Thm.~4.1.5]{Bucur}), one deduces that
\begin{align*}
	H_0^1(\omega)=X_{\mu_{\omega}},
\end{align*}
and therefore $u_{\mu_{\omega}}=u_{\omega}$ in $H_0^1(\omega)$.
Note that the comparison principle known for classical elliptic PDEs still holds for the relaxed Dirichlet problem, see \cite[Prop.~2.6]{MasoMurat}, which allows us to deduce that $u_{\mu_{\omega}} \ge 0$ q.e.~in $\Omega$ and $u_{\mu_{\omega}}\in L^{\infty}(\Omega).$ Furthermore, using the aforementioned comparison principle along with the relations shown above, we eventually obtain
\begin{align*}
	\tilde{H}_0^1(E)=H_0^1(\omega)=H_0^1(\left\{u_{\omega}>0\right\})=H_0^1(\left\{u_{\mu_\omega}>0\right\})=H_0^1(\left\{u_E>0\right\}).
\end{align*}
Using these identities, it is easy to see that
\begin{align*}
    H_0^1(\left\{u_{\omega}>0\right\})=\tilde{H}_0^1(\left\{u_{\omega}>0\right\}).
\end{align*}
This theory will be particularly helpful in Theorem~\ref{GamHelp}, where the $\limsup$ inequality will be established.

First of all, we now analyze the limit eigenvalue problem. After that, we will establish its connection to the diffuse eigenvalue problem \eqref{WEstateD}.

\begin{Thm}\label{LimEV}
 In addition to the assumptions made in Section~\ref{Sec:Form}, we suppose that the assumptions \ref{Asmpt:Dom}--\ref{Asmpt:Pot} are fulfilled. 
For any given $\varphi\in BV(\Omega,\left\{\pm 1\right\})\cap \mathcal{U}$, we consider the following eigenvalue problem:
Find $(w,\lambda)\in \left(V^{\varphi}\backslash\{0\}\right)\times \R$ such that
\begin{align}\label{EW0}
    \io{\nabla w\cdot \nabla \eta}=\lambda\io{w\eta}\quad \forall \eta\in V^{\varphi}.
\end{align}
Then the following holds true:
	\begin{enumerate}[label=\textnormal{(\alph*)},leftmargin=*]
		\item\label{itemEV}
		There exists a sequence
		\begin{align*}
    		  \left(
    		      w_k^{0,\varphi},\lambda_k^{0,\varphi}
    		  \right)_{k\in \N}
    		\subset \left(
    		V^{\varphi}\backslash\{0\}\right)\times \R,
		\end{align*}
		having the following properties:
		\begin{itemize}
			\item \label{itemSpec} For all indices $k\in\N$, $w^{0,\varphi}_k$ is an $L^2(\Omega)$-normalized eigenfunction of \eqref{EW0} to the eigenvalue $\lambda^{0,\varphi}_k$ and these eigenfunctions are pairwise orthogonal with respect to the canonical scalar product on $L^2(\Omega)$ denoted by $(\cdot,\cdot).$
			\item The eigenvalues $\lambda_k^{0,\varphi}$ (which are repeated according to their multiplicity) can be ordered in the following way:
			\begin{align*}
    			0<\lambda_1^{0,\varphi}
    			\le\lambda_2^{0,\varphi}
    			\le\lambda_3^{0,\varphi}
    			\le\cdots.
			\end{align*}
			Moreover, it holds that $\lambda_k^{0,\varphi}\to \infty$ as $k\to\infty$, and there exist no further eigenvalues of the limit problem $\eqref{EW0}$.
			\item The normalized eigenfunctions 
			\begin{align*}
			  \left\{
			    \frac{w^{0,\varphi}_1}{\sqrt{\lambda^{0,\varphi}_1}},
			    \frac{w^{0,\varphi}_2}{\sqrt{\lambda^{0,\varphi}_2}},
			    \dots
			  \right\}\subset V^{\varphi},
			\end{align*}
		form an orthonormal basis of the space $V^{\varphi}$ and any $u\in V^{\varphi}$ can be expressed as
		\begin{align*}
		    u=\sum_{i=1}^{\infty}(u,w_i^{0,\varphi})w_i^{0,\varphi},
		\end{align*}
		where the series converges in $V^{\varphi}$.
        \end{itemize}
		\item\label{itemCF}
		For any $k\in\N$, we have the Courant--Fischer characterization
		\begin{align}\label{CFD0}
                \lambda_k^{0,\varphi}
            =
                \underset{U\in\mathcal{S}_{k-1}}{\max}\min
                    \left\{
                        \left.
                            \frac{
                                \int_{\Omega}
                                    \abs{\nabla v}^2
                                \textup{\,d}x
                             }
                             {
                                \int_{\Omega}\abs{v}^2\textup{\,d}x
                             }
                      \right|
                        \begin{aligned}
                        	&v\in U^{\perp}\cap V^{\varphi}, \\
                            &v\neq 0
                        \end{aligned}
                \right\}.
        \end{align} 
		Here, $\mathcal S_{k-1}$ denotes the collection of all $(k-1)$-dimensional subspaces of $L^2(\Omega)$.
		The set $U^{\perp}$ denotes the orthogonal complement of $U\subset L^2(\Omega)$ with respect to the canonical scalar product.
		
		Moreover, the maximum is attained at the subspace
		\begin{align*}
		      U
          =
              \langle                   
                w^{0,\varphi}_1,\dots,w^{0,\varphi}_{k-1}
              \rangle_{\textup{span}}.
		\end{align*}
	\end{enumerate}
\end{Thm}
\begin{proof}
\textit{Proof of \ref{itemEV}.} Equipping $V^{\varphi}\subset \Hz$ with the canonical $\Hz$ scalar product 
\begin{align*}
    (\cdot,\cdot)_{V^{\varphi}}:
    V^{\varphi}\times V^{\varphi}&\to \mathbb{R},\quad
    (v,w)_{V^{\varphi}}:= \io{\nabla v\cdot \nabla w},
\end{align*}
it is a closed subspace of $\Hz$ and hence, it is a Hilbert space. Using standard arguments, we conclude the existence of a self-adjoint and compact solution operator 
\begin{align*}
    \mathcal{T}: H^{\varphi}&\to V^{\varphi}\hookrightarrow H^{\varphi},\quad
    \mathcal{T}(f):= v_{f},
\end{align*}
where $v_{f}$ denotes the unique solution of
\begin{align*}
        \io{\nabla v_{f}\cdot \nabla \eta}
    =
        \io{f\eta}
        \quad\forall \eta\in V^{\varphi}.
\end{align*}
We point out that the operator $\mathcal{T}$ is not necessarily injective. However, by assumption \ref{Asmp:NTV}, we have $C_0^{\infty}(B)\subset V^{\varphi}\subset H^{\varphi}$ and the operator $\mathcal{T}$ restricted to $C_0^{\infty}(B)$ is obviously injective due to the fundamental lemma in the calculus of variations. We thus conclude that the image $\mathcal{T}(H^{\varphi})$ of the non-restricted operator is an infinite dimensional space.

Now, the spectral theorem for compact self-adjoint linear operators yields the first two claims of \ref{itemSpec} as well as the decomposition
\begin{align}\label{EWDecomp}
    H^{\varphi}
 =
    N(\mathcal{T})\perp \langle w_1^{0,\varphi},w_2^{0,\varphi},\dots\rangle_{\text{span}},
\end{align}
where $N(\mathcal{T})$ denotes the kernel of $\mathcal{T}$. The remaining assertions of (a) are established by applying the same techniques as in the proof of \cite[Thm.~2]{Evans}.

\textit{Proof of \ref{itemCF}.} Using the results above, the claim can be verified by proceeding similarly as in \cite[Thm.~3.2(b)]{Garcke}.
\end{proof}

In general, due to possible cracks within the set $E^{\varphi}$, we cannot guarantee that an eigenfunction $w$ of the limit problem vanishes on the whole of $\partial E^{\varphi}$.
Nevertheless, we know from Remark~\ref{RemLimit}.\ref{itemCr} that $w=0$ on $\partial\left(\textup{int}(E^{\varphi}_c)\right)$ in the trace sense (provided that the boundary is sufficiently smooth), meaning that $w$ has trace zero at least on the \textit{outer} boundary of $E^{\varphi}$. For more details we refer to \cite[Sec.~3]{Hecht}.
 If $E^{\varphi}$ is actually an open set with Lip\-schitz boundary, then indeed $w\in H_0^1(E^{\varphi})=\tilde{H}_0^1(E^{\varphi})$ where $H_0^1(E^{\varphi})$ can be understood in the standard sense, meaning that the trace of $w$ vanishes on $\partial E^{\varphi}$.
We thus interpret \eqref{EW0} as the weak formulation of the classical eigenvalue problem
\begin{align*}
\left\{
\begin{aligned}
    -\Delta w&=\lambda w&&\quad\text{in }E^{\varphi},\\
    w&=0&&\quad\text{on }\partial E^{\varphi},
\end{aligned}
\right.
\end{align*}
where the boundary condition is included in the space $V^{\varphi}$ of test functions in a relaxed way.

\begin{Rem}\label{epsEV}
Recall from Theorem~\ref{EEW} that for given $\varepsilon>0$ and $\varphi\in\Li$, we can fix the sequence
\begin{align*}
    \left(
        w_k^{\varepsilon,\varphi},\lambda_k^{\varepsilon,\varphi}
    \right)_{k\in \mathbb{N}}\subset \Hz\times \mathbb{R}^{+}
\end{align*}
of eigenfunctions and eigenvalues to \eqref{WEstateD}, where $\left\{w_1^{\varepsilon,\varphi},w_2^{\varepsilon,\varphi},\dots,\right\}\subset L^2(\Omega)$ forms an orthonormal basis. This notation will be used throughout the paper, and we will drop the additional index $\varepsilon$ if the context is clear.
\end{Rem}

\paragraph{Notation.}
Here, and in the remainder of this paper, the simplified notation $\left(\ie{\zeta}\right)_{\varepsilon>0}$ stands for $\left(\zeta_{\eps_k}\right)_{k\in\N}$ where $(\eps_k)_{k\in\N}$ denotes an arbitrary sequence with $\eps_k \to 0$ as $k\to\infty$. In this sense, a subsequence extraction from $\left(\ie{\zeta}\right)_{\varepsilon>0}$ is to be understood as a subsequence extraction from the associated sequence $(\eps_k)_{k\in\N}$. For convenience, our subsequences will not be relabeled, meaning that for any subsequence, we will use the same notation as for the whole sequence it was extracted from.

\bigskip

The following lemma will now link the diffuse interface problem $\eqref{WEstateD}$ to the sharp interface problem \eqref{EW0} for the first eigenvalue and it will serve as initial case for all higher eigenvalues.
This rigorous continuity analysis is new to the best of the authors' knowledge. We would like to mention that there are further results in optimal partitioning for the principal eigenvalue using a phase-field approach which show the convergence of minimizing eigenvalues of the optimization problem, see \cite{Bourdin, BogoselVelichkov}.

\begin{Lem}\label{LEM:lam1Conv}
     In addition to the assumptions made in Section~\ref{Sec:Form}, we suppose that the assumptions \ref{Asmpt:Dom}--\ref{Asmpt:Pot} are fulfilled. 
     Let $\left(\ie{\varphi}\right)_{\varepsilon>0}\subset L^1(\Omega)$ with $\abs{\ie{\varphi}}\le 1$ almost everywhere in $\Omega$ and let $\varphi\in BV(\Omega,\left\{\pm 1\right\})\cap \mathcal{U}$ such that
     \begin{align}\label{phieL1}
         \underset{\varepsilon\searrow 0}{\lim}
         \norm{\ie{\varphi}-\varphi}_{L^1(\Omega)}=0,
     \end{align}
     and the convergence exhibits the rate 
     \begin{align*}
        \norm{\ie{\varphi}-\varphi}_{
           L^1\left(E^{\varphi}\cap\left\{\ie{\varphi}<0\right\}\right)
        }=\mathcal{O}(\varepsilon).
     \end{align*}
     
     Then there exists an eigenfunction $u\in V^{\varphi}$ of the limit problem \eqref{EW0} to the eigenvalue $\lambda_1^{0,\varphi}$ such that
     \begin{align*}
        \underset{\varepsilon\searrow 0}{\lim}
        \io{b_{\varepsilon}(\varphi_{\varepsilon})\abs{w_1^{\varphi_{\varepsilon}}}^2}=
        \io{b_0(\varphi)\abs{u}^2}&=0,
     \end{align*}
     as well as
     \begin{align*}
         \underset{\varepsilon\searrow 0}{\lim}
         \norm{w_1^{\varphi_{\varepsilon}}-u}_{\HO}=0
         \quad\text{and}\quad
         \limse\lambda_1^{\varepsilon,\varphi_{\varepsilon}}=\lambda_1^{0,\varphi},
     \end{align*}
     up to subsequence extraction.
\end{Lem}

\pagebreak[2]

\begin{Rem}~
    \begin{enumerate}[label=\textnormal{(\alph*)},leftmargin=*]
        \item We point out that we will always use the letter $u$ (or $u_1,u_2,\dots$) to denote the limit of eigenfunctions. This is done in order to avoid confusion with the orthogonal system $\{w_1^{0,\varphi},w_2^{0,\varphi},\dots\}\subset V^{\varphi}$ of eigenfunctions to the limit problem we obtained in Theorem~\ref{LimEV}. 
        From the above lemma, we merely know that $u$ belongs to the first eigenspace which is spanned by the first eigenfunctions $w_i^{0,\varphi}$ in accordance with the multiplicity of the space. However, we cannot relate $u$ and $w_i^{0,\varphi}$ any further.
        \item Note that, in contrast to \cite{GarHe}, it would suffice to demand the above convergence rate condition only on the set $\tilde{E}^{\varphi}\cap\left\{\varphi_{\varepsilon}<0\right\}$ with $\tilde{E}^{\varphi}\coloneqq\{x\in \tOmega \,\vert\, \varphi(x)=1\}\subset \tOmega$ instead of $E^{\varphi}=\{x\in \Omega \,\vert\,\varphi(x)=1\}\subset \Omega$. As we have $\ie{\varphi},\varphi\in \mathcal{U}$, the difference $\ie{\varphi}-\varphi$ vanishes on $S_0\cup S_1$ anyway. 
    \end{enumerate}
\end{Rem}

\begin{proof}[Proof of Lemma~\ref{LEM:lam1Conv}]
    Some ideas of the proof are the same as in \cite[Lem.~3]{GarHe}, and we will mention those parts only briefly. We will also divide the proof into several steps for the sake of readability.
    In the following, due to \eqref{phieL1}, we may consider a non-relabeled subsequence of $(\varphi_{\eps})_{\eps>0}$ such that $\varphi_{\eps}\to\varphi$ a.e.~in $\Omega$.
   
    \textit{Step 1: For almost every $x\in \Omega$, it holds that
                $$\limse\; \ie{b}(\ie{\varphi}(x))=b_0(\varphi(x)).$$}%
   For the inequality
    \begin{align*}
        \limsupse b_{\eps}(\varphi_{\eps}(x))\le b_0(\varphi(x)),
    \end{align*}
     in the proof of Step~1 of \cite[Lem.~1]{GarHe}, the additionally assumed continuity of $b_0$ in the point $1$ given in \ref{Asmp:b} is needed. More precisely, as $\varphi\in BV(\Omega,\left\{\pm 1\right\})$ the inequality is clear if $\varphi(x)=-1$ because then $b_0(\varphi(x))=+\infty$. If $\varphi(x)=1$ then due to the pointwise convergence $\varphi_{\eps}(x)\to \varphi(x)$, the continuity of $b_0$ in $1$ and the fact $b_{\eps}\le b_0$ pointwise, we obtain
    \begin{align*}
        \limsupse b_{\eps}(\varphi_{\eps}(x))\le b_0(\varphi(x)).
    \end{align*}
    To prove the $\liminf$ inequality, we proceed exactly as in \cite{GarHe}. We point out that for this step, the convergence rate imposed on $(\varphi_{\eps})_{\eps>0}$ is not needed.
    
    \textit{Step 2:
    For any $v\in \HO$ with
    \begin{align}\label{vEvan}
       v\vert_{ \Omega\backslash E^{\varphi}}=0,
    \end{align}
    i.e., $v\in V^{\varphi}$, it holds that
    \begin{align*}
    \limse \io{\ie{b}(\ie{\varphi})\abs{v}^2}=\io{b_0(\varphi)\abs{v}^2}=0.
    \end{align*}}%
    This step can be established as in \cite{GarHe} since by assumption \ref{Asmp:b}, the coefficient function $\ie{b}$ possesses all the properties of \cite{GarHe}. This step heavily relies on the convergence rate imposed on $(\varphi_{\eps})_{\eps>0}$.
    
    In the remainder of this proof, we establish the convergence properties for the eigenvalues and eigenfunctions using the Courant--Fisher characterization.
    In the following, we write $\ie{w}$ and $\ee{\lambda}$ instead of $w_1^{\ie{\varphi}}$ and $\lambda_1^{\varepsilon,\ie{\varphi}}$, respectively, for convenience.
    
    \textit{%
        Step 3: We find a subsequence of $\left(w_{\varepsilon}\right)_{\varepsilon>0}$ and an $L^2(\Omega)$-normalized function $u\in V^{\varphi}$
        such that
        \begin{alignat}{3}\label{wConv}
            \begin{aligned}
                \ie{w}\rightharpoonup u&&\text{in }\HO\quad\text{and}\quad
                \ie{w}\to u&&\text{in }L^2(\Omega).
            \end{aligned}
        \end{alignat}}%
    For any eigenfunction solving \eqref{WEstateD} to the smallest eigenvalue $\lambda^{\eps}$ we recall the Courant--Fischer characterization \eqref{CFD} which simplifies to
	\begin{align*}
            \ee{\lambda}
        &=
            \min
                \left\{
                    \left.
                        \frac{
                            \int_{\Omega}
                                \abs{\nabla v}^2
                            \text{\,d}x
                          +\int_{\Omega}
                                b_{\varepsilon}(\ie{\varphi})\abs{v}^2
                             \text{\,d}x
                         }
                         {
                            \int_{\Omega} \abs{v}^2\text{\,d}x
                         }
                  \right|
                    \begin{aligned}
                             	&v\in \Hz, \\
                        &v\neq 0
                    \end{aligned}
            \right\}\\
        &=
            \min
                \left\{
                    \left.
                            \int_{\Omega}
                                \abs{\nabla v}^2
                            \text{\,d}x
                          +\int_{\Omega}
                                b_{\varepsilon}(\ie{\varphi})\abs{v}^2
                             \text{\,d}x
                  \right|
                    \begin{aligned}
                             	&v\in \Hz, \\
                        &\norm{v}_{L^2(\Omega)}=1
                    \end{aligned}
            \right\},                    
    \end{align*}
 as we have fixed the coefficient functions $\ie{a}$ and $\ie{c}$ in assumption \ref{Asmp:ac}. We define
 \begin{align*}
        F_{\varepsilon}: Q&\to \mathbb{R}^{+}_0, \quad 
        v\mapsto
        \int_{\Omega}
             \abs{\nabla v}^2
        \text{\,d}x
        +\int_{\Omega}
            b_{\varepsilon}(\ie{\varphi})\abs{v}^2
        \text{\,d}x,
 \end{align*}
 with $Q\coloneqq \{v\in \Hz \,\vert\, \norm{v}_{L^2(\Omega)}=1 \}$.
 In this way, $w_{\varepsilon}\in Q$ fulfills
 \begin{align}\label{wMin}
    F_{\varepsilon}(w_{\varepsilon})=\underset{v\in Q}{\min}\,F_{\varepsilon}(v).
 \end{align}
 To describe the limit situation, we similarly define
 \begin{align*}
     F_0: Q\to [0,+\infty],\quad
    v\mapsto
       \int_{\Omega}
             \abs{\nabla v}^2
        \text{\,d}x
        +\int_{\Omega}
            b_{0}(\varphi)\abs{v}^2
        \text{\,d}x.
 \end{align*}   
 
 By assumption \ref{Asmp:NTV}, $V^{\varphi}$ is non trivial and hence, there exists a function $\overline{v}\in Q$ satisfying property \eqref{vEvan}.
 Then, $\eqref{wMin}$ obviously entails that
 \begin{align*}
    \ie{F}(\ie{w})\le \ie{F}(\overline{v}),
 \end{align*}
 and from \textit{Step~2}, we already know that there is a constant $C>0$ such that for all $\varepsilon>0$,
 \begin{align*}
    \io{\ie{b}(\ie{\varphi})\abs{\overline{v}}^2}\le C.
 \end{align*}
 We thus infer that
 \begin{align*}
    \norm{\nabla \ie{w}}^2_{L^2(\Omega)} \le \ie{F}(\ie{w}) \le \ie{F}(\overline{v}) \le C,
 \end{align*}
 and combining these two bounds, the Banach--Alaoglu theorem implies the desired convergences \eqref{wConv} up to subsequence extraction.
 
\textit{Step 4: The function $u\in Q$ is a minimizer of $F_0$ and we have
            \begin{align}\label{GammaConv}
                \limse F_{\varepsilon}(w_{\varepsilon})=F_0(u),
            \end{align}
        along a non-relabeled subsequence.}

If we can show that $F_{\varepsilon}$ $\Gamma$-converges to $F_0$ on $Q$ with respect to the weak topology on $\HO$, we can apply classical results from $\Gamma$-convergence theory which give exactly the claimed properties, cf.~\cite{DalMaso}.
To prove $\Gamma$-convergence we have to verify the corresponding $\limsup$ and $\liminf$ inequalities.

To verify the $\limsup$ inequality, for any $v\in Q$ we need to find a so called recovery sequence $(v_{\varepsilon})_{\varepsilon>0}\subset Q$ that converges to $v\in Q$ weakly in $\HO$ and satisfies
\begin{align}\label{limsupIneq}
    \limsupse \ie{F}(\ie{v})\le F_0(v).
\end{align}
Here, we can simply choose the constant sequence $v_\eps := v\in Q$. Without loss of generality, we can assume that $F_0(v)<+\infty$,
as otherwise \eqref{limsupIneq} is trivially fulfilled. This assumption implies that
\begin{align*}
    \io{b_0(\varphi)\abs{v}^2}<\infty,
\end{align*}
and since $b_0(-1)=+\infty$ and $\varphi\in BV(\Omega,\left\{\pm 1\right\})$, we conclude that $v\in V^{\varphi}$. Hence, we infer from \textit{Step~2} that
\begin{align}\label{infVan}
    \io{b_0(\varphi)\abs{v}^2}=\limse\; \io{\ie{b}(\ie{\varphi})\abs{v}^2}=0.
\end{align}
By construction of $\ie{F}$ and $F_0$, this already implies \eqref{limsupIneq}.

For the $\liminf$ inequality we need to show that for any sequence $(\ie{v})_{\varepsilon>0}\subset Q$ converging to a $v\in Q$ weakly in $\HO$ topology, it holds
\begin{align}\label{liminfIneq}
    F_0(v)\le \liminfse \ie{F}(\ie{v}).
\end{align}
By the compact embedding $\HO\hookrightarrow L^2(\Omega)$ we know that $v_{\varepsilon}\to v$ almost everywhere in $\Omega$ up to subsequence extraction. Furthermore, we have already seen in \textit{Step~1} that for almost every $x\in \Omega$,
\begin{align*}
    \limse \ie{b}(\ie{\varphi}(x))=b_0(\varphi(x)).
\end{align*}
Therefore, we deduce
\begin{align*}
        \io{b_0(\varphi_0(x))\abs{v(x)}^2}
    &=
        \io{\limse \ie{b}(\ie{\varphi}(x))\; \limse \abs{\ie{v}(x)}^2}\\
    &=
        \io{\liminfse\, \big[\ie{b}(\ie{\varphi}(x))\abs{\ie{v}(x)}^2\big]}\\
    &\le
        \liminfse \io{\ie{b}(\ie{\varphi}(x))\abs{\ie{v}(x)}^2},
\end{align*}
by means of Fatou's lemma. Noting that the $H^1_0(\Omega)$ norm $\norm{\nabla \,\cdot\,}_{L^2(\Omega)}$ is weakly lower semicontinuous this yields the $\liminf$ inequality \eqref{liminfIneq}.

In summary, this means that
\begin{align*}
    F_{\varepsilon}\overset{\Gamma}{\to}F_0,
\end{align*}
with respect to the weak $H^1(\Omega)$-topology. Eventually, we can use standard $\Gamma$-convergence results (see, e.g., \cite{DalMaso}) to deduce the claim of \textit{Step~4}.

\medskip

We now want to complete the proof by applying the results established in the previous steps.
Proceeding as in \textit{Step~4}, and using the convergence properties in \eqref{wConv}, we deduce
\begin{align*}
    \io{b_0(\varphi)\abs{u}^2}&\le \liminfse \io{\ie{b}(\ie{\varphi})\abs{\ie{w}}^2},\\
    \io{\abs{\nabla u}^2}&\le \liminfse\io{\abs{\nabla w_{\varepsilon}}^2}.
\end{align*}
As both sequences are bounded from below by zero, we can use \eqref{GammaConv} along with \cite[Lem.~4]{GarHe} to infer
\begin{align*}
    \io{b_0(\varphi)\abs{u}^2}&= \limse \io{\ie{b}(\ie{\varphi})\abs{\ie{w}}^2},\\
    \io{\abs{\nabla u}^2}&= \limse\io{\abs{\nabla w_{\varepsilon}}^2}.
\end{align*}
From the second convergence and the weak convergence \eqref{wConv} from \textit{Step~3}, we conclude that
\begin{align*}
     \underset{\varepsilon\searrow 0}{\lim}
     \norm{\ie{w}-u}_{\HO}=0.
\end{align*}
Furthermore, we have seen in the previous step that $u\in Q$ minimizes $F_0$. Hence,
\begin{align*}
    \io{b_0(\varphi)\abs{u}^2}<+\infty,
\end{align*}
and arguing as in \eqref{infVan}, we find that
\begin{align}\label{w0Van}
    \io{b_0(\varphi)\abs{u}^2}=0.
\end{align}

To complete the proof, we still have to prove the following assertion:

\textit{%
    Step 5: The function $u\in V^{\varphi}$ solves
    \begin{align}
            \io{\nabla u\cdot \nabla\eta}
        =
            \lambda_1^{0,\varphi}\io{u\eta} \quad \text{for all }\eta\in V^{\varphi},
    \end{align}
    and
    $\lambda_1^{0,\varphi}=\underset{\varepsilon\searrow 0}{\lim}\;{\ee{\lambda}}$.} 

If we can show that $u\in V^{\varphi}\cap Q$ is even a minimizer of
\begin{align}\label{LimMin1}
        v\mapsto
        \frac{
        \int_{\Omega}
            \abs{\nabla v}^2
        \text{\,d}x
     }
     {
        \int_{\Omega} \abs{v}^2\text{\,d}x
     }
     \quad\text{subject to}\quad 
    v\in V^{\varphi}, 
    v\neq 0,
\end{align}    
we directly infer from the first order condition associated with this minimization problem and the Courant--Fischer characterization \eqref{CFD0} that $u$ solves \eqref{EW0} to the eigenvalue $\lambda_1^{0,\varphi}$.

As $u\in V^{\varphi}$ is a minimizer of $F_0$ over $Q$ we use \eqref{w0Van} to deduce
\begin{align*}
    \min
        \left\{
            \left.
                \frac{
                    \int_{\Omega}
                        \abs{\nabla v}^2
                    \text{\,d}x
                 }
                 {
                    \int_{\Omega} \abs{v}^2\text{\,d}x
                 }
          \right|
            \begin{aligned}
                &v\in V^{\varphi}, \\
                &v\neq 0
            \end{aligned}
    \right\}
    &\le
        \frac{
            \int_{\Omega}
                \abs{\nabla u}^2
            \text{\,d}x
            }
            {
                \int_{\Omega} \abs{u}^2\text{\,d}x
            }
    =F_0(u)\\        
    &=\min
        \left\{
            \left.
                \frac{
                    \int_{\Omega}
                        \abs{\nabla v}^2
                    \text{\,d}x
                  +\int_{\Omega}
                        b_0(\varphi)\abs{v}^2
                     \text{\,d}x
                 }
                 {
                    \int_{\Omega} \abs{v}^2\text{\,d}x
                 }
          \right|
            \begin{aligned}
                              	&v\in \Hz, \\
                &v\neq 0
            \end{aligned}
        \right\}\\
    &\le\min
        \left\{
            \left.
                \frac{
                    \int_{\Omega}
                        \abs{\nabla v}^2
                    \text{\,d}x
                  +\int_{\Omega}
                        b_0(\varphi)\abs{v}^2
                     \text{\,d}x                    
                 }
                 {
                    \int_{\Omega} \abs{v}^2\text{\,d}x
                 }
          \right|
            \begin{aligned}
                &v\in V^{\varphi}, \\
                &v\neq 0
            \end{aligned}
        \right\} .   
\end{align*}
Here, the last inequality holds because $V^{\varphi}\subset \Hz$. However, we already know from \textit{Step~2} that $\io{b_0(\varphi)\abs{v}^2}=0$
for all $v\in V^{\varphi}$. Hence, we conclude from the above estimate that $u\in V^{\varphi}$ minimizes \eqref{LimMin1}, and in particular, $F_0(u)=\lambda_1^{0,\varphi}$.
Now, the second claim of \textit{Step~5} follows from \eqref{GammaConv} since $\ie{F}(\ie{w})=\ee{\lambda}$ holds by construction.

The proof of Lemma~\ref{LEM:lam1Conv} is now complete.
\end{proof}

This lemma now serves as initial step to show the analogous properties also for all higher eigenvalues and eigenfunctions via induction.
\begin{Thm}\label{Thm:Sconv}
 In addition to the assumptions made in Section~\ref{Sec:Form}, we suppose that the assumptions \ref{Asmpt:Dom}--\ref{Asmpt:Pot} are fulfilled. 
     Let $k\in \mathbb{N}$, and suppose that $\left(\ie{\varphi}\right)_{\varepsilon>0}\subset L^1(\Omega)$ and $\varphi\in BV(\Omega,\left\{\pm 1\right\})\cap\mathcal{U}$ fulfill the same assumptions as in Lemma~\ref{LEM:lam1Conv}.
     
     Then, there exists an eigenfunction $u_k\in V^{\varphi}$ of the limit problem \eqref{EW0} to the eigenvalue $\lambda_k^{0,\varphi}$ such that the convergences
    \begin{gather*}
        \underset{\varepsilon\searrow 0}{\lim}
        \io{\ie{b}(\ie{\varphi})\abs{w^{\ie{\varphi}}_k}^2}=
        \io{b_0(\varphi)\abs{u_k}^2}=0,\\
         \underset{\varepsilon\searrow 0}{\lim}
         \norm{w^{\ie{\varphi}}_k-u_k}_{\HO} =0,
         \quad
         \limse\lambda_k^{\ie{\varphi}} =\lambda_k^{0,\varphi}
    \end{gather*}
    hold up to subsequence extraction.
\end{Thm}
\begin{proof}
 We prove the assertion via induction. The initial step was carried out in Lemma~\ref{LEM:lam1Conv}.
 
 Let us now assume that the assertion is already established for the first $k-1$ eigenfunctions $w_1^{\varepsilon}\coloneqq w_1^{\ie{\varphi}},\dots,w_{k-1}^{\varepsilon}\coloneqq w_{k-1}^{\ie{\varphi}}$ of \eqref{WEstateD}. 
 
 To prove the result for $w_k^{\varepsilon}\coloneqq w_k^{\ie{\varphi}}$, 
 let $\ie{W}\coloneqq\langle \ee{w_1},\dots,\ee{w_{k-1}}\rangle_{\text{span}}\subset L^2(\Omega)$ denote the space spanned by the first $k-1$ eigenfunctions and $\ie{W}^{\perp}$ its orthogonal complement with respect to the canonical scalar product on $L^2(\Omega)$ that is denoted by $\left(\cdot,\cdot\right)$.
 We further define the space
 \begin{align*}
        \ie{Q}
    \coloneqq
        \left\{
            v\in \Hz
          \left|
            \norm{v}_{L^2(\Omega)}=1\text{ and }
            v\in \ie{W}^{\perp}
          \right.  
        \right\},
 \end{align*}
 as well as the operator 
 \begin{align*}
        F_{\varepsilon}: Q_{\varepsilon}&\to [0,+\infty],\quad
        v\mapsto
        \int_{\Omega}
             \abs{\nabla v}^2
        \text{\,d}x
        +\int_{\Omega}
            b_{\varepsilon}(\ie{\varphi})\abs{v}^2
        \text{\,d}x.
 \end{align*}
Then, from the Courant-Fischer characterization \eqref{CFD} we infer that
 \begin{align*}
        \lambda^{\varepsilon}_k
    \coloneqq
        \lambda^{\ie{\varphi}}_k
    =
        \underset{v\in \ie{Q}}{\min} F_{\varepsilon}(v).                      
 \end{align*}

In the limit situation, we define the space
$W_0\coloneqq \langle u_1,\dots,u_{k-1}\rangle_{\text{span}}$ where the functions $u_i\in V^{\varphi}$, ${i=1,\dots,k-1}$ are determined by the induction hypothesis as the limits
 \begin{align*}
    \ee{w}_i\to u_i\quad \text{in } H^1(\Omega) \text{ as }\varepsilon\to 0,
 \end{align*} 
up to subsequence extractions, as discussed in Lemma~\ref{LEM:lam1Conv}. 
In particular, for $i,j=1,\dots,{k-1}$ with $i\neq j$, we have $(u_i,u_j)=0$ since $(\ee{w}_i,\ee{w}_j)=0$.
Moreover, we know that $\norm{\ee{w}_i}_{L^2(\Omega)}=1$ and hence, we also have $\norm{u_i}_{L^2(\Omega)}=1$ for $i=1,\dots,k-1$.
This means that $\left\{u_1,\dots,u_{k-1}\right\}\subset L^2(\Omega)$ is an orthonormal basis of the $(k-1)$-dimensional space $W_0\subset V^{\varphi}$.
We further set
 \begin{align*}
    Q_0
    \coloneqq
        \left\{
            v\in \Hz
          \left|
            \norm{v}_{L^2(\Omega)}=1\text{ and }
            v\in W_0^{\perp}
          \right.  
        \right\},    
 \end{align*}
and we define the operator
 \begin{align*}
     F_0: Q_0&\to [0,+\infty],\quad
    v\mapsto
       \int_{\Omega}
             \abs{\nabla v}^2
        \text{\,d}x
        +\int_{\Omega}
            b_{0}(\varphi)\abs{v}^2
        \text{\,d}x.
 \end{align*}

 Now we introduce the orthogonal projections
  \begin{align*}
     &\ie{P}: L^2(\Omega)\to \ie{W},\quad
     v\mapsto \sum_{i=1}^{k-1}(v,\ee{w}_i)\ee{w}_i,\\
     &P_0: L^2(\Omega)\to W_0,\quad
     v\mapsto \sum_{i=1}^{k-1}(v,u_i)u_i.
  \end{align*}
  In the following, these projections will be a useful tool to construct recovery sequences.
  
 Per construction, $\ee{w_k}\in \Hz$ is a minimizer of $F_{\varepsilon}$. Now, we need to show that there exists a constant $C>0$ that does not depend on $\eps>0$ such that  
 \begin{align}\label{boundF}
    \ie{F}(\ee{w_k})\le C,
 \end{align}
 as this allows us to bound $\left(w_k^{\varepsilon}\right)_{\varepsilon>0}$ in the $\HO$ norm.
 
 As in \textit{Step~3} of the proof of Lemma~\ref{LEM:lam1Conv}, we want to choose suitable elements $\ee{v}$ in the feasible sets $\ie{Q}$ for which we can bound the sequence $(\ie{F}(\ee{v}))_{\varepsilon>0}.$ Here, the situation is more complicated compared to Lemma~\ref{LEM:lam1Conv} as the feasible set $\ie{Q}$ depends on $\varepsilon$.
 
 Due to assumption \ref{Asmp:NTV}, we find $v^0\in V^{\varphi}$ such that
 \begin{align*}
    v^0\in W_0^{\perp}\backslash \left\{0\right\}.
 \end{align*}
 Otherwise, $V^{\varphi}$ would be a subset of the $(k-1)$-dimensional space $W_0$, which is a contradiction to the fact that $V^{\varphi}$ is infinite dimensional.
 Let us define the sequence
 \begin{align*}
    \ee{v}\coloneqq v^0-\sum_{i=1}^{k-1}\left(v^0,\ee{w_i}\right)\ee{w_i}=v^0-\ie{P}(v^0)\in \Hz\cap \ie{W}^{\perp}.
 \end{align*}
 Now, by the induction hypothesis, for every $i=1,\dots,k-1$, we know that
 \begin{align}\label{wiConv}
    \ee{w_i}\to u_i \quad\text{in } H^1(\Omega),
 \end{align}
along a suitable non-relabeled subsequence. Hence, from the construction of $v^\eps$, we infer
 \begin{align}\label{vConv}
    \ee{v}\to v^0\quad\text{in }H^1(\Omega).
 \end{align}
 In particular, for $\varepsilon>0$ sufficiently small, we thus have $\ee{v}\neq 0$.
 Altogether this allows us to define the sequence
 \begin{align*}
    \ee{\overline{v}}=\frac{\ee{v}}{\norm{\ee{v}}_{L^2(\Omega)}}\in \ie{Q},
 \end{align*}
 which fulfills the convergence 
 \begin{align*}
    \ee{\overline{v}}\to \overline{v}^0\coloneqq \frac{v^0}{\norm{v^0}_{L^2(\Omega)}}\in {Q}_0
    \quad\text{in $H^1(\Omega)$.}
 \end{align*}
 If we can now verify that 
 \begin{align}\label{wvEst}
  F_{\varepsilon}(\ee{\overline{v}})\le C,
 \end{align}
 uniformly in $\varepsilon$, \eqref{boundF} directly follows as our minimizer $\ee{w}_k\in Q_{\varepsilon}$ obviously fulfills
 \begin{align*}
     \ie{F}(\ee{w}_k)\le F_{\varepsilon}(\ee{\overline{v}}).
 \end{align*}
 Therefore, we recall that
 \begin{align*}
     F_{\varepsilon}(\ee{\overline{v}})=\io{\abs{\nabla \ee{\overline{v}}}^2}+\io{\ie{b}(\ie{\varphi})\abs{\ee{\overline{v}}}^2}.
 \end{align*}
 For the first summand on the right-hand side, we obtain
 \begin{align*}
    \io{\abs{\nabla \ee{\overline{v}}}^2}=
    \frac{1}{\norm{\ee{v}}^2_{L^2(\Omega)}}
    \left(
        \io{\abs{\nabla \ee{v}}^2}
    \right).
 \end{align*}
 Hence, this term is bounded because of \eqref{vConv} which further entails the convergence $\norm{\ee{v}}_{L^2(\Omega)}\to\norm{v^0}_{L^2(\Omega)}>0$.
For the second summand, we use Young's inequality to obtain
 \begin{align*}
    \io{\ie{b}(\ie{\varphi})\abs{\ee{\overline{v}}}^2}\le
    \frac{2}{\norm{\ee{v}}^2_{L^2(\Omega)}}
    \left(
        \io{\ie{b}(\ie{\varphi})\bigabs{v^0}^2}+
        \io{\ie{b}(\ie{\varphi})\bigabs{\ie{P}(v^0)}^2}
    \right).
 \end{align*}
 As, per construction, $v^0\in V^{\varphi}$ fulfills property $\eqref{vEvan}$ we can apply \textit{Step 2} of the proof of Lemma~\ref{LEM:lam1Conv} which yields
 \begin{align*}
 \limse \io{\ie{b}(\ie{\varphi})\bigabs{v^0}^2}=0.
 \end{align*}
 Furthermore, Lemma~\ref{LEM:lam1Conv} implies that for $i=1,\dots,k-1$,
 \begin{align*}
 \limse \io{\ie{b}(\ie{\varphi})\bigabs{\ee{w}_i}^2}=0.
 \end{align*}
 Now, as
 \begin{align*}
    \ie{P}(v^0)=\sum_{i=1}^{k-1}\left(v^0,\ee{w_i}\right)\ee{w_i},
 \end{align*}
 we obtain  
 \begin{align*}
    \limse \io{\ie{b}(\ie{\varphi})\bigabs{\ie{P}(v^0)}^2}=0,
 \end{align*}
 by applying Young's inequality again.
 Altogether, we deduce
 \begin{align*}
       \limse \io{\ie{b}(\ie{\varphi})\abs{\ee{\overline{v}}}^2}=0.     
 \end{align*}
 This proves the estimate \eqref{wvEst} which directly entails the uniform bound \eqref{boundF}. In particular, we have
 \begin{align*}
        \io{\abs{\nabla\ee{w}_k}^2}
    \le
        \ie{F}(\ee{w}_k)
    \le
        C.    
 \end{align*}
 Applying the Banach--Alaoglu theorem and the compact embedding $\HO\hookrightarrow L^2(\Omega)$ we infer the existence of a limit $u_k\in \Hz$ such that
 \begin{align}\label{wkConv}
        \ee{w}_k\rightharpoonup u_k\quad\text{in }\HO,
        \quad
        \ee{w}_k\to u_k\quad\text{in }L^2(\Omega),
        \quad
        \ee{w}_k\to u_k\quad\text{a.e.~on }\Omega
 \end{align}
as $\varepsilon\to 0$ up to subsequence extraction.
 
 Our next task is to show that $u_k$ belongs to $Q_0$ and fulfills
 \begin{align}\label{wkMin}
    F_0(u_k)=\underset{v\in Q_0}{\min}F_0(v).
 \end{align}
 First of all, we can use the convergence \eqref{wiConv} of the first $k-1$ eigenfunctions along with \eqref{wkConv} to obtain the convergence
 \begin{align*}
    \limse (\ee{w}_k,\ee{w}_i)=(u_k,u_i),
 \end{align*}
 for $i=1,\dots,k-1$. However, by the orthogonality of the eigenfunctions for $\varepsilon>0$, we know $0=(\ee{w}_k,\ee{w}_i)$, and thus $(u_k,u_i)=0$, $i=1,\dots,k-1$. This already proves that $u_k\in W_0^{\perp}$.
 Notice that $\norm{u_k}_{L^2(\Omega)}=1$,  as the $\ee{w_k}$ are assumed to be $L^2(\Omega)$-normalized. All in all, we get $u_k\in Q_0$.
 
To verify \eqref{wkMin} we cannot directly apply the theory of $\Gamma$-convergence as in Lemma~\ref{LEM:lam1Conv} but we can establish similar estimates that will help us to obtain the desired properties.
For the sake of a clearer presentation we divide this part of the proof into several steps.
 
 \textit{%
    Step 1: The following $\liminf$ inequality holds:
    \begin{align}\label{wkLiminf}
        F_0(u_k)\le \liminfse \ie{F}(\ee{w}_k).
    \end{align}
 }%
 To prove the assertion, we recall that
 \begin{align*}
        F_0(u_k)
    =
        \io{\abs{\nabla u_k}^2}+\io{b_0(\varphi_0)\abs{u_k}^2}.
 \end{align*}
 For the gradient term, we obtain the inequality
 \begin{align*}
        \io{\abs{\nabla u_k}^2}
    \le
        \liminfse\io{\abs{\nabla \ee{w}_k}^2},
 \end{align*}
 by using the weak lower semicontinuity of this expression.
 Now, due to the convergence properties \eqref{wkConv}, we are exactly in the same situation as in \textit{Step~4} of the proof of Lemma~\ref{LEM:lam1Conv}. Hence, Fatou's lemma yields
 \begin{align*}
        \io{b_0(\varphi)\abs{u_k}^2}
    \le
        \liminfse\io{\ie{b}(\ie{\varphi})\abs{\ee{w_k}}^2}.
 \end{align*}
 In summary, we infer \eqref{wkLiminf}.
 
 \textit{%
    Step 2: For any $v\in Q_0$ there exists a sequence $(\ee{\overline{v}})_{\eps>0} \subset Q_{\varepsilon}$ which satisfies
    \begin{align}\label{wkLimsup}
        \limsupse \ie{F}(\ee{\overline{v}})\le F_0(v).
    \end{align}
 }%
 Here, finding such a recovery sequence is more complicated than in \textit{Step~4} of Lemma~\ref{LEM:lam1Conv}, as we cannot just take the constant sequence $v$.
 Without loss of generality, we assume that $F_0(v)<+\infty$ as otherwise \eqref{wkLimsup} is trivially fulfilled. This guarantees that
 \begin{align*}
     \io{b_0(\varphi_0)\abs{v}^2}<+\infty.
 \end{align*}
Hence, $v$ fulfills property \eqref{vEvan} (i.e., $v\in V^{\varphi}$) which will be needed later.
Analogously to the beginning of this proof, we now define the sequence
 \begin{align*}
    \ee{v}:= v-\sum_{i=1}^{k-1}(v,\ee{w}_i)\ee{w}_i=v-\ie{P}(v)\in \ie{W}^{\perp}.
 \end{align*}
 Exactly as in \eqref{vConv}, we obtain the convergence
 \begin{align*}
    \ee{v}\to v-\sum_{i=1}^{k-1}(v,u_i)u_i=v-P_0(v)\in W_0^\perp
        \quad\text{in $H^1(\Omega)$.}
 \end{align*}
However, since $v\in Q_0$, we also have $v\in W_0^{\perp}$ meaning that $P_0(v)=0$. We thus get
$\ee{v}\to v$ in $\HO$.
 Furthermore, $v\in Q_0$ ensures that $\norm{v}_{L^2(\Omega)}=1>0$ and hence, we infer that for $\varepsilon>0$ sufficiently small, it holds that
$\norm{\ee{v}}_{L^2(\Omega)}>0$.
We can thus consider the normalized sequence
 \begin{align}
    \ee{\overline{v}}\coloneqq \frac{\ee{v}}{\norm{\ee{v}}_{L^2(\Omega)}}\in \ie{Q},
    \quad\text{which fulfills}\quad
    \label{uConv}
    \ee{\overline{v}}\to v\quad\text{in }\HO.
 \end{align}
We now prove \eqref{wkLimsup} by again considering the gradient term and the term involving $b_\eps$ appearing in $F_\eps$ separately.
Using \eqref{uConv}, we infer that
 \begin{align*}
    \limse\io{\abs{\nabla \ee{\overline{v}}}^2}=\io{\abs{\nabla v}^2}.
 \end{align*}
For the second term, considering the representation
 \begin{align*}
        \ee{\overline{v}}
    =
        \frac{1}{\norm{\ee{v}}}\left[v-\sum_{i=1}^{k-1}(v,\ee{w}_i)\ee{w}_i\right],
 \end{align*}
 we see that this sequence has exactly the same properties as the same-named sequence in the beginning of this proof. Hence, proceeding as above, we use the convergence properties known for $\ee{w_i}$, the convergence $\norm{\ee{v}}_{L^2(\Omega)}\to \norm{v}_{L^2(\Omega)}>0$ and the crucial fact that $v\in V^{\varphi}$ to deduce
 \begin{align*}
    \limse \io{\ie{b}(\ee{\varphi})\abs{\ee{\overline{v}}}}=0=\io{b_0(\varphi)\abs{v}^2}.
 \end{align*}
 Hence, in particular, this verifies \eqref{wkLimsup}.
 
 \medskip
 
 Now, combining these two steps, we obtain for any arbitrary $v\in Q_0$,
 \begin{align*}
        F_0(u_k)
    \le
        \liminfse \ie{F}(\ee{w}_k)
    \le
        \limsupse \ie{F}(\ee{w}_k)
    \le
        \limsupse \ie{F}(\ee{\overline{v}})
    \le       
        F_0(v),    
 \end{align*}
since $\ee{w}_k$ is a minimizer of $\ie{F}$ over $\ie{Q}$.
 As $v\in Q_0$ was arbitrary, this finally shows that
 \begin{align*}
    F_0(u_k) = \underset{v\in Q_0}{\min}F_0(v).
 \end{align*}
 Furthermore, plugging $v=u_k$ into the above chain of estimates we get
 \begin{align*}
    F_0(u_k)=\liminfse \ee{F}(\ee{w_k})=\limsupse \ee{F}(\ee{w_k}),
 \end{align*}
 which directly yields
 \begin{align}\label{FkConv}
    \limse \ie{F}(\ee{w}_k)=F_0(u_k).
 \end{align}
 As in the proof of Lemma~\ref{LEM:lam1Conv}, this allows us to deduce
 \begin{align*}
      \underset{\varepsilon\searrow 0}{\lim}
      \norm{\ee{w}_k-u_k}_{\HO}=0,
      \quad\text{and}\quad
      u_k\in V^{\varphi}.
 \end{align*}

Therefore, it only remains to prove the following statement.

 \textit{%
     Step 3: The function $u_k\in V^{\varphi}$ solves
     \begin{align}\label{LimUk}
            \io{\nabla u_k\cdot \nabla \eta}
         =
            \lambda_k^{0,\varphi}
            \io{u_k\eta}
            \quad\text{for all } \eta\in V^{\varphi},
     \end{align}
      and it holds that
      $\lambda_k^{0,\varphi}=\limse\lambda_k^{\ie{\varphi}}$.
 }
 
 With an analogous reasoning as in \textit{Step~5} of the proof of Lemma~\ref{LEM:lam1Conv}, we see that
 \begin{align}\label{lamk}
     \tilde{\lambda}\coloneqq\min
         \left\{
             \left.
                 \frac{
                     \int_{\Omega}
                         \abs{\nabla v}^2
                     \text{\,d}x
                  }
                  {
                     \int_{\Omega} \abs{v}^2\text{\,d}x
                  }
           \right|
             \begin{aligned}
                    &v\in V^{\varphi}\cap W_0^{\perp}, \\
                    &v\neq 0
             \end{aligned}
        \right\}=
    \frac{
            \int_{\Omega}
                \abs{\nabla u_k}^2
            \text{\,d}x
        }
        {
            \int_{\Omega} \abs{u_k}^2\text{\,d}x
        }
     =F_0(u_k).
 \end{align}
 As $W_0=\langle u_1,\dots,u_{k-1}\rangle_{\text{span}}$ is a $(k-1)$-dimensional subspace of $V^{\varphi}$, the Courant--Fischer characterization \eqref{CFD0} entails  that $\lambda_k^{0,\varphi}\ge \tilde{\lambda}$. 

Furthermore, for $\varepsilon>0$, we have 
$\lambda_k^{\ie{\varphi}}\ge \lambda_{k-1}^{\ie{\varphi}}$,
 and by \eqref{FkConv} and the induction hypothesis we infer
 \begin{align*}
     \limse \lambda_k^{\ie{\varphi}}=
     \limse \ie{F}(w_k^{\varepsilon})=
     F_0(u_k)=
     \tilde{\lambda},
     \quad\text{and}\quad
     \limse \lambda_{k-1}^{\ie{\varphi}}=\lambda_{k-1}^{0,\varphi}.
 \end{align*}
 This proves that $\tilde{\lambda}\ge \lambda_{k-1}^{0,\varphi}$.
 
 Now, to show that $\tilde\lambda = \lambda_{k}^{0,\varphi}$, we need to consider two cases.
 
 \textit{Case 1:} It holds that $\lambda_k^{0,\varphi}=\lambda_{k-1}^{0,\varphi}$. 
 Then, from the above considerations we already infer that $\lambda_k^{0,\varphi}=\tilde{\lambda}=\lambda_{k-1}^{0,\varphi}$.
 
 \textit{Case 2:} It holds that $\lambda_k^{0,\varphi}>\lambda_{k-1}^{0,\varphi}$. 
 Then, the span of the eigenfunctions $\{w_1^{0,\varphi},\dots,w_{k-1}^{0,\varphi}\}\subset V^{\varphi}$ (given as in Theorem~\ref{LimEV}) contains the union of all eigenspaces belonging to the eigenvalues $\lambda_1^{0,\varphi},\dots,\lambda_{k-1}^{0,\varphi}$. On the other hand, we know from the induction hypothesis that $\left\{u_1,\dots,u_{k-1}\right\}$ is a linearly independent family of eigenfunctions belonging to the aforementioned eigenvalues. Hence, we conclude that
 \begin{align}
 \label{ID:W0}
        W_0=\langle u_1,\dots,u_{k-1}\rangle_{\text{span}}
    =
        \langle w_1^{0,\varphi},\dots,w_{k-1}^{0,\varphi}\rangle_{\text{span}}.
 \end{align}
 This means that $W_0$ is exactly the $(k-1)$-dimensional vector space on which the maximum in the Courant--Fischer characterization \eqref{CFD0} is attained. Thus, by \eqref{lamk}, we infer $\tilde{\lambda}=\lambda_k^{\varphi}$.
 
Now, computing the first-order optimality condition (Euler--Lagrange equation) of the minimization problem associated with  \eqref{lamk}, we conclude that $u_k\in V^{\varphi}$ and $\lambda_k^{\varphi}$ fulfill the weak formulation \eqref{EW0} of the limit problem for all test functions in $V^{\varphi}\cap W_0^{\perp}$. However, we know from \eqref{ID:W0} that $W_0$ is spanned by eigenfunctions to the limit problem \eqref{EW0}. This means that $u_k$ and $\lambda_k^{\varphi}$ trivially satisfy the weak formulation \eqref{EW0} for all test functions in $W_0$. In summary, this proves that $u_k$ is an eigenfunction of the limit problem \eqref{EW0} to the eigenvalue $\lambda_k^{0,\varphi}$. 

This completes the proof by induction. 
 \end{proof}
 
 \subsection{Sharp interface limit of the optimal control problem}\label{Sec:RigSIA}
We now show that a sequence of minimizers of the cost functionals for $\varepsilon>0$ converges, as $\varepsilon\to 0$, to a minimizer of the cost functional associated with the sharp interface setting which will be defined in the following.

First of all, we demand again that the coefficient functions satisfy the assumptions \ref{Asmp:ac} and \ref{Asmp:b}. For $\varepsilon>0$, we extend the cost functional of the problem \eqref{PD} to the space $L^1(\Omega)$ by defining
\begin{alignat}{3}\label{JeDef}
    \begin{aligned}
        J^{\varepsilon}(\varphi)\coloneqq
        \begin{cases}
            \begin{aligned}
                &\Psi(\lambda_{i_1}^{\varepsilon,\varphi},\dots,\lambda_{i_l}^{\varepsilon,\varphi})+
                \gamma E_{\text{GL}}^{\varepsilon}(\varphi)
                &&\quad\text{if } \varphi\in \Phi_{\text{ad}},\\
                &+\infty
                &&\quad\text{if } \varphi\in L^1(\Omega)\backslash \Phi_{\text{ad}}.
            \end{aligned}
        \end{cases}
    \end{aligned}    
\end{alignat}
Here, for any $k\in\mathbb{N}$, $\lambda^{\varepsilon,\varphi}_k$ denotes the $k$-th eigenvalue of the Dirichlet problem \eqref{WEstateD} with $\varepsilon>0$ and $\varphi\in \Phi_{\text{ad}}\subset L^{\infty}(\Omega)$.
In the sharp interface situation, we consider
\begin{align*}
    \tilde{\mathcal{G}}^{\beta}=
    \left\{
        \varphi\in L^1(\tOmega)
        \left|
            \abs{\varphi}\le 1,\;
            \beta_1\bigabs{\tOmega}\le \ito{\varphi}\le \beta_2\bigabs{\tOmega}
        \right.
    \right\},
\end{align*}
and define the cost functional as
\begin{align}\label{J0Def}
    \begin{aligned}
    &J^0(\varphi)\coloneqq
    \begin{cases}
        \begin{aligned}
            &\Psi(\lambda_{i_1}^{0,\varphi},\dots,\lambda_{i_l}^{0,\varphi})+
            \gamma c_0 P_{\tOmega}(\tilde{E}^{\varphi})
            &&\quad\text{if }
            \varphi\in \Phi_{\textup{ad}}^0,\\
            &+\infty
            &&\quad\text{if }
            \varphi\in L^1(\Omega)\backslash \Phi_{\textup{ad}}^0,
        \end{aligned}
    \end{cases}
    \\[1ex]
    &\text{where}\quad 
    \Phi_{\textup{ad}}^0
    \coloneqq
    BV(\Omega,\left\{\pm 1\right\})\cap\mathcal{U}\cap \tilde{\mathcal{G}}^{\beta}.
    \end{aligned}
\end{align}
Here, for any $k\in\mathbb{N}$, $\lambda_k^{0,\varphi}$ denotes the $k$-th eigenvalue of the limit problem \eqref{EW0} that was introduced in Theorem~\ref{LimEV} for $\varphi\in BV(\Omega,\left\{\pm 1\right\})\cap \mathcal{U}$. 

Let $P_{\tOmega}(\tilde{E}^{\varphi})$ denote the relative perimeter in $\tOmega$ of the set $\tilde{E}^{\varphi}\coloneqq\{ x\in \tOmega \,\vert\, \varphi(x)=1\}$,
 i.e.,
\begin{align*}
	P_{\tOmega}(\tilde{E}^{\varphi})\coloneqq 
	\sup \left\{\left.\int_{\tilde{E}^{\varphi}}\text{div}\B{\zeta} \text{\,d}x\,\right\vert\, \B{\zeta} \in C_0^1(\tOmega,\R^d), \norm{\zeta}_{L^{\infty}(\tOmega)}\le 1\right\}.
\end{align*}	
We further set
\begin{align*}
    c_0\coloneqq\int_{-1}^{1}\sqrt{2\psi_0(x)}\text{\,d}x,
\end{align*}
where $\psi_0$ is the potential appearing in the regularized Ginzburg--Landau energy
\begin{align*}
        E^{\eps}(\varphi)
    =
        \int_{\tOmega}
            \left(
                \frac{\eps}{2}\abs{\nabla\varphi}^2+\frac{1}{\eps}\psi_0(\varphi)
            \right)
        \text{\,d}x, 
        \quad \eps>0.
\end{align*}
Additionally to the assumptions in Section~\ref{SOPT}, we make the following assumption that is supposed to hold throughout the remainder of this section.
\begin{enumerate}[label = \textnormal{\bfseries{(A\arabic*)}}, start=7]
\item \label{PBounded}
    $\Psi$ is bounded from below, i.e., we find a constant $C_{\Psi}>0$ such that $\Psi(\bx)\ge -C_{\Psi}$ for all $\bx\in \left(\mathbb{R}_{>0}\right)^l$. Without loss of generality, we assume $C_{\Psi}=0$.
\end{enumerate}

Now, the goal is to show $\Gamma$-convergence of the cost functionals as this yields that a subsequence of minimizers $\ie{\varphi}$ of $\ee{J}$ converges in $L^1(\Omega)$ to a minimizer of $J^0$. In this sense, the diffuse interface optimization problem can be regarded as an approximation of the sharp interface optimization problem.

In our previous considerations, we needed to impose the rate condition
 \begin{align}
 \label{COND:RATE}
    \norm{\ie{\varphi}-\varphi_0}_{
       L^1\left(E^{\varphi_0}\cap\left\{\ie{\varphi}<0\right\}\right)
    }=\mathcal{O}(\varepsilon),
 \end{align}
 in order to show the desired properties such as the convergence of the eigenvalues as $\varepsilon\to 0$. However, to obtain a true unconditional $\Gamma$-convergence result, we do not want to impose such an additional assumption on our sequence of minimizers.
The $\liminf$ inequality can be shown for general cost functionals, i.e., for $\Psi$ fulfilling only the current assumptions. Furthermore, the proof does not rely on the continuity of eigenvalues when passing from diffuse to sharp interfaces. Therefore, no rate condition needs to be assumed.

For the $\limsup$ inequality, the classical recovery sequence for the Ginzburg--Landau energy constructed in \cite{BloweyElliott} fulfills the rate condition. However, it is a delicate aspect that this recovery sequence can only be constructed explicitly for sets fulfilling suitable regularity assumptions, but not for general finite perimeter sets. As also seen in \cite{Modica, BloweyElliott}, one therefore needs to approximate finite perimeter sets on the sharp interface level in a suitable way such that the perimeter converges and, in our framework, also the eigenvalues. This convergence of eigenvalues on the sharp interface level was studied in \cite{BucurHenrot, BogoselOudet} and can be applied here also in a slightly modified way in order to take care of the constraint formulated in $\mathcal{U}$. As done there we also need to assume that the cost functional satisfies a componentwise monotonicity.
Note that in \cite{BogoselOudet} the $\Gamma$-convergence was studied without any additional volume constraint, which allows the usage of the recovery sequence of \cite{ModicaMortola}. After the authors had shown the convergence of eigenvalues on the sharp interface level, their $\limsup$ inequality on the diffuse interface level was a direct consequence of the monotonicity of the cost functional. Hence, no continuity property for the eigenvalues was required. 
In our situation with an additional volume constraint, even though we also need to assume the monotonicity of the cost functional, we can rely on the continuity of eigenvalues on the diffuse interface level in the sense of Theorem~\ref{Thm:Sconv}. This allows us to use the recovery sequence from \cite{GarHe} which is based on the construction of \cite{Modica, BloweyElliott}.

To motivate the additional monotonicity assumption on $\Psi$, we first establish the following lemma. 
\begin{Lem}\label{Psiliminf}
    Let $X\subset \mathbb{R}^l$. We consider a continuous function
    \begin{align*}
        f: X\to \mathbb{R}.
    \end{align*}
    Then the following assertions are equivalent.
    \begin{enumerate}[label=\textnormal{(\alph*)},leftmargin=*]
        \item For any sequence $\left(\B{x}_k\right)_{k\in\mathbb{N}}\subset X$ and $\B{x}\in X$ fulfilling
        \begin{align*}
            \B{x}
        \le
            \underset{k\to \infty}{\liminf}\,\B{x}_k\in X \text{     componentwise},
    \end{align*}
    it holds
    \begin{align*}
            f(\B{x})
        \le 
            \underset{k\to \infty}{\liminf}\,f(\B{x}_k).
    \end{align*}
    \item f is monotonically increasing in the sense that for $\B{x},\B{y}\in X$,
    \begin{align}\label{Psimon}
        \B{x}\le\B{y}\,\text{componentwise }
        \Rightarrow
        f(\B{x})\le f(\B{y}).
    \end{align}
    \end{enumerate}
\end{Lem}
\begin{proof}
    The implication (a)$\Rightarrow$(b) follows by choosing the constant sequence $x_k=y$ for all $k\in \mathbb{N}$.
    In order to show (b)$\Rightarrow$(a) we recall the definition of the limes inferior,
    and we use the monotonicity and the continuity of $f$ to obtain
    \begin{align*}
            f(\B{x})
        \le
            f\left(\underset{k\to \infty}{\liminf}\,\B{x}_k\right)
        =
            \underset{n\to\infty}{\lim}
            f\left(
                \inf\left\{\left.
                    \B{x}_k\right| k\ge n
                \right\}
            \right).
    \end{align*}
    Exploiting again the monotonicity of $f$, we deduce that for all $n\in\mathbb{N}$,
    \begin{align*}
            f\left(
                \inf\left\{\left.
                    \B{x}_k\right| k\ge n
                \right\}
            \right)
        \le
            \inf \left\{\left.
                    f(\B{x}_k)\right| k\ge n
                \right\}.
    \end{align*}
    This implies that
    \begin{align*}
            f(\B{x})
        \le
            \underset{n\to \infty}{\lim} \inf\left\{\left.
                    f(\B{x}_k)\right| k\ge n
                \right\}
        =
            \underset{k\to \infty}{\liminf}\, 
                    f(\B{x}_k),
    \end{align*}
    an thus, the claim is established.
\end{proof}

In order to be able to apply Lemma~\ref{Psiliminf}, we make the following additional assumption on the function $\Psi$, which is supposed to hold throughout the remainder of this section. 
\begin{enumerate}[label = \textnormal{\bfseries{(A\arabic*)}}, start=8]
\item \label{AsP}
    The function $\Psi: \left(\mathbb{R}_{>0}\right)^{l}\to \mathbb{R}_{\ge 0}$
    is assumed to be monotonically increasing in the sense of Lemma~\ref{Psiliminf}
    and exhibit the coercivity property
    \begin{align}\label{Psibound}
        \left(\Psi(\B{x}_k)\right)_{k\in\mathbb{N}} \text{ is bounded}
        \Rightarrow
        \left(\B{x}_k\right)_{k\in\mathbb{N}} \text{ is bounded},
    \end{align}
    for any sequence $\left(\B{x}_k\right)_{k\in\mathbb{N}}\subset \left(\mathbb{R}_{>0}\right)^{l}$.
\end{enumerate}

These properties are for example fulfilled if $\Psi$ is given as a positive linear combination of the components, i.e.,
    \begin{align*}
        \Psi(\B{x})=\sum_{j=1}^{l}\alpha_j x_j,
    \end{align*}
where $\alpha_j>0$ for $j=1,\dots,l$. In the context of our cost functional this would mean that linear combinations of eigenvalues $\lambda_{i_j}^{\varepsilon,\varphi}$ and $\lambda_{i_j}^{0,\varphi}$ respectively are involved in our optimization process.
In particular, by choosing $l=1$ and $\psi(x) = x$ for all $x\in \mathbb{R}_{\ge 0}$, the minimization of just one single eigenvalue of course also fulfills the assumption on $\Psi$.

Assumption~\ref{AsP} might look a bit technical at first sight, but it is exactly what we need in order to establish the $\liminf$ inequality. The monotonicity of $\Psi$ combined with Lemma~\ref{Psiliminf} allows us to infer the $\liminf$ inequality for the cost functional from the $\liminf$ inequality for the eigenvalues. On the other hand the coercivity property \eqref{Psibound} entails that the sequence of eigenvalues is bounded uniformly in $\eps$ if the cost functionals stay bounded.

Under Assumption~\ref{AsP} it is now possible to establish an unconditional $\Gamma$-convergence result. 
\begin{Thm}\label{GamCost}
     In addition to the assumptions made in Section~\ref{Sec:Form}, we suppose that the assumptions \ref{Asmpt:Dom}--\ref{AsP} are fulfilled. 
    Then, it holds that
    \begin{align*}
        \ee{J}\overset{\Gamma}{\to}J^0 \quad \text{as }\varepsilon\to 0.
    \end{align*}
\end{Thm}

By standard results in the theory of $\Gamma$-convergence (see e.g., \cite{DalMaso}), this theorem directly yields the following corollary
due to the compactness properties of the Ginzburg--Landau energy from \cite[Thm.~3.7]{BloweyElliott}.

\begin{Cor}
     In addition to the assumptions made in Section~\ref{Sec:Form}, we suppose that the assumptions \ref{Asmpt:Dom}--\ref{AsP} are fulfilled.
    Let $\left(\ie{\varphi}\right)_{\varepsilon>0}$ be a sequence of minimizers of the functionals $(\ee{J})_{\varepsilon>0}$. Then there exists a function $\varphi_0\in L^1(\Omega)$, such that
    \begin{align*}
        \limse \norm{\ie{\varphi}-\varphi_0}_{L^1(\Omega)}
        =0
        ,\qquad
        \limse \ee{J}(\ie{\varphi})
        =J^0(\varphi_0),
    \end{align*}
    and $\varphi_0$ is a minimizer of $J^0$. 
    In particular, this means that $\varphi_0 \in \Phi_{\mathrm{ad}}^{0} \subseteq BV(\Omega,\{\pm 1\})$.
\end{Cor}

We now conclude this section by presenting the proof of the above theorem.
In order to tackle the volume constraint
\begin{align*}
	\beta_1\bigabs{\tOmega}\le \int_{\tOmega}\varphi \text{\,d}x\le \beta_2 \bigabs{\tOmega},
\end{align*}	
we first show a $\Gamma$-convergence result similar to \cite[Thm.~3.1]{BogoselOudet}, where the volume constraint is omitted and then, in a further step, we suitably modify the recovery sequence such that it actually fulfills the volume constraint.
\begin{Thm}\label{GamHelp}
	In addition to the assumptions made in Section~\ref{Sec:Form}, we suppose that the assumptions \ref{Asmpt:Dom}--\ref{AsP} are fulfilled.
	Let
	\begin{alignat}{3}\label{IeDef}
		\begin{aligned}
			I^{\varepsilon}(\varphi)\coloneqq
			\begin{cases}
				\begin{aligned}
					&\Psi(\lambda_{i_1}^{\varepsilon,\varphi},\dots,\lambda_{i_l}^{\varepsilon,\varphi})+
					\gamma E_{\mathrm{GL}}^{\varepsilon}(\varphi)
					&&\quad\text{if } \varphi\in \Lambda_{\textup{ad}},\\
					&+\infty
					&&\quad\text{if } \varphi\in L^1(\Omega)\backslash \Lambda_{\textup{ad}}
				\end{aligned}
			\end{cases}
		\end{aligned}
	\end{alignat}
and 
	\begin{align}\label{I0Def}
		\begin{aligned}
			&I^0(\varphi)\coloneqq
			\begin{cases}
				\begin{aligned}
					&\Psi(\lambda_{i_1}^{0,\varphi},\dots,\lambda_{i_l}^{0,\varphi})+
					\gamma c_0 P_{\tOmega}(\tilde{E}^{\varphi})
					&&\quad\text{if }
					\varphi\in \Lambda_{\textup{ad}}^0,\\
					&+\infty
					&&\quad\text{if }
					\varphi\in L^1(\Omega)\backslash \Lambda_{\textup{ad}}^0,
				\end{aligned}
			\end{cases}
		\end{aligned}
	\end{align}
with
\begin{align*}
		\Lambda_{\textup{ad}}
	&\coloneqq
		\left\{\left.\varphi\in H^1(\tOmega) \,\right\vert\, \abs{\varphi}\le 1\right\}\cap\mathcal{U},\\
		\Lambda_{\textup{ad}}^0
	&\coloneqq
		BV(\Omega,\left\{\pm 1\right\})\cap\mathcal{U}.
\end{align*}
Then, it holds that $I^{\eps}\overset{\Gamma}{\to} I$ as $\eps\to 0$.
\end{Thm}

To proof the assertion we will follow the reasoning in \cite{BogoselOudet}. Although, the arguments in \cite{BogoselOudet} contains highly valuable ideas, we have the impression that at some points the authors do not distinguish carefully enough between the global perimeter $P_{\R^d}$ on $\R^d$ and the relative perimeter $P_{\tOmega}$ on $\tOmega$ which does not see the boundary $\partial\tOmega$. This plays a crucial role when the $\Gamma$-convergence results of \cite{Modica,ModicaMortola,BloweyElliott} are applied. In the following, we thus present a very detailed proof where we take care that all steps are applicable for the relative perimeter $P_{\tOmega}$.

We further point out that in contrast to \cite{BogoselOudet}, our proof does \textit{not} rely on the property that the recovery sequence $(\varphi_{\eps})_{\eps>0}$ constructed in \cite{ModicaMortola} fulfills the inclusion
\begin{align*}
	\left\{\varphi=1\right\}\subset \left\{\varphi_{\eps}=1\right\} \quad \text{for all }\eps>0,
\end{align*}
where $\varphi_\eps\to \varphi$ in $L^1(\Omega)$. In \cite{BogoselOudet}, this inclusion is crucial to obtain the $\limsup$ inequality for the eigenvalues. As we do not require this condition, we can construct our recovery sequence in the spirit of \cite{GarHe, BloweyElliott}. Our strategy is based on the continuity properties of eigenvalues shown in the previous section. In this way, we achieve that our $\Gamma$-convergence result holds for any general coefficient function $b_{\eps}$ fulfilling Assumption~\ref{Asmp:b}. In particular, this means that the coefficient function can be chosen in a more general way compared to the explicit affine linear construction in \cite{BogoselOudet}.

\begin{proof}[Proof of Theorem~\ref{GamHelp}]
As previously explained, we need to approximate any general finite perimeter set by a sequence of (sufficiently) smooth sets in order to construct a recovery sequence for the $\lim\sup$ inequality. The construction of such an approximate sequence of smooth sets is presented now.

\textit{Step 1: For $\tilde{E}\subset \tOmega$ with $P_{\tilde{\Omega}}(\tilde{E})<\infty$ there exists a sequence of bounded smooth open sets
$E_{k}\subset \R^d$ fulfilling
\begin{alignat}{2}
    \left\{
	\begin{aligned}\label{ApproxE}
		\mathcal{H}^{d-1}(\partial E_{k}\cap\partial\tOmega)&=0,\\
		\underset{k\to\infty}{\lim}P_{\tilde{\Omega}}(E_{k})&=P_{\tOmega}(\tilde{E}),\\
		\underset{k\to\infty}{\lim}\varphi_{k}&=\varphi\text{ in }L^1(\Omega),\\
		\underset{k\to\infty}{\limsup}\,\lambda^{0,\varphi_{k}}&\le\lambda^{0,\varphi},
	\end{aligned}
	\right.
\end{alignat}
where $\varphi_{k}\coloneqq2\chi_{E_{k}^{\tOmega}\cup S_1}-1$ and $\varphi\coloneqq2\chi_{\tilde{E}\cup S_1}-1$, and for $m\in\N$, $\lambda^{0,\varphi}=\lambda_m^{0,\varphi}$, stands for an arbitrary eigenvalue. Here, for any set $A\subset \R^d$, we use the notation $A^{\tOmega}:=A\cap \tOmega$.}

To construct an approximate sequence of bounded smooth open sets, we follow the proof of \cite[Lem.~1]{Modica} which can also be found in \cite[Lem.~13.9]{Rindler}. For the sake of readability, we merely explain the key steps of this construction and will stick to the notation of \cite{Rindler}. Note that we cannot assume that the finite perimeter set $\tilde{E}$ contains an open ball. For pure perimeter minimization this assumption would be justified by \cite[Thm.~1]{Gonzalez} if only the convergence of minimizers is to be shown. For that reason, we cannot easily adjust the volume of the approximating sets $E_k$ by including or excluding balls as it was done in \cite{Modica, Rindler}. To overcome this, we will adjust the volume of the recovery sequence only at the diffuse interface level which will eventually be done in the proof of Theorem~\ref{GamCost}. An alternative way of tackling the volume constraint on the sharp interface level is performed in \cite{Sternberg} which does not need the finite perimeter set $\tilde{E}$ to contain any open ball. There, the approximating sequence $E_k$ is modified by adding or subtracting suitable hypercubes, but due to the rather technical construction this would require a delicate discussion in order to analyze the $\limsup$ inequality of eigenvalues in \eqref{ApproxE}.

The key idea of constructing a sequence $(E_k)_{k\in\N}$ of bounded smooth open sets is to extend $\chi_E\in BV(\tOmega)\cap L^{\infty}(\tOmega)$ to a function $v\in BV(\R^d)\cap L^{\infty}(\R^d)$ with $\abs{Dv}(\partial\tOmega)=0$ (which is possible as $\tOmega$ is assumed to be a bounded Lipschitz domain, see \cite[Prop.~3.21]{AmbrosioFuscoPallara}). Note that the set $\left\{v\neq 0\right\}$ is still bounded. Here, $Dv$ denotes the Radon measure associated with $v\in BV(\R^d)$ and $\abs{\,\cdot\,}$ denotes the total variation. It is crucial that $\abs{Dv}(\partial\tOmega)=0$ as we want to approximate the \textit{relative} perimeter which does not see the boundary of the design domain $\partial\tOmega$ but only the parts of the boundary of $\tilde{E}$ lying within $\tOmega$.

Now, in order to construct a sequence of smooth approximating sets $E_k$ fulfilling \eqref{ApproxE}, we choose a standard sequence of mollifieres $(\rho_n)_{n\in\N}\subset C_0^{\infty}(\R^d)$ and consider the superlevel sets
\begin{align*}
	\left\{v_n>t\right\},
	\quad\text{where $v_n\coloneqq v\ast \rho_n$,}
\end{align*}
for $t\in(0,1)$. 
In contrast to \cite{Modica,Rindler}, where for each $n\in\N$, a specific $t_n\in \big(\frac{1}{n},1-\frac{1}{n}\big)$ is selected in order to show the convergence of the corresponding super level sets with respect to perimeter and measure, we use the ideas of \cite[Proof of Thm.~3.1]{BogoselOudet} to obtain these convergences even for almost all $t\in (0,1)$. 

Due to our extension, we have $\abs{Dv}(\partial\tOmega)=0$. Proceeding as in \cite{Rindler}, we thus get
\begin{align*}
	\underset{n\to\infty}{\lim}\int_{\tilde{\Omega}}\abs{\nabla v_n}\text{\,d}x=P_{\tOmega}(\tilde{E}).
\end{align*}
In combination with the coarea formula for the relative perimeter (see \cite[Thm.~3.40]{AmbrosioFuscoPallara}) and Fatou's lemma, we deduce as in \cite{BogoselOudet} that
\begin{align*}
	P_{\tOmega}(\tilde E)=\underset{n\to\infty}{\lim}\int_0^1 P_{\tOmega}(\left\{v_n>t\right\})\text{\,d}t\ge 
	\int_0^1 \underset{n\to\infty}{\liminf}\, P_{\tOmega}(\left\{v_n>t\right\})\text{\,d}t.
\end{align*}	
On the other hand, as in \cite{Rindler}, we infer that for almost every $t\in (0,1)$,
\begin{align*}
	\abs{\left(\left\{v_n>t\right\}\cap\tOmega\right)\bigtriangleup \tilde E}\to 0
\end{align*}	
as $n\to\infty$ and thus,
\begin{align*}
	P_{\tOmega}(\tilde E)\le \underset{n\to\infty}{\liminf}\,P_{\tOmega}(\left\{v_n>t\right\}),
\end{align*}
due to the lower semicontinuity of the perimeter.
Combining the previous inequalities, we conclude
\begin{align}\label{convPer}
	P_{\tOmega}(\tilde E)=\underset{n\to\infty}{\liminf}\,P_{\tOmega}(\left\{v_n>t\right\})
\end{align}	
for almost every  $t\in (0,1)$.
Now, according to \cite{Rindler}, the properties
\begin{align}
	\nabla v_n (x)\neq 0 \text{ for all }x\in  \R^d \text{ with } v_n(x)&=t\quad\text{and}\label{Regk}\\
	\mathcal{H}^{d-1}\left(\left\{\left.x\in \partial\tOmega\right\vert v_n(x)=t\right\}\right)&=0,\label{Transk}
\end{align}
hold for all $n\in \N$ and almost all $t\in (0,T)$.
In summary, this means that we can choose a Lebesgue null set $\mathcal{N}\subset (0,1)$ such that for every $t\in (0,1)\backslash\mathcal{N}$ the sets $E_{n,t}\coloneqq \left\{v_n>t\right\}$ are bounded, smooth and fulfill the transversality condition
\begin{align*}
	\mathcal{H}^{d-1}(\partial E_{n,t}\cap\partial\tOmega)=0.
\end{align*}
After extracting a suitable (non-relabeled) subsequence (possibly depending on the choice of $t$), we further infer the convergence properties
\begin{alignat}{2}
    \left\{
	\begin{aligned}\label{SharpStabD}
		\underset{n\to\infty}{\lim}P_{\tilde{\Omega}}(E_{n,t})&=P_{\tOmega}(\tilde{E}),\\
		\underset{n\to\infty}{\lim}\varphi_{n,t}&=\varphi \quad\text{in }L^1(\Omega),
	\end{aligned}
	\right.
\end{alignat}
where $\varphi_{n,t}\coloneqq 2\chi_{E_{n,t}^{\tOmega}\cup S_1}-1$.

It thus remains to establish the $\limsup$ inequality for the eigenvalues. As the eigenvalue equation is formulated on the whole of $\Omega$ (not only $\tOmega$), we now consider $E\coloneqq\tilde{E}\cup S_1$. Here we can exactly apply the strategy employed in \cite[Thm.~3.1]{BogoselOudet} which can also be found in \cite[Thm.~3.5]{BucurHenrot}. For the sake of readability, we explain the key steps.

As mentioned in the beginning of Subsection~\ref{Sec:LimProp}, there is a quasi-open set $\omega\subset \Omega$ such that $V^{\varphi}=\tilde{H}_0^1(E)=H_0^1(\omega)$. Now, we choose $u_{\omega}\in H_0^1(\omega)$ as the solution of the Laplace equation \eqref{Lap1H}. It then holds $H_0^1(\omega)=H_0^1(\left\{u_{\omega}>0\right\})=\tilde{H}_0^1(\left\{u_{\omega}>0\right\})$. Hence, in particular, we have $\lambda^{0,\varphi}=\lambda^{0,\chi_{\left\{u_{\omega}>0\right\}}}$. Furthermore, we know $u_{\omega}\in L^{\infty}(\Omega)$ and hence, without loss of generality, we may assume that $u_{\omega}\le 1$ a.e.~on $\Omega$. 
Due to the inclusion $\left\{u_{\omega}>0\right\}\subset \omega\subset E$, we have $u_{\omega}\le \chi_{E} = v$ a.e.~on $\Omega$ and hence up to a Lebesgue null set,
\begin{align*}
	\left\{u_{\omega}\ast \rho_n>t\right\}\cap\Omega
	\subset \left\{v \ast\rho_n>t\right\}\cap \Omega
	\subset E_{n,t}^{\Omega} 
	\subset E_{n,t}^{\tOmega} \cup S_1,
\end{align*}
for all $t\in(0,1)$ and $n\in\N$.
In particular, we have
\begin{align*}
	\left\{u_{\omega}\ast \rho_n>t\right\}\cap \omega\subset E_{n,t}^{\tOmega}\cup S_1,
\end{align*}
up to a Lebesgue null set, for all $t\in (0,1)$ and $n\in\N$. Hence, due to the monotonicity of eigenvalues with respect to set inclusion, it holds
\begin{align*}
	\lambda^{0,\varphi_{n,t}}\le \lambda^{0,\chi_{\left\{u_{\omega}\ast \rho_n>t\right\}\cap \omega}}\le \lambda^{0,\chi_{\left\{u_{\omega}\ast \rho_n>t\right\}\cap\left\{u_{\omega}>t\right\}\cap \omega}}.
\end{align*}
Now, using the density result \cite[Prop.~5.5]{MasoMurat}, it was shown in \cite{BogoselOudet,BucurHenrot} that for all $t\in (0,1)\backslash\mathcal{N}$, 
\begin{align*}
	\left\{u_{\omega}\ast \rho_n>t\right\}\cap\left\{u_{\omega}>t\right\}\overset{\gamma}{\to}\left\{u_{\omega}>t\right\},
\end{align*}
as $k\to\infty$, in the sense of $\gamma$-convergence (see, e.g., \cite{HenrotPierre}).
As the $\gamma$-convergence is stable under intersection with quasi-open sets (see e.g. \cite[Prop.~4.5.6]{Bucur}), we conclude
\begin{align*}
		\left\{u_{\omega}\ast \rho_n>t\right\}\cap\left\{u_{\omega}>t\right\}\cap \omega\overset{\gamma}{\to}\left\{u_{\omega}>t\right\}\cap \omega.
\end{align*}
Now due to the continuity of the eigenvalue with respect to $\gamma$-convergence (see e.g., \cite[Cor.~6.1.8]{Bucur}), we have	
\begin{align*}
	\underset{n\to\infty}{\lim}\lambda^{0,\chi_{\left\{u_{\omega}\ast \rho_n>t\right\}\cap\left\{u_{\omega}>t\right\}\cap \omega}}=\lambda^{0,\chi_{\left\{u_{\omega}>t\right\}\cap \omega}},
\end{align*}
and thus,
\begin{align*}
	\underset{n\to\infty}{\limsup}\,\lambda^{0,\varphi_{n,t}}\le \lambda^{0,\chi_{\left\{u_{\omega}>t\right\}\cap \omega}}.
\end{align*}	
By means of the density result mentioned above, one can further show that for any zero sequence $(t_n)_{n\in \N}$, it holds
\begin{align*}
        \left\{u_{\omega}>t_n\right\}
    \overset{\gamma}{\to}
        \left\{u_{\omega}>0\right\},
\end{align*}
as $n\to\infty$ and thus
\begin{align*}
	\underset{n\to\infty}{\lim}\lambda^{0,\chi_{\left\{u_{\omega}>t_n\right\}\cap \omega}}=\lambda^{0,\chi_{\left\{u_{\omega}>0\right\}\cap \omega}}=\lambda^{0,\chi_{\left\{u_{\omega}>0\right\}}}
	=\lambda^{0,\varphi},
\end{align*}
where the second equality is valid due to the inclusion $\left\{u_{\omega}>0\right\}\subset \omega$.

Hence, by a diagonal sequence argument, we can now choose a zero sequence $(t_k)_{k\in\N}\subset (0,1)\backslash \mathcal{N}$ and a sequence of indices $(n_k)_{k\in\N}\in\N$ such that $E_k\coloneqq E_{n_k,t_k}$ fulfills the desired properties \eqref{ApproxE}.

\textit{%
	Step 2:
	Let $\varphi\in L^1(\Omega)$ be arbitrary. There exists a recovery sequence $(\varphi_{\varepsilon})_{\varepsilon>0}\subset L^1(\Omega)$ with
	\begin{align*}
		\limse\norm{\ie{\varphi}-\varphi}_{L^1(\Omega)}=0,
	\end{align*}
	such that the $\limsup$ inequality
	\begin{align*}
		\limsupse I^{\varepsilon}(\varphi_{\varepsilon})\le I^0(\varphi),
	\end{align*}
	holds.
}

Without loss of generality, we assume $I^0(\varphi)<\infty$. We thus have $\varphi\in \Lambda_{\textup{ad}}^0\subseteq BV(\Omega,\left\{\pm 1\right\})$.  Due to the previous step there exists a sequence of bounded smooth open sets $(E_k)_{k\in\N}\subset \R^d$ approximating $\tilde{E}^{\varphi}$ satisfying all the properties in \eqref{ApproxE}.
Now, the idea is to construct for each $k$ a recovery sequence $(\varphi_{k,\eps})_{\eps>0}\subset \Lambda_\text{ad}$ for $\varphi_k\coloneqq 2\chi_{E_k^{\tOmega}\cup S_1}-1\in \Lambda_\text{ad}^0$.
Due to the properties of the set $E_k$ we can proceed as in \cite[Thm.~2]{GarHe} (which relies on the ideas of \cite{BloweyElliott, Sternberg, Modica}). 
Note that we operate on the open subset $\tOmega\subset \Omega$ where the pointwise constraints incorporated in $\mathcal{U}$ do not play any role.
In this way, for every $k\in\N$, we obtain a recovery sequence 
\begin{align*}
	(\varphi_{k,\varepsilon})_{\eps>0}\subset \left\{\left.\varphi\in H^1(\tOmega)\right\vert \abs{\varphi}\le 1\right\},
\end{align*}	
which satisfies
\begin{align}\label{PerSup}
	\limsupse
	\ito{
		\frac{\gamma\varepsilon}{2}
		\abs{\nabla \varphi_{k,\varepsilon}}^2
		+\frac{\gamma}{\eps}\psi(\varphi_{k,\varepsilon})
	}
	\le
	\gamma c_0P_{\tOmega}(\tilde{E}^{\varphi_k}),
\end{align}
and
\begin{align*}
	\norm{\varphi_{k,\varepsilon}-\varphi_k}_{L^1(\tOmega)}
	=\mathcal{O}(\varepsilon).
\end{align*}
For any $\eps>0$, the function $\varphi_{k,\varepsilon}$ can be extended onto the whole design domain $\Omega$ by choosing $\varphi_{k,\varepsilon}:=-1$ on $S_0$ and $\varphi_{k,\varepsilon}:=1$ on $S_1$.
In particular, $\ie{\varphi} \in L^1(\Omega)$ for all $\eps>0$, and it holds that
\begin{align*}
	\norm{\varphi_{k,\varepsilon}-\varphi_k}_{L^1(\Omega)}=
	\norm{\varphi_{k,\varepsilon}-\varphi_k}_{L^1(\tOmega)}
	=\mathcal{O}(\varepsilon).
\end{align*}
It is worth mentioning that the constant hiding in $\mathcal{O}(\eps)$ might strongly depend on $k$.
Now, Theorem~\ref{Thm:Sconv} implies that for $k\in\N$ and for each $m=1,\dots,l$ we have
\begin{align}
	\label{CONV:LAM}
	\lambda_{m}^{\varphi_{k,\varepsilon}}
	\to \lambda_{m}^{0,\varphi_k}\quad\text{for }\eps\to 0,
\end{align}
along a non-relabeled subsequence,
where $\lambda_{m}^{\varphi_{k,\varepsilon}}$ and $\lambda_{m}^{0,\varphi_k}$ denote the $m$-th eigenvalues of the diffuse interface problem \eqref{WEstateD} and the limit problem \eqref{EW0}, respectively. Recalling that $\Psi$ is continuous, we use \eqref{PerSup} and \eqref{CONV:LAM} to conclude that
\begin{align*}
	\limsupse
	I^{\varepsilon}(\varphi_{k,\varepsilon})\le
	I^0(\varphi_k).
\end{align*}
By \textit{Step 1}, we also know from the properties \eqref{ApproxE} and Assumption~\ref{AsP} that
\begin{align*}
	\underset{k\to\infty}{\limsup\,}I^0(\varphi_k)\le I^0(\varphi).
\end{align*}
Therefore, by a diagonal sequence argument, we find a zero sequence $(\eps_k)_{k\in\N}$ such that
\begin{align*}
		\underset{k\to\infty}{\limsup\,}I^{\eps_k}(\varphi_{k,\eps_k})\le I^0(\varphi).
\end{align*}	
This proves \textit{Step~2}.

\textit{%
	Step 3:
	Let $\varphi\in L^1(\Omega)$ be arbitrary. For any sequence $\left(\ie{\varphi}\right)_{\varepsilon>0}\subset L^1(\Omega)$ with
	\begin{align*}
		\limse \norm{\ie{\varphi}-\varphi}_{L^1(\Omega)}=0,
	\end{align*}
	it holds that
	\begin{align*}
		I^0(\varphi)\le\liminfse I^{\varepsilon}(\ie{\varphi}).
	\end{align*}
} 

This is also shown in the setting of \cite{BogoselOudet} using the compactness of the $\gamma$-convergence (see e.g., \cite[Prop.~4.3.7]{Bucur}). In this paper, we provide an alternative proof which does not rely on $\gamma$-convergence but directly uses the Courant--Fischer characterization of eigenvalues.

Without loss of generality, we may assume
\begin{align}\label{JeBound*}
	\liminfse \ee{I}(\ie{\varphi})<+\infty.
\end{align}
Moreover, after extracting a suitable subsequence, we have
\begin{align}\label{JeBound}
	\limse I^{\varepsilon}(\varphi_{\varepsilon})=\liminfse \ee{I}(\ie{\varphi})<+\infty.
\end{align}
Applying \cite[Prop.~3.8]{BloweyElliott}, we conclude again that $\varphi\in \Lambda^0_{\text{ad}}$ and that
\begin{align}\label{GinzSub}
	\gamma c_0P_{\tOmega}(\tilde{E}^{\varphi})\le
	\liminfse
	\ito{
		\left(
		\frac{\gamma\varepsilon}{2}
		\abs{\nabla \varphi_{\varepsilon}}^2+
		\frac{\gamma}{\varepsilon}\psi(\varphi_{\varepsilon})
		\right)
	}.
\end{align}
Therefore, for $n=i_j\in\mathbb{N}$ with $j=1,\dots,l$ recall the Courant--Fischer characterization of the diffuse interface problem \eqref{CFD} for $\varepsilon>0$, that is
\begin{align}
	\label{CFD*}
	\lambda_n^{\varepsilon,\varphi}
	=
	\underset{W\in\mathcal{S}_{n-1}}{\max}\min
	\left\{
	\left.
	\frac{
		\int_{\Omega}
		\abs{\nabla v}^2
		\text{\,d}x
		+\int_{\Omega}
		b_{\varepsilon}(\varphi)\abs{v}^2
		\text{\,d}x
	}
	{
		\int_{\Omega}\abs{v}^2\text{\,d}x
	}
	\right|
	\begin{aligned}
		&v\in \Hz\cap W^{\perp}, \\
		&v\neq 0
	\end{aligned}
	\right\},
\end{align}
and the Courant--Fischer characterization for the sharp interface problem \eqref{CFD0}, that is
\begin{align}
	\label{CFD0*}
	\lambda_n^{0,\varphi}
	=
	\underset{W\in\mathcal{S}_{n-1}}{\max}\min
	\left\{
	\left.
	\frac{
		\int_{\Omega}
		\abs{\nabla v}^2
		\text{\,d}x
	}
	{
		\int_{\Omega}\abs{v}^2\text{\,d}x
	}
	\right|
	\begin{aligned}
		&v\in V^{\varphi}\cap W^{\perp}, \\
		&v\neq 0
	\end{aligned}
	\right\},
\end{align} 
where in both cases the maximum is taken over all $(n-1)$-dimensional subspaces of $L^2(\Omega)$.

Now, our goal is to show that
\begin{align}
	\label{TODO}
	\lambda_n^{0,\varphi}\le
	\liminfse\lambda_n^{\varepsilon,\varphi_{\eps}} \in \mathbb{R}_{>0},
\end{align}
since then Lemma~\ref{Psiliminf} implies that
\begin{align*}
	\Psi(\lambda_{i_1}^{0,\varphi},\dots,\lambda_{i_l}^{0,\varphi})
	\le
	\liminfse\;
	\Psi(
	\lambda_{i_1}^{\varepsilon_k,\varphi_{\eps}},
	\dots,\lambda_{i_l}^{\varepsilon_k,\varphi_{\eps}}
	),
\end{align*}
and along with \eqref{JeBound} and \eqref{GinzSub}, this proves the assertion of \textit{Step~3}.

First of all, Theorem~\ref{LimEV} yields that the maximum in \eqref{CFD0*} is attained in the space $W\coloneqq\langle w_1^{0,\varphi},\dots, w_{n-1}^{0,\varphi}\rangle_{\text{span}}\subset L^2(\Omega)$, where $w_1^{0,\varphi},\dots,w_{n-1}^{0,\varphi}\in V^{\varphi}$ are the first $n-1$ eigenfunctions of the limit problem \eqref{EW0}. Hence, we can reformulate the Courant--Fischer characterization as
\begin{align}\label{lam0Min}
	\lambda_n^{0,\varphi}
	=
	\min
	\left\{
	\left.
	\int_{\Omega}
	\abs{\nabla v}^2
	\text{\,d}x
	\right|
	\begin{aligned}
		&v\in V^{\varphi}\cap W^{\perp}, \\
		&\norm{v}_{L^2(\Omega)}=1
	\end{aligned}
	\right\}.
\end{align}
Since $W\subset L^2(\Omega)$ is a $(n-1)$-dimensional subspace, we infer from \eqref{CFD*} that
\begin{align}\label{lamBound}
	\lambda_n^{\varepsilon,\ie{\varphi}}
	\ge
	\min
	\left\{
	\left.
	\int_{\Omega}
	\abs{\nabla v}^2
	\text{\,d}x
	+\int_{\Omega}
	b_{\varepsilon}(\ie{\varphi})\abs{v}^2
	\text{\,d}x
	\right|
	\begin{aligned}
		&v\in \Hz \cap W^{\perp}, \\
		&\norm{v}_{L^2(\Omega)}=1
	\end{aligned}
	\right\}.
\end{align}
By means of the direct method in the calculus of variations it is straightforward to show that for any fixed $\varepsilon>0$, there exists a $L^2(\Omega)$-normalized function $\ie{v}\in \Hz\cap W^{\perp}$ at which the minimum in \eqref{lamBound} is attained.

Next, from \eqref{JeBound} we deduce that the sequence $\big(\Psi(\lambda_{i_1}^{\varepsilon,\ie{\varphi}},\dots,\lambda_{i_l}^{\varepsilon,\ie{\varphi}})\big)_{\eps>0}$ is bounded. Hence, using condition \eqref{Psibound} from Assumption~\ref{AsP}, we conclude that the sequence $(\lambda_n^{\varepsilon,\ie{\varphi}})_{\eps>0}$ is also bounded and in particular the limes inferior of this sequence exists.

Now, using \eqref{lamBound}, we infer that $(\norm{\nabla \ie{v}}_{L^2(\Omega)})_{\eps>0}$ is bounded and hence, we find a function $\overline{v}\in \Hz$ such that the convergences
\begin{alignat}{3}\label{vMinConv}
	\begin{aligned}
		\ie{v}&\rightharpoonup \overline{v}\quad\text{in }\Hz,
		\qquad
		\ie{v}\to \overline{v}\quad\text{in }L^2(\Omega),
		\qquad
		\ie{v}\to \overline{v}\quad\text{a.e.~in }\Omega
	\end{aligned}
\end{alignat}
hold along a non-relabeled subsequence.
Moreover, we have
\begin{align*}
	\ie{b}(\ie{\varphi}(x))\to b_0(\varphi(x))
	\quad\text{as $\eps\to 0$ for almost all $x\in\Omega$,}
\end{align*}
up to subsequence extraction.
This convergence was shown in \textit{Step~1} of the proof of Lemma~\ref{LEM:lam1Conv} and its proof did not require any rate assumption on $\left(\ie{\varphi}\right)_{\varepsilon>0}$. We thus obtain
\begin{align*}
	\liminfse \lambda_n^{\varepsilon,\ie{\varphi}}
	&\ge
	\liminfse
	\left[
	\io{\abs{\nabla \ie{v}}^2}+\io{\ie{b}(\ie{\varphi})\abs{\ie{v}}^2}
	\right]\\
	&\ge
	\io{\abs{\nabla \overline{v}}^2}+\io{b_0(\varphi)\abs{\overline{v}}^2},
\end{align*}
by applying Fatou's lemma on the $\ie{b}$ term, and employing the weak lower semi-continuity of $\norm{\nabla\,\cdot\,}_{L^2(\Omega)}$.
In particular, this implies that 
\begin{align*}
	\io{b_0(\varphi)\abs{\overline{v}}^2}<+\infty.
\end{align*}
Hence, recalling that $\varphi\in\Phi_\mathrm{ad}^0$, we conclude that $\overline{v}\in V^{\varphi}$ which in turn implies
\begin{align*}
	\io{b_0(\varphi)\abs{\overline{v}}^2}=0.
\end{align*}
Now, by construction, we have $\norm{\ie{v}}_{L^2(\Omega)}=1$ and $\ie{v}\in \Hz\cap W^{\perp}$. Due to the convergences in \eqref{vMinConv} the same holds for the limit $\overline{v}\in V^{\varphi}$. This means that $\overline{v}$ belongs to the set appearing in \eqref{lam0Min}, and we finally deduce
\begin{align*}
	\liminfse \lambda_n^{\varepsilon,\ie{\varphi}}
	\ge
	\io{\abs{\nabla \overline{v}}^2}
	\ge
	\min
	\left\{
	\left.
	\int_{\Omega}
	\abs{\nabla v}^2
	\text{\,d}x
	\right|
	\begin{aligned}
		&v\in V^{\varphi}\cap W^{\perp}, \\
		&\norm{v}_{L^2(\Omega)}=1
	\end{aligned}
	\right\}
	=\lambda_n^{0,\varphi}>0,    
\end{align*}
which proves \eqref{TODO}. This means that \textit{Step~2} is established and thus, the proof of Theorem~\ref{GamHelp} is complete.
\end{proof}

Now, using Theorem~\ref{GamHelp}, we can finally prove the desired $\Gamma$-convergence result $\ee{J}\overset{\Gamma}{\to}J^0$ where the volume constraint is incorporated.

\begin{proof}[Proof of Theorem~\ref{GamCost}]
	Due to the inclusion $\Phi_{\text{ad}}\subset \Lambda_{\text{ad}}$ we have $I^{\eps}(\varphi)\le J^{\eps}(\varphi)$ for all $\varphi\in L^1(\Omega)$. Furthermore for a sequence $(\varphi_{\eps})_{\eps>0}\subset \Phi_{\text{ad}}$ with $\varphi_{\eps}\to\varphi$ in $L^1(\Omega)$ we deduce $\varphi\in \tilde{\mathcal{G}}^{\beta}$ and therefore $I^0(\varphi)=J^0(\varphi)$. Hence, the $\liminf$ inequality for $J^{\eps}$ directly follows from Theorem~\ref{GamHelp}.
	
	It remains to prove that for any $\varphi\in \Phi_{\text{ad}}^0$, there exists a recovery sequence $(\tve)_{\eps>0}\subset \Phi_{\text{ad}}$ such that
	\begin{align}
		\limse\norm{\tve-\varphi}_{L^1(\Omega)}&=0,\label{L1Phi}\\
		\limsupse J^{\eps}(\tve)&\le J^0(\varphi).\label{LimSupF}
	\end{align}
	Here our strategy is now to use the recovery sequence from Theorem~\ref{GamHelp}. In the following, it will be denoted by $(\varphi_{\eps})_{\eps>0}\subset \Lambda_{\text{ad}}$. For any $\eps>0$, we now carefully modify the function $\varphi_\eps$ via a diffeomorphism in order to ensure that it additionally fulfills the volume constraint comprised in $\tilde{\mathcal{G}}^{\beta}$. In the following, we will always understand the functions $\varphi_{\eps}, \varphi\in L^1(\R^d)$ as being trivially extended onto $\R^d$, i.e., these functions are constant zero on $\R^d\backslash\Omega$.
	
	The key idea is to construct for any $\eps>0$ a suitable transformation $T_{s(\eps)}:\R^d \to \R^d$ with $T_{s(\eps)}(\tilde\Omega)=\tilde\Omega$ such that the modified functions 
	$\tve := \varphi_{\eps}(T_{s(\eps)}^{-1})$
	belong to $\Phi_{\text{ad}}$
	and satisfy the convergence properties \eqref{L1Phi} and \eqref{LimSupF}.
	This is a common method in geometric analysis and a similar procedure in the sharp interface case can be found for example in the proof of \cite[Thm.~19.8]{Maggi}. 
	
	We now fix an arbitrary function $\varphi\in \Phi_{\text{ad}}^{0}$.
	First of all, we find a vector field $\B{\xi}\in C_0^1(\tOmega;\R^d)$ such that
	\begin{align*}
		\int_{\tOmega}\varphi\nabla\cdot \B{\xi}\dLL>0,
	\end{align*}
	as otherwise the total variation of the associated Radon measure would vanish, i.e., $\abs{D\varphi}(\tOmega)=0$, which would imply that $\varphi$ is constant almost everywhere in $\tOmega$. However, this is not possible as neither $\varphi\equiv 1$ nor $\varphi\equiv -1$ in $\tOmega$ would fulfill the mean value constraint in $\tilde{\mathcal{G}}^{\beta}$ due to the choice of $\beta_1,\beta_2$ in \eqref{GDef}.
	
	Using the vector field $\B{\xi}$, we now define a family of transformations
	\begin{align*}
		T_s: \R^d \to \R^d,\quad
		x \mapsto x+s\B{\xi}(x),
	\end{align*}
	for $s\in\mathbb{R}.$ As this map is a perturbation of the identity via a $C^1$-map with compact support in $\tOmega$, it is clear that, for $\abs{s}$ sufficiently small, $T_s$ is a $C^1$-diffeomorphism with $T_s(\tOmega) = \tOmega$. Hence,
	\begin{align*}
		T_s\vert_{\tOmega}: \tOmega\to\tOmega
	\end{align*}
	is also a $C^1$-diffeomorphism.
	Moreover, the chain rule for Sobolev functions (see e.g., \cite[4.26]{Alt}) implies
	\begin{align*}
		\varphi_{\eps}\circ T_s^{-1} \in H^1(\tOmega).
	\end{align*}
	
	Since $\varphi_{\eps}\in \mathcal{U}$
	and $T_s\vert_{\mathbb{R}^d\backslash \tOmega}=\mathrm{id}_{\mathbb{R}^d\backslash \tOmega}$, 
	we infer
	\begin{align*}
		\int_{\R^d}\varphi_{\eps}\circ T_s^{-1}\dLL=\int_{\tOmega}\varphi_{\eps}\circ T_s^{-1}\dLL+\abs{S_1}-\abs{S_0},
	\end{align*}
	due to the trivial extension of $\varphi_{\eps}$ on $\R^d\backslash\Omega.$
	Moreover, we use the representation
	\begin{align}\label{fCont0}
		\int_{\R^d}\varphi_{\eps}\circ T_{s}^{-1}\dLL
		=\int_{\R^d}\varphi_{\eps}\abs{\det DT_{s}}\dLL.
	\end{align}
	Recalling that the determinant is a multilinear form, a straightforward computation reveals that there is a $\delta>0$ and a constant $C>0$ depending only on $\delta$ and $\B{\xi}$
	such that for all $x\in \mathbb{R}^d$ and $s\in (-\delta,\delta)$,
	\begin{align}\label{detExp}
		\frac 12 \le  1-C\abs{s} \le \det DT_{s}(x)\le 1+C\abs{s}.
	\end{align}
	In the following, we use the convenient notation $\varphi_0:=\varphi$. We now define the function
	\begin{align*}
		f:B_{\delta}(0)\subset \R^2\to \R,
	    \quad
		(\eps,s)\mapsto
		\int_{\R^d}\varphi_{|\eps|}\circ T_s^{-1}\dLL-\int_{\tOmega}\varphi\text{\,d}x-\abs{S_1}+\abs{S_0},
	\end{align*}%
	where $0<\delta<1$ is chosen sufficiently small in order to ensure that $T_s\vert_{\tOmega}:\tOmega\to\tOmega$ is a $C^1$-diffeomorphism and \eqref{detExp} holds
	for all $s\in (-\delta,\delta)$.

	Now our next goal is to apply the implicit function theorem formulated in \cite[Thm.~4.B]{Zeidler1} to the equation $f(\eps,s)=0$.
	First of all, $f(0,0)=0$ is clear since $\varphi\in \Phi_\text{ad}^0$.
	We next prove that $f$ is continuous at $(0,0)$.
	To this end, let us choose zero sequences $\left(\eps_k\right)_{k\in\mathbb{N}},\left(s_k\right)_{k\in\mathbb{N}}$.
	Due to the symmetry of $f$ with respect to its first argument, we may assume without loss of generality that $\eps_k\geq 0$.
	As per construction, we find a non-relabeled subsequence $(\varphi_{\eps_k})_{k\in\N}$ which converges to $\varphi$ almost everywhere as $k\to \infty$, and since $\norm{\varphi_{\eps_k}}_{L^{\infty}(\mathbb{R}^n)}\le 1$, we apply Lebesgue's dominated convergence theorem to the right-hand side of \eqref{fCont0}. Along with \eqref{detExp}, we deduce 
	\begin{align*}
		f(\eps_k,s_k)\to f(0,0)\quad\text{for }k\to \infty.
	\end{align*}
	As this limit does not depend on the choice of the subsequence, this argument can be repeated for any subsequence, thus, continuity of $f$ at $(0,0)$ is shown.
	
	In order to apply the implicit function theorem it remains to show that $\frac{\partial}{\partial s}f$ exists on $B_{\delta}(0)$, is continuous at $(0,0)$ and does not vanish at $(0,0)$.
	First of all, due to \eqref{detExp}, the modulus in \eqref{fCont0} can be omitted. Furthermore, the proof of \cite[Lem.~1]{Schmidt} implies that for fixed $x\in \R^d$ and $t\in (-\delta,\delta)$ we have
	\begin{align}
	    \label{EQ:DSDET}
		\frac{\mathrm{d}}{\mathrm{d}s}\left[\det DT_s\right]_{\vert s=t}&=
		\text{tr}\left(
		\frac{\mathrm{d}}{\mathrm{d}s}\left[DT_s\right]_{\vert s=t}(DT_t)^{-1}
		\right)
		\det DT_t
		\notag\\
		&=
		\text{tr}\left(\nabla\B{\xi}(\mathrm{Id}+t\nabla\B{\xi})^{-1}\right)
		\det(\mathrm{Id}+t\nabla\B{\xi})
		\notag\\
		&=
		\text{tr}\left(
		\nabla\B{\xi}\sum_{k=0}^{\infty}(-t\nabla\B{\xi})^k
		\right)
		\det(\mathrm{Id}+t\nabla\B{\xi}).        
	\end{align}
	As $\B{\xi}\in C_0^1(\tOmega,\R^d)$, we directly see that
	\begin{align*}
	    \underset{t\in (-\delta,\delta)}{\sup}\;
		\norm{\frac{\mathrm{d}}{\mathrm{d}s}\left[\det DT_s\right]_{\vert s=t}}_{C^0(\tOmega)} <\infty.
	\end{align*}
	Noticing again that $\norm{\varphi_\eps}_{L^{\infty}(\Omega)}\le 1$ for all $\eps\ge0$, we deduce via Lebesgue's dominated convergence theorem that for any $(\eps,t)\in B_{\delta}(0)$, we have
	\begin{align*}
		\frac{\partial f}{\partial s}(\eps,t)=
		\int_{\R^n}\varphi_{ \eps}\,
		\text{tr}\left(
		\nabla\B{\xi}\sum_{k=0}^{\infty}(-t\nabla\B{\xi})^k
		\right)
		\det(\mathrm{Id}+t\nabla\B{\xi})\dLL.
	\end{align*}
	Therefore we directly infer
	\begin{align*}
		\frac{\partial f}{\partial s}(0,0)=
		\int_{\R^n}\varphi_0 \nabla\cdot\B{\xi}\dLL=
		\int_{\tOmega}\varphi\nabla\cdot\B{\xi}\dLL>0,
	\end{align*}
	by our choice of $\B{\xi}\in C_0^1(\tOmega,\R^n)$ at the beginning of this proof.
	Now, the continuity of $\frac{\partial f}{\partial s}$ at $(0,0)$ follows again via Lebesgue's dominated convergence theorem using that for any $x\in \mathbb{R}^d$,
	\begin{align*}
		\text{tr}\left(
		\nabla\B{\xi}\sum_{k=0}^{\infty}(-t\nabla\B{\xi})^k
		\right)
		\det(\mathrm{Id}+t\nabla\B{\xi}).    
	\end{align*}
	is continuous at $t=0$.
	
	Now that we have checked all the assumptions of the implicit function theorem \cite[Thm.~4.B]{Zeidler1}, we deduce the existence of a $\tilde{\delta}>0$ and a function
	\begin{align*}
		s: (-\tilde{\delta},\tilde{\delta})\to (-\delta,\delta),
	\end{align*}
	which is continuous at $0$,
	such that
	\begin{align*}
		f(\eps,s(\eps))&=0\quad \text{for all }
		\eps\in (-\tilde{\delta},\tilde{\delta})\text{ and}\\
		s(0)&=0.
	\end{align*}
	In our framework, this means that having started with the recovery sequence $({\varphi}_{\eps})_{\eps>0}$ of Theorem~\ref{GamHelp}, we now know
	\begin{align*}
		\tve\coloneqq
		{\varphi}_{\eps}\circ T_{s(\eps)}^{-1}\in \Phi_{\text{ad}},
	\end{align*}
	i.e., we have constructed an admissible sequence. Note that the pointwise constraints $\tilde{\varphi}_{\eps}=1$ a.e.~in $S_1$ and $\tilde{\varphi}_{\eps}=-1$ a.e.~in $S_0$ are fulfilled since $T_s\vert_{\mathbb{R}^d\backslash \tOmega}=\mathrm{id}_{\mathbb{R}^d\backslash \tOmega}$.
	Hence, it remains to show
	\begin{align*}
		\limse\norm{\tve-\varphi}_{L^1(\Omega)}&=0,\\
		\limsupse J^{\eps}(\tve)&\le J^0(\varphi).
	\end{align*}
	The $L^1$ convergence follows from the triangle inequality
	\begin{align*}
		\norm{\varphi_{\eps}\circ T_{s(\eps)}^{-1}-\varphi}_{L^1(\Omega)}\le 
		\norm{(\varphi_{\eps}-\varphi)\circ T_{s(\eps)}^{-1}}_{L^1(\Omega)}+
		\norm{\varphi\circ T_{s(\eps)}^{-1}-\varphi}_{L^1(\Omega)}.
	\end{align*}
	Here, the convergence of the first summand can be shown by the same argumentation as for the continuity of $f$ at $0$, whereas the convergence of the second summand can be established via Lebesgue's dominated convergence theorem after approximating $\varphi_0$ by a sequence of $C_0^0(\Omega)$ functions.
	
	To verify the $\limsup$ inequality, let us first consider the Ginzburg--Landau energy separately.
	For the potential term we compute with the help of \eqref{detExp}
	\begin{align*}
		\int_{\tOmega}\psi(\tve)\dLL=
		\int_{\tOmega}\psi(\varphi_{\eps})\abs{\det DT_{s(\eps)}}\dLL\le
		(1+C\abs{s(\eps)})\int_{\tOmega}\psi(\varphi_{\eps})\dLL.
	\end{align*}
	Using the fact that for every $x\in \R^d$,
	\begin{align*}
	    \Big(D\big(T_s^{-1}\big)\Big)(x) = \Big(DT_s\big(T_s^{-1}(x)\big)\Big)^{-1},
	\end{align*}	
	we infer that the gradient term can be expressed as
	\begin{alignat*}{2}
    \begin{aligned}
		\int_{\tOmega}\abs{\nabla{\tve}(x)}^2\dLL
		&=
		\int_{\tOmega}\abs{ \Big( DT_{s(\eps)}\big(T_{s(\eps)}^{-1}(x)\big) \Big)^{-T} \nabla\varphi_\eps\Big(T_{s(\eps)}^{-1}(x)\Big)}^2\dLL\\
		&=
		\int_{\tOmega}\abs{\big(\mathrm{Id}+s(\eps)\nabla\B{\xi}(x)\big)^{-T}\nabla\varphi_{\eps}(x)}^2\abs{\det DT_{s(\eps)}(x)}\dLL\\
		&=
		\int_{\tOmega}\abs{\left(\sum_{k=0}^{\infty}(-s(\eps)\nabla\B{\xi})^k\right)^{T}\nabla\varphi_{\eps}(x)}^2\abs{\det DT_{s(\eps)}(x)}\dLL\\
		&\le
		\big(1+C\abs{s(\eps)}\big)\int_{\tOmega}\abs{\nabla\varphi_{\eps}(x)}^2\dLL,
    \end{aligned}    
	\end{alignat*}
	where we use \eqref{detExp} and a straightforward computation involving the geometrical series. Therefore, the constant $C$ only depends on $\tilde{\delta}$ and $\B{\xi}$.
	Altogether, we deduce
	\begin{align}
    \label{supPer}
		\limsupse \int_{\tOmega} \frac{\eps}{2}\abs{\nabla \tve}^2+\frac{1}{\eps}\psi(\tve)\dLL
		&\le
		\limse (1+C\abs{s(\eps)})\,
		\limsupse \int_{\tOmega} \frac{\eps}{2}\abs{\nabla \varphi_{\eps}}^2+\frac{1}{\eps}\psi(\varphi_{\eps})\dLL
		\notag\\
		&\le c_0\, P_{\tOmega}(\tilde{E}^\varphi),
	\end{align}
	as $(\varphi_{\eps})_{\eps>0}$ was the recovery sequence for $\varphi$ of Theorem \ref{GamHelp}.
	
    Now, we consider the eigenvalue term of the cost functional $J^{\eps}$. Due to Theorem \ref{GamHelp}, we already know
    \begin{align}\label{supEV}
       \limsupse\lambda_{i_j}^{\eps,\ie{\varphi}}\le \lambda_{i_j}^{0,\varphi},
    \end{align}
    for $j=1,\dots,l$. In the following, to provide a cleaner presentation, we will write $k\coloneqq i_j$.
        We intend to show that there exists a sequence $(\alpha_{s(\eps)})_{\eps>0}$ (that may depend on $k$) with
        \begin{align*}
            \limse \alpha_{s(\eps)}=1,
        \end{align*}
        such that for all $\eps>0$ small enough
        \begin{align}\label{LambdaTve}
            \alpha_{s(\eps)}\lambda_k^{\tve}\le \lambda_k^{\varphi_{\eps}}.
        \end{align}
        Using \eqref{supEV} this then directly gives
    \begin{align}\label{supEV2}
       \limsupse\lambda_{k}^{\eps,\ie{\tilde{\varphi}}}\le \lambda_{k}^{0,\varphi}.
    \end{align}      
        So for $\eps>0$ let us consider the Courant--Fischer characterization of $\lambda_k^{\ie{\varphi}}$ which due to Theorem~\ref{EEW} reads as
      		\begin{align}\label{CFkeps} 
                \lambda_k^{\varphi_{\eps}}
            =
                \underset{V\in\mathcal{S}_{k-1}}{\max}\min
                    \left\{
                        \left.
                            \frac{
                                \int_{\Omega}
                                    \abs{\nabla v}^2
                                \text{\,d}x
                              +\int_{\Omega}
                                    \ie{b}({\varphi}_{\eps})\abs{v}^2
                                 \text{\,d}x
                             }
                             {
                                \int_{\Omega}\abs{v}^2\text{\,d}x
                             }
                      \right|
                        \begin{aligned}
                              	&v\in V^{\perp,L^2(\Omega)}\cap \Hz, \\
                            &v\neq 0
                        \end{aligned}
                \right\}.
        \end{align}
  Now let us choose the subspace
  \begin{align*}
    V_{T_{\eps}}\coloneqq
    \left\{
        \left(w_1^{\tilde{\varphi}_{\eps}}\circ T_{s(\eps)}\right)\abs{\det{T_{s(\eps)}}},\dots,
        \left(w_{k-1}^{\tilde{\varphi}_{\eps}}\circ T_{s(\eps)}\right)\abs{\det{T_{s(\eps)}}}        
    \right\}\subset L^2(\Omega).
  \end{align*}
  As the family of eigenfunctions
  \begin{align*}
    W_{\eps}\coloneqq\left\{w_1^{\tilde{\varphi}_{\eps}},\dots, w_{k-1}^{\tilde{\varphi}_{\eps}} \right\}\subset L^2(\Omega)
  \end{align*}
  is linearly independent, $V_{T_{\eps}}\subset L^2(\Omega)$ is indeed a $(k-1)$-dimensional subspace.
  Hence, we infer from \eqref{CFkeps} that
  \begin{align*}
    \lambda_k^{\varphi_{\eps}}\ge
                \min
                    \left\{
                        \left.
                            \frac{
                                \int_{\Omega}
                                    \abs{\nabla v}^2
                                \text{\,d}x
                              +\int_{\Omega}
                                    \ie{b}({\varphi}_{\eps})\abs{v}^2
                                 \text{\,d}x
                             }
                             {
                                \int_{\Omega}\abs{v}^2\text{\,d}x
                             }
                      \right|
                        \begin{aligned}
                              	&v\in V_{\ie{T}}^{\perp,L^2(\Omega)}\cap \Hz, \\
                            &v\neq 0
                        \end{aligned}
                \right\}.    
  \end{align*}
  Now we want to show that
  \begin{alignat}{2}
  \label{minTrafo}
       &\min
                    \left\{
                        \left.
                            \frac{
                                \int_{\Omega}
                                    \abs{\nabla v}^2
                                \text{\,d}x
                              +\int_{\Omega}
                                    \ie{b}({\varphi}_{\eps})\abs{v}^2
                                 \text{\,d}x
                             }
                             {
                                \int_{\Omega}\abs{v}^2\text{\,d}x
                             }
                      \right|
                        \begin{aligned}
                              	&v\in V_{\ie{T}}^{\perp,L^2(\Omega)}\cap \Hz, \\
                            &v\neq 0
                        \end{aligned}
                \right\}
    \notag\\
  \ge
    &\min
        \left\{
            \left.
                \frac{
                    \int_{\Omega}
                        \abs{\nabla \left(v\circ T_{s(\eps)}\right)}^2
                    \text{\,d}x
                  +\int_{\Omega}
                        \ie{b}({\varphi}_{\eps})\abs{v\circ T_{s(\eps)}}^2
                     \text{\,d}x
                 }
                 {
                    \int_{\Omega}\abs{v\circ T_{s(\eps)}}^2\text{\,d}x
                 }
          \right|
            \begin{aligned}
                           	&v\in W_{\eps}^{\perp,L^2(\Omega)}\cap \Hz, \\
                &v\neq 0
            \end{aligned}
    \right\}.
  \end{alignat}
  To verify this, denote with $0\neq\overline{v}\in V_{\ie{T}}^{\perp,L^2(\Omega)}\cap \Hz$ a function at which the minimum in the first line is attained. Then, by the transformation formula it holds for $m=1,\dots,k-1$
  \begin{align*}
    0=\int_{\Omega}
        \overline{v}\left(w_m^{\tve}\circ T_{s(\eps)}\right)\abs{\det T_{s(\eps)}}
    \dLL
    =
    \int_{\Omega}
        \left(\overline{v}\circ T_{s(\eps)}^{-1}\right)w_m^{\tve}
    \dLL.  
  \end{align*}
  As we additionally know $T_{s(\eps)}(\partial\Omega)= \partial\Omega$, the function
  \begin{align*}
    0\neq\left(\overline{v}\circ T_{s(\eps)}^{-1}\right)\in W_{\eps}^{\perp,L^2(\Omega)}\cap H_0^1(\Omega)
  \end{align*}
  is admissible and, per construction, \eqref{minTrafo} holds. For any arbitrary $v\in H_0^1(\Omega)$, we find that
  \begin{align*}
        \int_{\Omega}
            \abs{\nabla \left(v\circ T_{s(\eps)}\right)(x)}^2
        \text{\,d}x
        &=\int_{\Omega}
            \abs{\left(\mathrm{Id}+s(\eps)\nabla\xi(x)\right)^{T} \,\nabla v\left(T_{s(\eps)}(x)\right)}^2\text{\,d}x\\
       &\ge \left(1-C\abs{s(\eps)}\right)
       \int_{\Omega}\abs{\nabla v\left(T_{s(\eps)}(x)\right)}^2\dLL,      
  \end{align*}
  with a constant $C>0$ depending only on $\tilde{\delta}$ and $\B{\xi}$.
 Hence, invoking the transformation theorem and using \eqref{detExp}, we conclude for the remaining terms in \eqref{minTrafo} that
  \begin{align*}
    &\min
        \left\{
            \left.
                \frac{
                    \int_{\Omega}
                        \abs{\nabla \left(v\circ T_{s(\eps)}\right)}^2
                    \text{\,d}x
                  +\int_{\Omega}
                        \ie{b}({\varphi}_{\eps})\abs{v\circ T_{s(\eps)}}^2
                     \text{\,d}x
                 }
                 {
                    \int_{\Omega}\abs{v\circ T_{s(\eps)}}^2\text{\,d}x
                 }
          \right|
            \begin{aligned}
                           	&v\in W_{\eps}^{\perp,L^2(\Omega)}\cap \Hz, \\
                &v\neq 0
            \end{aligned}
    \right\}\\
    \ge
    &\left(1-C\abs{s(\eps)}\right)\min
        \left\{
            \left.
                \frac{
                    \int_{\Omega}
                        \abs{\nabla v}^2
                    \text{\,d}x
                  +\int_{\Omega}
                        \ie{b}(\tve)\abs{v}^2
                     \text{\,d}x
                 }
                 {
                    \int_{\Omega}\abs{v}^2\text{\,d}x
                 }
          \right|
            \begin{aligned}
                           	&v\in W_{\eps}^{\perp,L^2(\Omega)}\cap \Hz, \\
                &v\neq 0
            \end{aligned}
    \right\}.     
  \end{align*}
  Combining all the previous computations and recalling that $W_{\eps}$ is the space spanned by the first $k-1$ eigenfunctions corresponding to $\tve=\varphi_{\eps}\circ T^{-1}_{s(\eps)}$, we use Theorem~\ref{EEW}
  to conclude \eqref{LambdaTve} with $\alpha(\eps) := (1-C\abs{s(\eps)})$, and we eventually arrive at \eqref{supEV2}.
  
  Finally, we use the monotonicity of $\Psi$ from Assumption~\ref{AsP} along with \eqref{supPer} to deduce
  \begin{align*}
    \limsupse J^{\eps}(\tve)\le J^0(\varphi).
  \end{align*}
This completes the proof.
\end{proof}	

\begin{Rem}
    Assume that the sets $S_0$ and $S_1$ are compactly contained in the design domain $\Omega$ (i.e., $\overline{S_0},\overline{S_1}\subset \Omega$).
    Then, the minimization of $J^{\eps}$ implicitly enforces the Neumann boundary condition $\frac{\partial \ie{\varphi}}{\partial \B{\nu}}=0$ on $\partial\Omega$. As discussed in Section~\ref{Sec:Form}, we could alternatively impose the Dirichlet condition $\ie{\varphi}=-1$ on $\partial \Omega$ which does not change the line of reasoning of the current section but produces a different limiting cost functional for $\eps\to 0$ in which a term penalizing the energy of transitions from $\varphi_0=1$ to $\varphi_0=-1$ when approaching $\partial\Omega$ needs to be added; see \cite{Owen} for a detailed discussion. In Figure~\ref{fig:num:tight}, we show a numerical example where the Dirichlet condition $\ie{\varphi}=-1$ on $\partial \Omega$ is explicitly imposed in order for the boundary of $\Omega$ to act as an obstacle.
\end{Rem}


\section{Numerical Computations}
 \label{sec:num}

 In this section, we validate our approach by presenting several numerical examples.
 In Section~\ref{sec:num:FEVMPT}, we describe the methods we use for the numerical approximation of 
 a solution to \eqref{PD} or \eqref{PN}, respectively.
 In Section~\ref{sec:num:find-abc}, we study the solutions to some standard examples with known analytical solution to fix the parameters $b_\eps$ and $c_\eps$ in our approach.
Thereafter, in Section~\ref{sec:num:noAnalyticalSolution}, we show further capabilities of our proposed method by solving problems whose solution is analytically unknown.

 \subsection{The numerical realization}
 \label{sec:num:FEVMPT}
 To discretize \eqref{PD} and \eqref{PN}, we use standard piecewise linear and globally continuous finite elements, provided by
 the finite element toolbox \texttt{FEniCS} \cite{fenics1,fenics_book}, for all appearing functions
 to obtain finite dimensional approximations $\varphi^h$ of $\varphi$ and $w_i^h$ of $w_i$, where $i$ corresponds to the index of the eigenvalue. 
 The finite dimensional variants of the state equations \eqref{statewD} and \eqref{statewN} are solved by the eigenvalue solver \texttt{SLEPc} \cite{slepc:2005} provided by \texttt{PETSc} \cite{petsc-user-ref,petsc-efficient} to obtain 
 approximate eigenvalues $\lambda_i^h$ and $\mu_i^h$.
 The optimization problems \eqref{PD} and \eqref{PN} are treated by the \texttt{VMPT} method, see \cite{BlankRupprecht}. 
 This method can be understood as an extension of the projected gradient method to non-reflexive Banach spaces. 
 In our setting, we consider  $\varphi \in H^1(\Omega) \cap L^\infty(\Omega)$.
 As part of this method we need to solve projection-type subproblems, that have the form of linear-quadratic optimization problems. These are solved using the package \texttt{IPOPT}, see \cite{WachterIPOPT}.
 
 Since the phase field $\varphi^h$ changes its value between $-1$ and $1$ over a length-scale of size $\eps$, a very high resolution of 
 certain parts of the computational domain is required. Here, we use locally refined meshes. 
 For given $\eps>0$, we fix the mesh parameter $h_{\min}$ as $h_{\min} = s\eps$, with $s \approx 0.08$.  
 This leads to about 12 cells over a length of $\eps$.
 We start the optimization with a very coarse mesh and use the \texttt{VMPT} 
 method until convergence to a numerical solution $\varphi^h$ occurs. 
 Thereafter, we adapt the mesh and 
 refine all cells $K$ with diameter larger than $h_{\min}$ which satisfy $|\varphi^h (K_m)| \leq 0.99$ 
 (where $K_m$ denotes the midpoint of $K$), 
 whereas cells that satisfy $|\varphi^h (K_m)| \geq 1.0$ are coarsened. 
 We then optimize again on the the new mesh. 
 This loop is executed until no refinement is performed during the adaptation step.
 Alternative concepts for local error estimation might be used, e.g., residual based error estimation or dual weighted residuals,
 but we stress that, in any case, a high resolution of the interface $|\varphi^h| \leq 0.99$ is required for successful numerical calculations.
 
 We point out that for small values of $\eps$ and $\gamma$, the interfaces tend to become very thin 
 and thus, starting with a coarse mesh is numerically not feasible. 
 In this situation a homotopy starting from larger values for $\gamma$ is used. 
 We choose $\gamma$ as homotopy parameter because $\gamma$ can be varied over larger scales than $\varepsilon$.
 
 \subsection{Fixing model parameters}
\label{sec:num:find-abc}
  The considered optimization problems \eqref{PD} and \eqref{PN}
  involve several parameters that need to be chosen. 
  Here, we fix some of them, and we will stick to this choice unless stated differently.
  We fix $\Omega = (0,1)^d$, $d\in\{2,3\}$,  $\int_\Omega\frac{\varphi+1}{2} = ({1}/{2})^d|\Omega|$, i.e., $\beta_1 = \beta_2 = -0.5$
  if $d=2$ and   $\beta_1 = \beta_2 = -0.75$ if $d=3$, and start from a constant initial value $\varphi^0$. 
  Moreover, we use $\eps=0.02,\gamma=1$ and $\psi_0(\varphi) = \frac{1}{2}(1-\varphi^2)$ as the regular part of the potential $\psi$.
  
  The phase field approximations \eqref{statewD} and \eqref{statewN} involve three model functions, 
  namely $a_\eps(\varphi)$, $b_\eps(\varphi)$, and $c_\eps(\varphi)$.
  Here we make the following settings 
  \begin{itemize}
  \item We fix $a_\eps(\varphi) = \frac{1-\eps}{2} \varphi + \frac{1+\eps}{2}$,   meaning that $a_\eps(1) = 1$ and $a_\eps(-1) = \eps$.
  \item We fix $c_\eps(\varphi) =\frac{1-c\eps}{2} \varphi + \frac{1+c\eps}{2}$ with some $c>0$, meaning that $c_\eps(1) = 1$ and $c_\eps(-1) = c\eps$.
  In case we consider Dirichlet boundary data we fix $c=1$.
  \item We fix $b_\eps(\varphi) = b\frac{1-\varphi}{2\eps^{4/3}}$ with some $b>0$, meaning that $b_\eps(1) = 0$ and $b_\eps(-1) = \frac{b}{\eps^{4/3}}$.
  \end{itemize}
  We note that, in the following, the rate $\eps^{4/3}$ appearing in $b_\eps$ leads to a common choice for $b$ independent of $\eps$.
  In summary, in case that we apply Dirichlet boundary data, we have one unknown parameter, namely $b$, and in case that we use Neumann boundary data, we have one unknown parameter, namely $c$.
  

\begin{Rem}
In case of Dirichlet boundary data, we could potentially set $a_\eps(\varphi) \equiv c_\eps(\varphi) \equiv 1$ as stated in Assumption~\ref{Asmp:ac}.
However, in this setting, we experienced when solving the minimization problem that the shape tends to attain the form of one large ball, even in cases where, for instance, two balls are the optimal solution. 
Nevertheless, using $a_\eps(-1)=\eps$ and $c_\eps(-1)=c\eps$ as introduced above does not conflict with the assumptions made in Section~\ref{Sub:coeff}. Hence, the analytical results obtained in Section~\ref{SEC:Analysis} are valid in the current setting.

Functions like $a_\eps$ are often chosen as polynomials to mimic the SIMP approach (see, e.g., \cite{Pedersen}). 
However, during the minimization process such an
approach would lead to very thin interfaces that can barely be resolved by a finite element mesh. 
As we solve a minimization problem, the final shape of the interface is adjusted in an optimal way in terms of the chosen parameter functions.
\end{Rem}

 \paragraph{Fixing $b$ for Dirichlet boundary data.}
To fix $b$ we solve the minimization problem
related to minimizing the first eigenvalue $\lambda_1$ for the Laplace operator with Dirichlet boundary data
for sequences of $b$ and $\eps$. The analytical result is given by Theorem~\ref{thm:num:min_lam1} which is known as the Faber--Krahn theorem.

  \begin{Thm}[Faber--Krahn, cf.~{\cite[Thm.~3.2.1]{Henrot}}]
 \label{thm:num:min_lam1}
        The minimum of $\lambda_1(D)$ among all bounded open sets $D \subset \mathbb{R}^d$, $d\in \N$, 
        with given volume is achieved by one ball.
 \end{Thm}
  
 In Table~\ref{tab:num:eigval_balls}, we present the first four analytical eigenvalues
 on one ball of the given volume and, for later reference, on two balls with the given volume in sum.

  \begin{table}
     \centering

    \begin{tabular}{c|cc|cc}
         &  \multicolumn{2}{c|}{2D, $|D| = \frac{1}{4}$} & \multicolumn{2}{c}{3D, $|D| = \frac{1}{8}$} \\
         \hline
         & one ball & two balls & one ball & two balls\\
         \hline
    $r$  & $\sqrt{\frac{1}{4\pi}}$ &   $\sqrt{\frac{1}{8\pi}}$ & $\sqrt[3]{\frac{3}{32\pi}}$ & $\sqrt[3]{\frac{3}{64\pi}}$\\
    \hline  
    $\lambda_1$ & 72.67 & 145.34 & 102.59 & 162.84\\
    $\lambda_2$ & 184.50 & 145.34 & 209.86 & 162.84\\
    $\lambda_3$ & 184.50 & 369.00 & 209.86 & 333.14\\
    $\lambda_4$ & 331.43 & 369.00 & 209.86 & 333.14
    \end{tabular}

     \caption{The first eigenvalues of the Laplace operator with Dirichlet boundary condition in two dimensions on a ball
     of volume $\frac{1}{4}$ and on two balls of volume $\frac{1}{8}$ each, and in three dimensions on a ball
     of volume $\frac{1}{8}$ and two balls of volume $\frac{1}{16}$ each.
     The value $r$ denotes the radius of one ball.
     We refer to \cite[Prop.~1.2.14]{Henrot} on how to compute these values.
     }
     \label{tab:num:eigval_balls}
 \end{table}

 We solve the minimization problem related to Theorem~\ref{thm:num:min_lam1} for 
 $b \in \{300,400,500,550,600,$ $700,800\}$ and $\eps\in \{0.04,0.02,0.01,0.005\}$
 and compare the numerically found eigenvalue $\lambda_1^h$ to the analytical known values $\lambda_1$ provided in Table~\ref{tab:num:eigval_balls}. The relative errors $\eta_1^\lambda :=|\lambda_1 - \lambda_1^h|/\lambda_1$ are presented in 
 Figure~\ref{fig:num:lam1_over_eps}.
  From Figure~\ref{fig:num:lam1_over_eps}, we obtain that for the scaling $b(-1) = b\eps^{-4/3}$ the choice of $b = 550$ is optimal in this situation and thus, in the following, we fix $b = 550$.  
  
  For $\eps = 0.02$ and $b=550$, the eigenfunction $w_1^{D,h}$ related to the eigenvalue $\lambda_1^h$ approximates the analytical $w_1^D$  with a relative error
  $\|w_1^D-w_1^{D,h}\|_{L^2(\Omega)} / \|w_1^D\|_{L^2(\Omega)} = 12\cdot 10^{-4}$. Here, $w_1^D$ is a scaled Bessel function, see~\cite[Prop.~1.2.14]{Henrot}. 
  It is extended to the whole computational domain $\Omega$ with the constant value zero.

   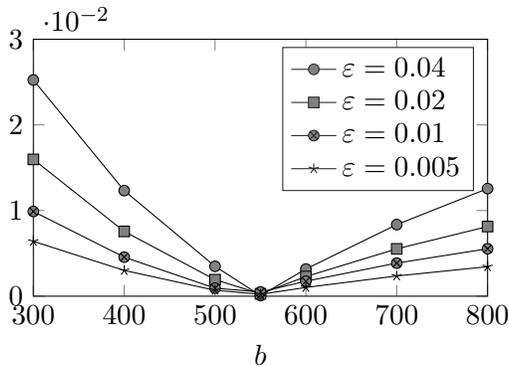
\begin{figure}
      \centering
      \begin{tikzpicture}

\begin{axis}[%
width=0.5\textwidth,
height=5cm,
xmin=300,
xmax=800,
xminorticks=true,
xlabel={$b$},
ymin=0.000001,
ymax=0.03,
yminorticks=true,
axis background/.style={fill=white},
cycle list name=black white,
legend style={legend cell align=left, align=left, draw=white!15!black},
legend pos={north east}
]
\addplot 
  table[row sep=crcr]{%
800	0.0125635422672526\\
700	0.00835998959710818\\
600	0.00317292330894519\\
550	7.22560652575229e-05\\
500	0.00348847769639253\\
400	0.0123284176458366\\
300	0.0252487268153804\\
};
\addlegendentry{ $\eps =0.04$}

\addplot+ 
  table[row sep=crcr]{%
800	0.00813296987427096\\
700	0.00552451353407975\\
600	0.002292943234061\\
550	0.000281235852548621\\
500	0.00192984652982375\\
400	0.00756524168147051\\
300	0.0159737368260996\\
};
\addlegendentry{ $\eps = 0.02$}

\addplot+ 
  table[row sep=crcr]{%
800	0.00553739330499591\\
700	0.00386266531447425\\
600	0.00174670884699903\\
550	0.000493793353983073\\
500	0.000934141162839679\\
400	0.00455926001863171\\
300	0.00989112740652888\\
};
\addlegendentry{ $\eps = 0.01$ }

\addplot+ 
  table[row sep=crcr]{%
800	0.00344532495964745\\
700	0.0023605620313711\\
600	0.00102076312642673\\
550	0.000223909863620299\\
500	0.000693636209769698\\
400	0.00300814205176406\\
300	0.0064068466440353\\
};
\addlegendentry{ $\eps = 0.005$}

\end{axis}

\end{tikzpicture}%
      \caption{The relative error $\eta_1^\lambda = |(\lambda_1 - \lambda_1^h)|/ \lambda_1$ when minimizing $\lambda_1$ in two dimensions for several values of $b$ and $\eps$. Here $b_\eps(-1) = b\eps^{-4/3}$.
      }
      \label{fig:num:lam1_over_eps}
  \end{figure}
  We validate our choice by solving the optimization problem related to $\lambda_2$ and $\lambda_3$ for $\eps = 0.02$. 
  The global optimal solutions are stated in Theorem~\ref{thm:num:min_lam2}
 and Theorem~\ref{thm:num:min_lam3}. 
     \begin{Thm}[{\cite[Thm.~4.1.1]{Henrot}}]
  \label{thm:num:min_lam2}
            The minimum of $\lambda_2(D)$ among all bounded open sets $D \subset \mathbb{R}^d$, $d\in\N$, 
            with given volume is achieved by the union of two identical disjoint balls.
  \end{Thm}
  
  \begin{Thm}[{\cite[Cor.~5.2.2]{Henrot}}]
  \label{thm:num:min_lam3}
    The minimum of $\lambda_3(D)$ among all bounded open sets $D \subset \mathbb{R}^d$, $d \in \{2,3\}$, 
      with given volume is achieved by one ball.
 \end{Thm} 
 
\begin{Rem}[Eigenvalues with multiplicity larger one]
\label{rm:num:multiEV}
In the situation of Theorem~\ref{thm:num:min_lam1}, $\lambda_2$ has multiplicity two, 
while in the situation of Theorem~\ref{thm:num:min_lam3}, $\lambda_3$ has multiplicity equal to the spatial dimension.
If eigenvalues have multiplicity larger one,
the corresponding gradient is no longer unique and depends on the random ordering that the eigenvalue solver provides for these identical eigenvalues.
This problem can be detected by considering the relative difference of subsequent eigenvalues during the optimization run.
If this problem is detected for an eigenvalue of multiplicity two,
we modify the objective functional to minimize the arithmetic mean of these equal eigenvalues. 
As both eigenvalues are equal, this does not change the value attained by the objective at the current local optimum.
We stress, that changing the objective functional in advance does lead to a different optimization problem and thus, we typically obtain different local minimizers. 
We also refer to \cite[Sec.~4.5]{Oudet} for more details.

In practice, we notice that this modification does not work for eigenvalues with a multiplicity larger than two. 
This is because it is rather unlikely, that the pairwise relative difference of more than two eigenvalues becomes small at the same time and thus jointly trigger the modification of the objective.
In this situation, the above modification actually changes the objective and we would thus solve a different problem. 
This situation appears, for instance, when minimizing $\lambda_3$ in three spatial dimensions, where $\lambda_2=\lambda_3=\lambda_4$ holds for the optimal shape, namely a ball. Luckily, in this situation, by solving the modified optimization problem we still detect the correct minimizer predicted by Theorem~\ref{thm:num:min_lam3} if $\gamma$
is initially chosen sufficiently large and decreased subsequently. 
\end{Rem}
 
The global optimal solutions are successfully found in two and three spatial dimensions. 
In Figure~\ref{fig:num:lam5-min} we show optimal shapes for minimizing $\lambda_i$, $i=1,2,3$ in two spatial dimensions.
We note, that in case of minimizing $\lambda_3$ there also is an attracting local minimum, containing three small balls.
Here, we need to start with a large value of $\gamma=10^3$ to guide the optimization process to the correct global optimum in combination  
with a homotopy of decreasing values for $\gamma$ towards $\gamma=1$. 
Moreover, in three dimensions we need to substitute the minimization problem for $\lambda_3$ by 
$\frac{1}{3}(\lambda_2+\lambda_3+\lambda_4)$ to deal with the multiple eigenvalue. In two dimensions this problem is handled as described in Remark~\ref{rm:num:multiEV}. 
The correct topologies are found and in Table~\ref{tab:num:lam123}, we present our numerical results in terms of the eigenvalues that we obtained. 

\begin{table}
    \centering
    \begin{tabular}{c|rrr|rrr}
         &  \multicolumn{3}{c|}{2D, $|\{\varphi >0\}| = \frac{1}{4}$} & \multicolumn{3}{c}{3D, $|\{\varphi >0\}| = \frac{1}{8}$} \\
       $k$  &  \multicolumn{1}{c}{$\lambda_k$} & \multicolumn{1}{c}{$\lambda_k^h$} & \multicolumn{1}{c|}{$\eta_k^\lambda$} 
            &   \multicolumn{1}{c}{$\lambda_k$} &  \multicolumn{1}{c}{$\lambda_k^h$} & \multicolumn{1}{c}{$\eta_k^\lambda$} \\
       \hline
       1  &  72.67 & 72.68  & $1\cdot 10^{-4}$     & 102.59 & 101.43 & $113\cdot 10^{-4}$\\
       2  & 145.34 & 143.16 & $150\cdot 10^{-4}$   & 162.84  & 162.40     & $27\cdot 10^{-4} $  \\
       3  & 184.50 & 183.53 & $54\cdot 10^{-4}$    & 209.86 & 209.21 & $31\cdot 10^{-4}$ \\
    \end{tabular}
    \caption{Analytical and numerically found eigenvalues related to minimizing $\lambda_k$, $k\in\{1,2,3\}$ in two and three spatial dimensions.
    Here $\lambda_k$ denotes the analytical value of the $k$-th eigenvalue, $\lambda_k^h$ denotes the numerical found approximation of this value and 
    $\eta_k^\lambda := {|\lambda_k-\lambda_k^h|}/{\lambda_k}$ denotes the related relative numerical error. 
     }
    \label{tab:num:lam123}
\end{table}

\paragraph{One additional example with known solution.}
As another example for which a reference solution exists, we consider the minimization of $\lambda_5$.
From \cite[Fig.~11.1]{AntunesOudet-NumericalResultsExtremalProblemsEigenvaluesLaplace} 
we expect a butterfly-like shape and the proposed eigenvalue is $\lambda_5 = 312.60$. 
In contrast to the simulation above, we use $\gamma=0.1$ to get closer to the features of the butterfly.
In Figure~\ref{fig:num:lam5-min} (right),
we show our numerically obtained shape together with the corresponding eigenfunction $w_5^{D,h}$. 
We obtain the eigenvalue $\lambda_5^h = 311.59$. The normalized amplitude of $w_5^{D,h}$ is $\|w_5^{D,h}\|_{L^\infty(\Omega)} = 3.94$,
while the amplitude of $w_5^{D,h}$ on the zero-level line of $\varphi^h$ 
is of order $\|w_5^{D,h}\|_{L^\infty(\{\varphi^h = 0\})} = 0.50$. 
We point out that the result proposed in  \cite[Fig.~11.1]{AntunesOudet-NumericalResultsExtremalProblemsEigenvaluesLaplace}
is more pronounced in the middle part and also on the left and right there are additional small deflections.

\begin{figure}
    \centering
    \includegraphics[width=0.22\textwidth]{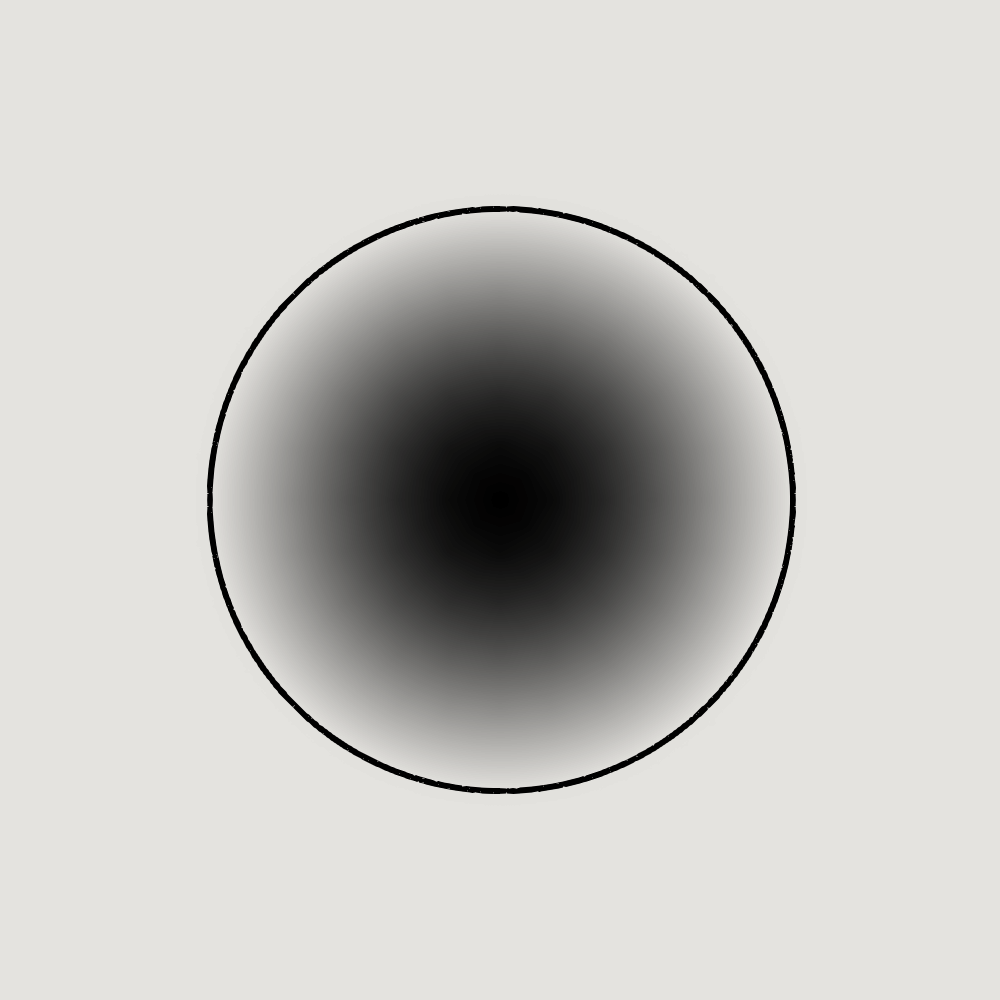}
    \includegraphics[width=0.22\textwidth]{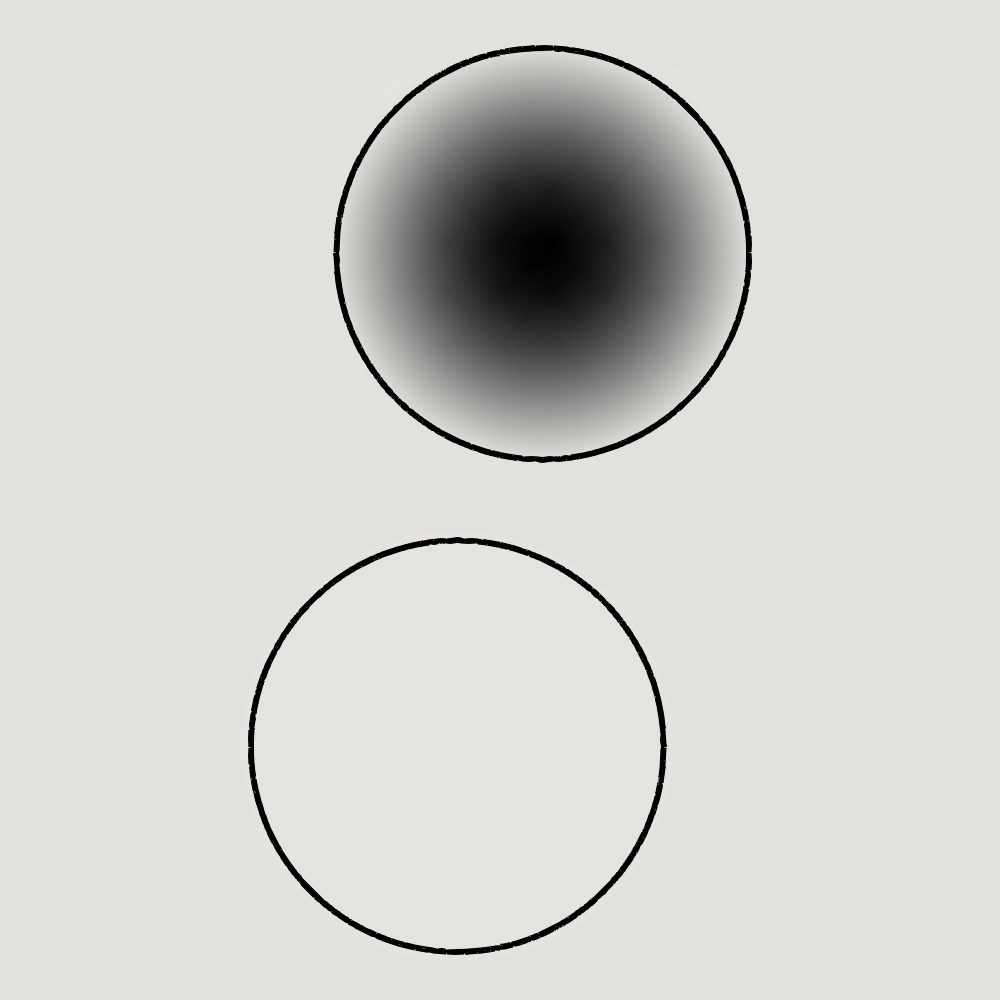}
    \includegraphics[width=0.22\textwidth]{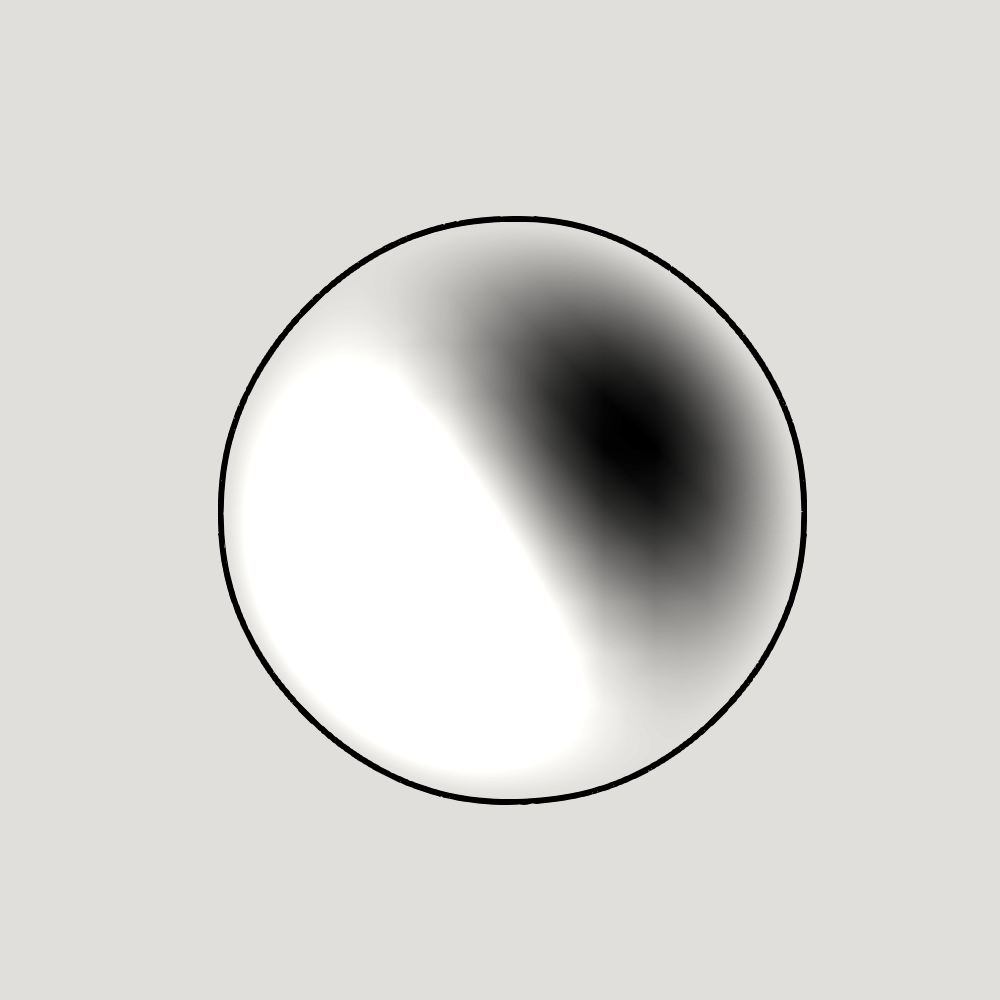}
    \includegraphics[width=0.22\textwidth]{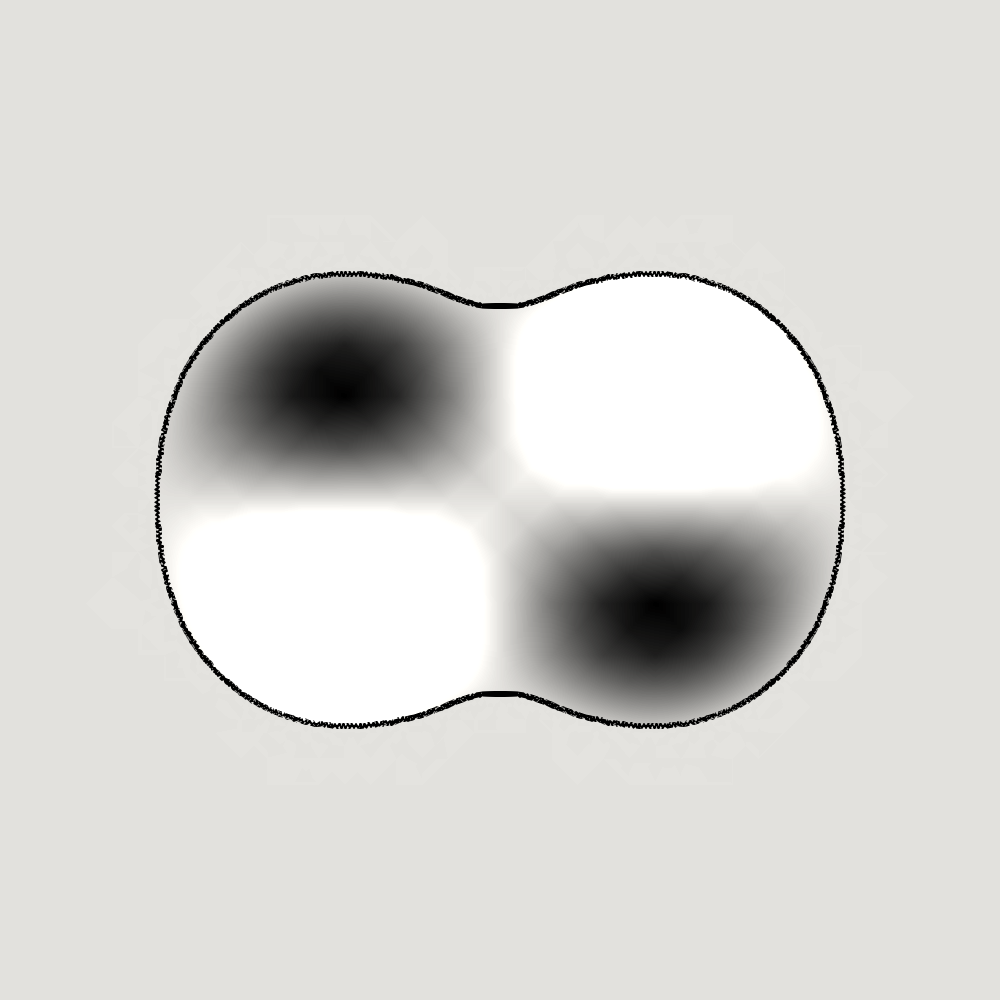}
    \caption{The optimal shape for minimization of $\lambda_i$, $i=1,2,3,5$ (left to right) is presented by the
    zero level line of $\varphi^h$ black in each case, 
    while the value of the corresponding eigenfunction $w_i^{D,h}$ is shown in gray scale.
    The gray outer domain corresponds to $w_i^{D,h} \approx 0$.
    Note that in the case of minimizing $\lambda_2$, there is a second eigenfunction (to same eigenvalue) that is supported on the bottom circle, while in the case of minimizing $\lambda_3$, there is a second eigenfunction (to same eigenvalue) that is rotated by $90\degree$.
    }
    \label{fig:num:lam5-min}
\end{figure}

 \paragraph{Fixing $c$ for Neumann boundary data.}
  Here we proceed as in the case of fixing $b$. 
 We solve the corresponding maximization problem for $\mu_1$ for $c \in \{0.2,0.15,0.1,0.05\}$ and $\eps \in \{0.04,0.02,0.01,0.005\}$ 
 on $\Omega = (0,1)^2$.
In this situation, the optimal shape is a disc of radius $r=\sqrt{{1}/{(4\pi)}}$ with first eigenvalue $\mu_1 = 42.6002$. This can be obtained from \cite[Prop.~1.2.14]{Henrot}.  
 
 For the sake of brevity, we omit the presentation of relative errors as in Figure~\ref{fig:num:lam1_over_eps} and just state
 that we observe that $c=0.1$ is a good choice independent of $\eps$. In the following, we fix $c=0.1$.

\subsection{Numerical examples without known solution}
\label{sec:num:noAnalyticalSolution}
In the following, we investigate some numerical examples where the analytical solution is unknown
in order to show the strength of our approach in finding unknown shapes with a priori unknown topologies. We also remark that the boundary of $\Omega$ acts as an obstacle which can be seen in several computations below. We refer to \cite{Harrell,Henrot-Zucco} for more information on obstacle type problems for eigenvalues.

\paragraph{Minimization of $\mu_1$.}
As discussed above, the \textit{maximization} of $\mu_1$ leads to a disc. 
Now we ask for the optimal shape and topology when \textit{minimizing} $\mu_1$ on $\Omega = (0,1)^2$.
In Figure~\ref{fig:num:min_mu1-phi}, we present the found optimal topology and the
influence of $\gamma$ on the result by showing the iterates for a homotopy reducing $\gamma$ from $10^3$ to $10^{-3}$.
In Figure~\ref{fig:num:min_mu1-u}, we present the first three corresponding non-trivial eigenfunctions on the optimal topology for $\gamma=10^{-3}$. We observe that the boundary of $\Omega$ might act as an obstacle for the shape optimization problem.

\begin{figure}
    \centering
    \includegraphics[width=0.25\textwidth]{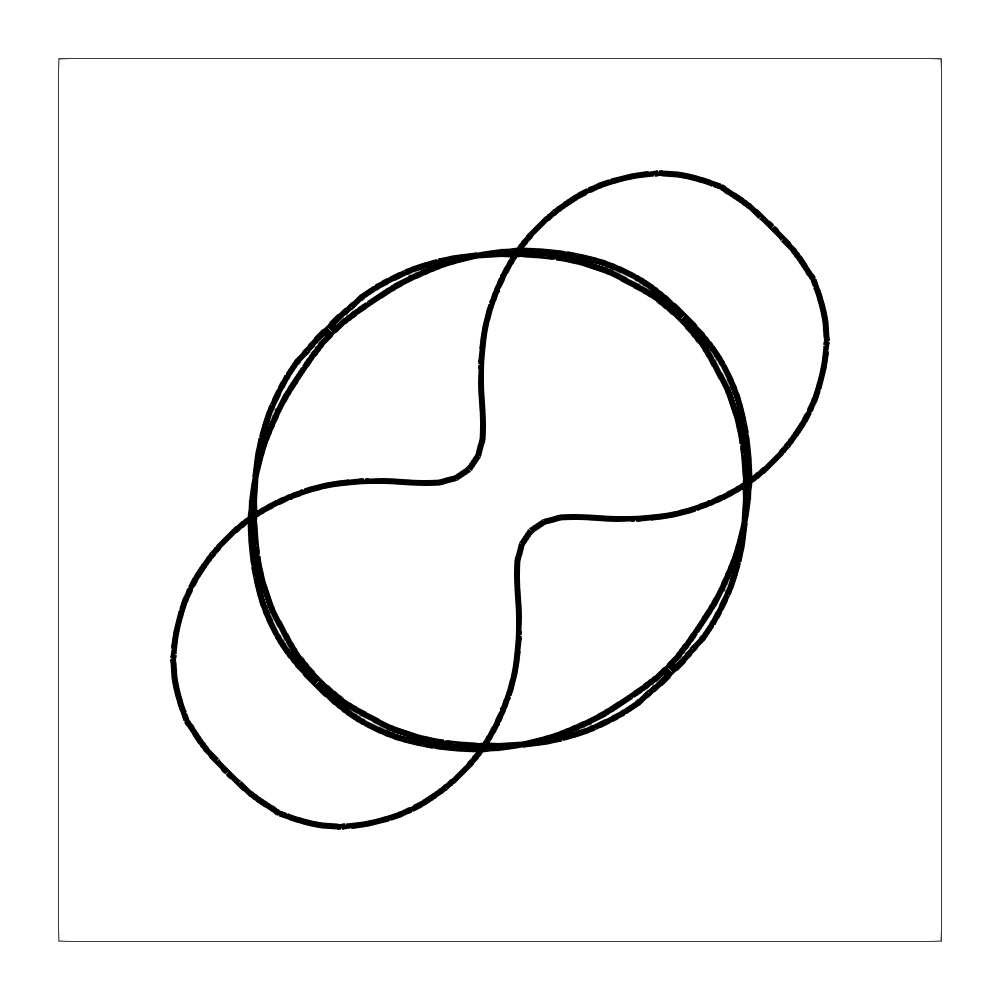}
        \includegraphics[width=0.25\textwidth]{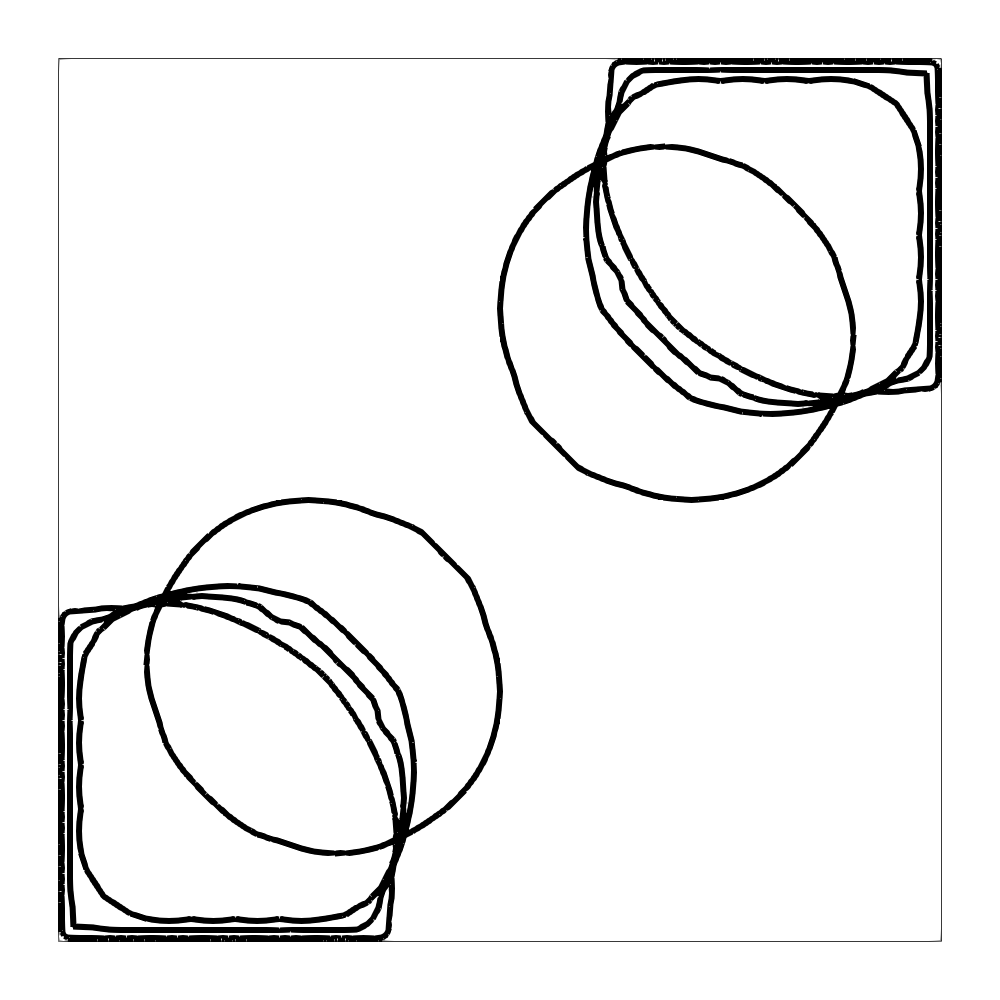}\\
    
    \caption{The zero level lines of $\varphi^h$ for minimizing $\mu_1$ with $\gamma\in \{10^3,10^2,10^1\}$ (left) and with $\gamma \in \{10^0,10^{-1},10^{-2},10^{-3}\}$ (right). We observe that a large value of $\gamma$ leads to a circle and that decreasing $\gamma$ allows the optimizer to find topologies with longer boundaries.}
    \label{fig:num:min_mu1-phi}
\end{figure}

\begin{figure}
    \centering
        \includegraphics[width=0.25\textwidth]{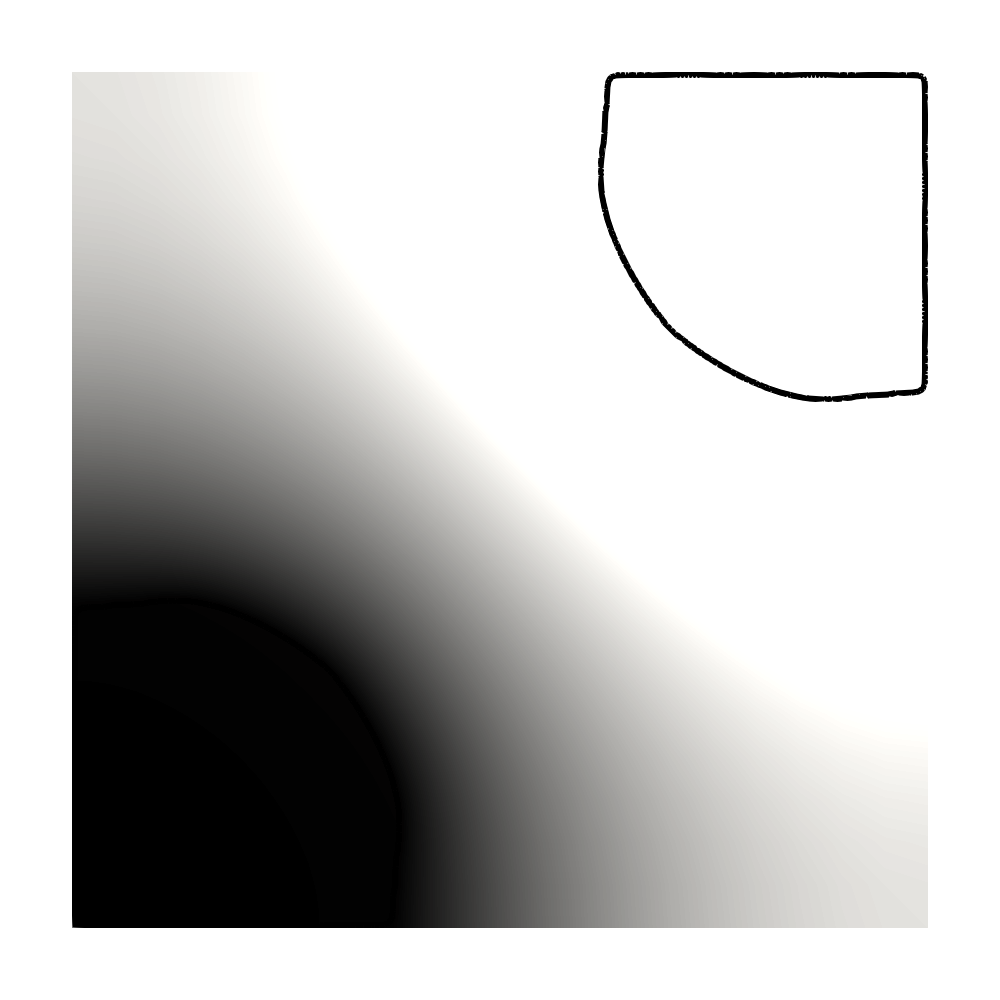}
        \includegraphics[width=0.25\textwidth]{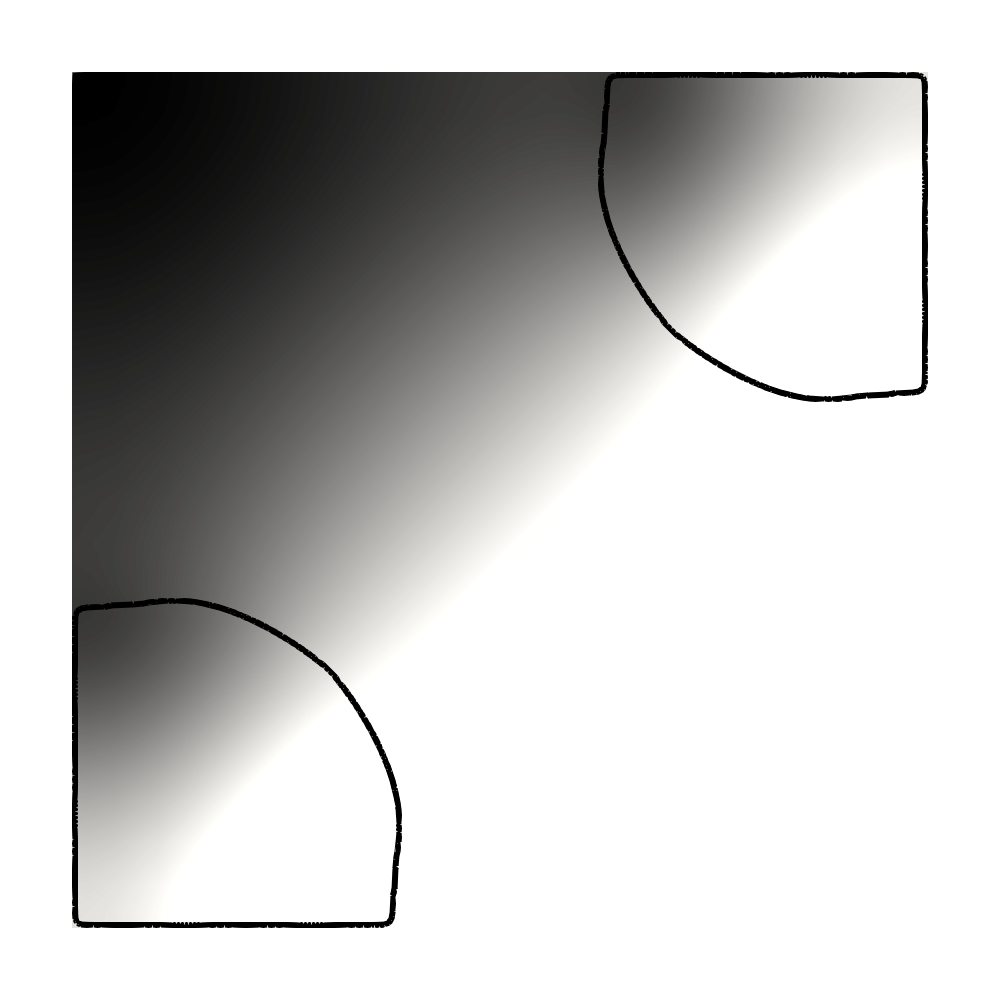}
        \includegraphics[width=0.25\textwidth]{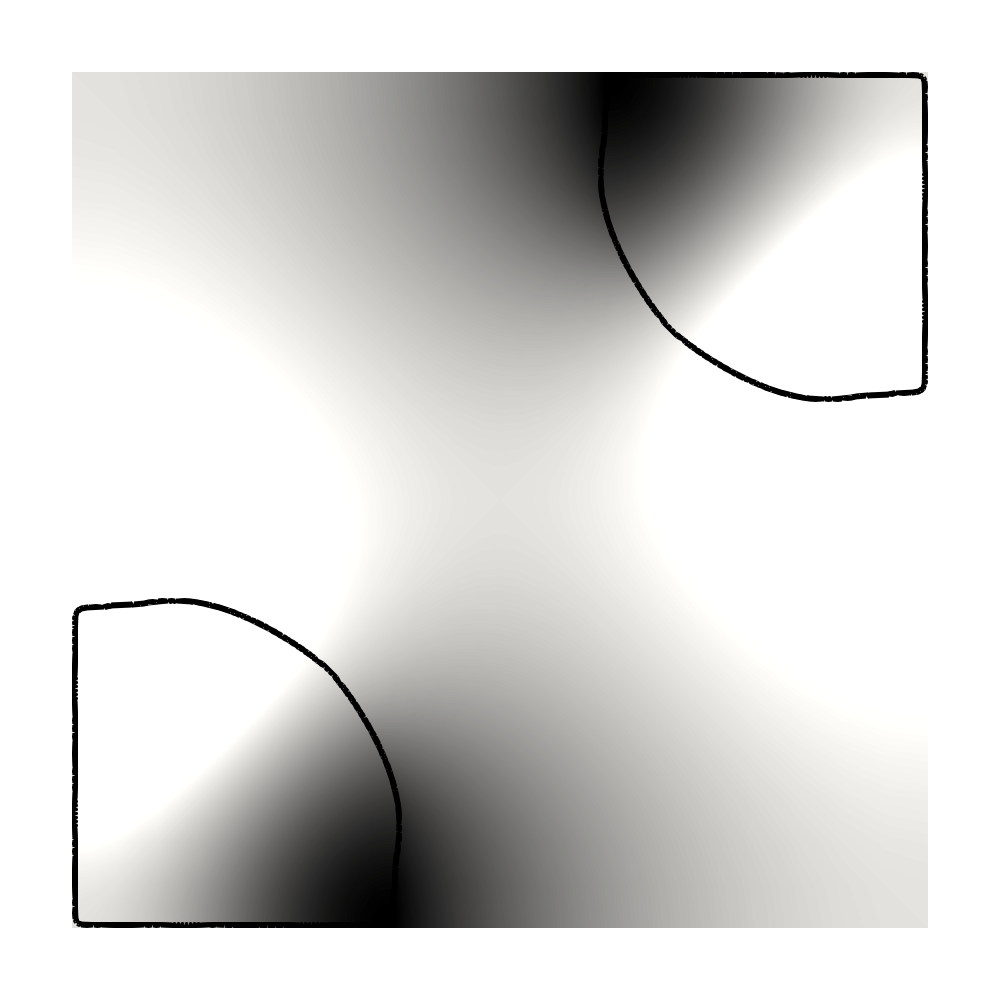}
    \caption{When minimizing $\mu_1$ we obtain the shape that is indicated by black lines, which is the zero-level line of $\varphi^h$.
    From left to right, we show the first three (non-trivial) eigenfunctions $w_1^{N,h}$, $w_2^{N,h}$, and $w_3^{N,h}$ on this shape.
    Here, gray corresponds to $w_i^{N,h} \approx 0$.
    By our approach, the eigenfunctions are defined on the complete domain $\Omega$, and as we are considering the Neumann case, they do not degenerate on the complement of the shape.
    The corresponding eigenvalues are $\mu_1^h = 0.40$, $\mu_2^h = 67.55$, and $\mu_3^h = 68.23$.
    In this example we chose $\gamma=10^{-3}$.
    }
    \label{fig:num:min_mu1-u}
\end{figure}

\paragraph{Mixing minimization and maximization.}
In this example, we consider the minimization of a weighted sum of eigenvalues. Especially, 
we consider weights with different signs which leads to simultaneous minimization and maximization of certain eigenvalues.
In Figure~\ref{fig:num:mixed-lambda}, we present numerical results for the objective
$J^\lambda = \frac{6\lambda_1 -  \lambda_3}{7} + E^{0.02}(\varphi)$ and in 
Figure~\ref{fig:num:mixed-mu}, we present numerical results for the objective
$J^\mu = \frac{12\mu_1 -\mu_3-\mu_4}{14} + 10 E^{0.02}(\varphi)$.

\begin{figure}
    \centering
    \includegraphics[width=0.25\textwidth]{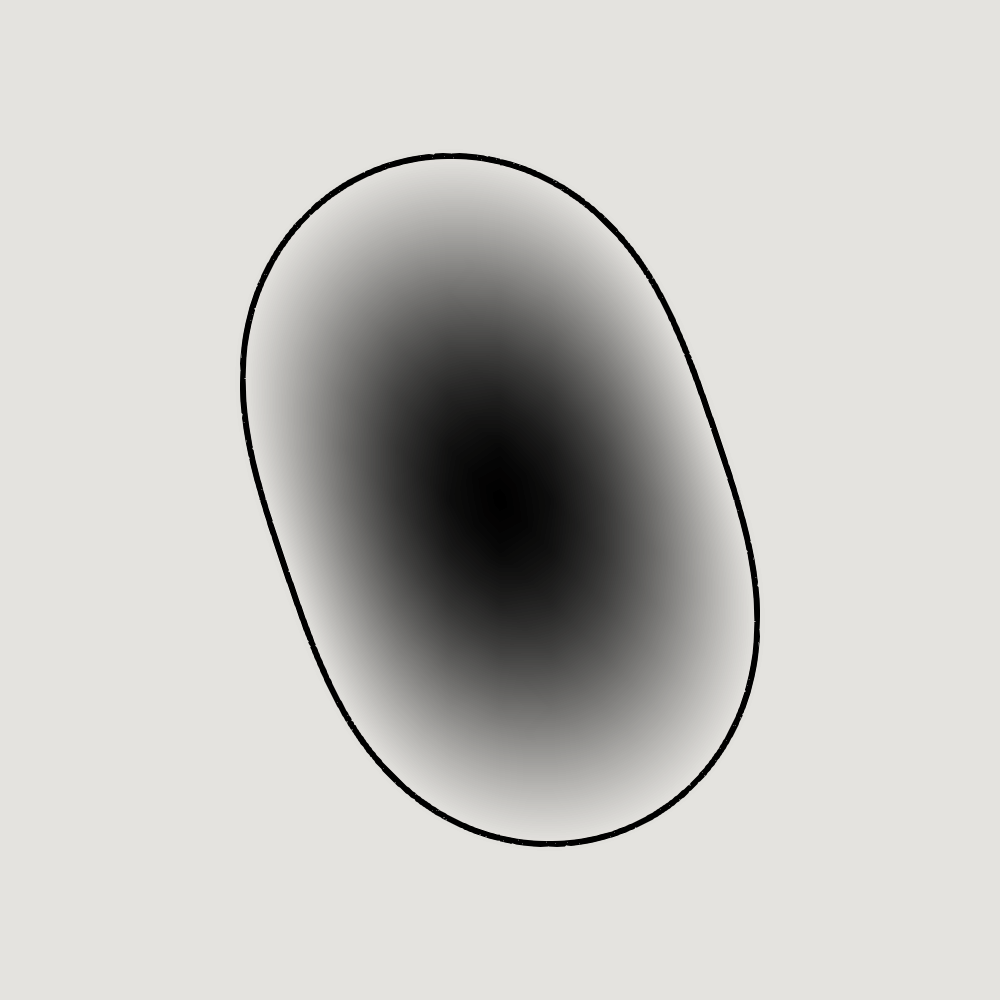} 
    \includegraphics[width=0.25\textwidth]{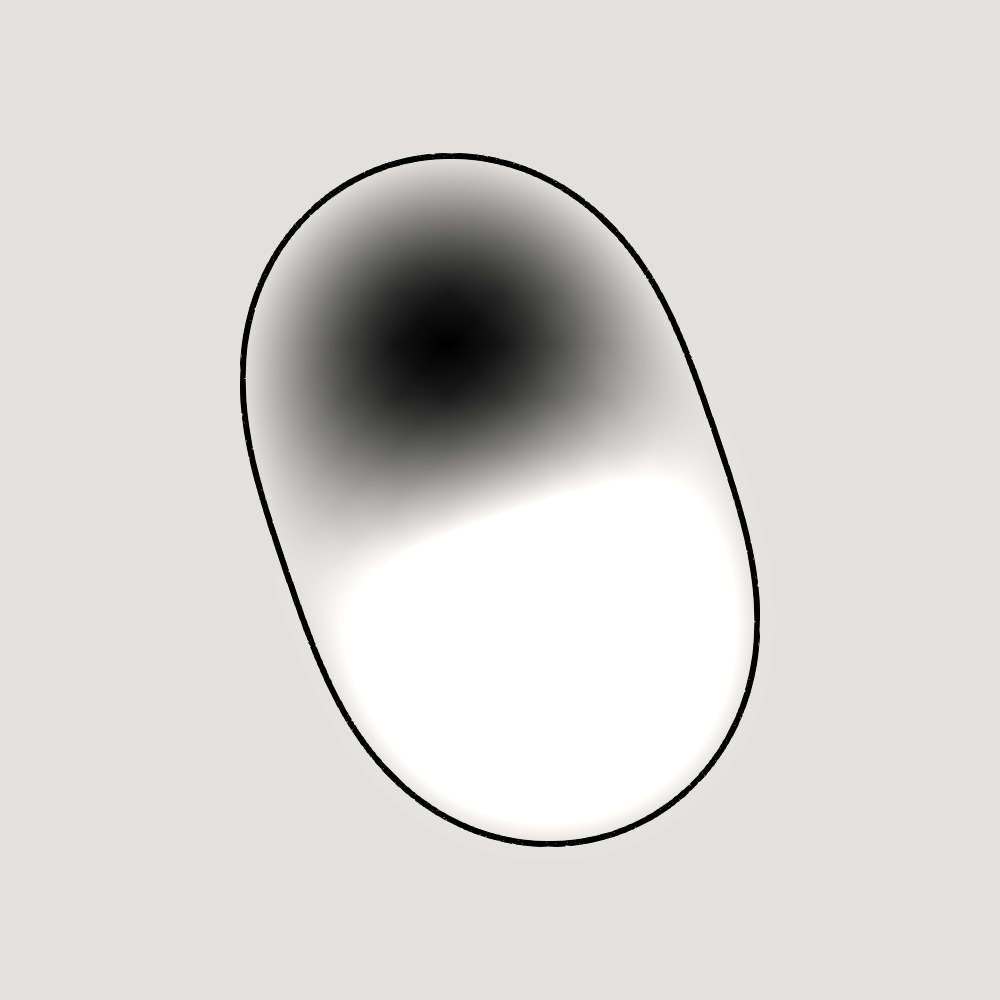} 
    \includegraphics[width=0.25\textwidth]{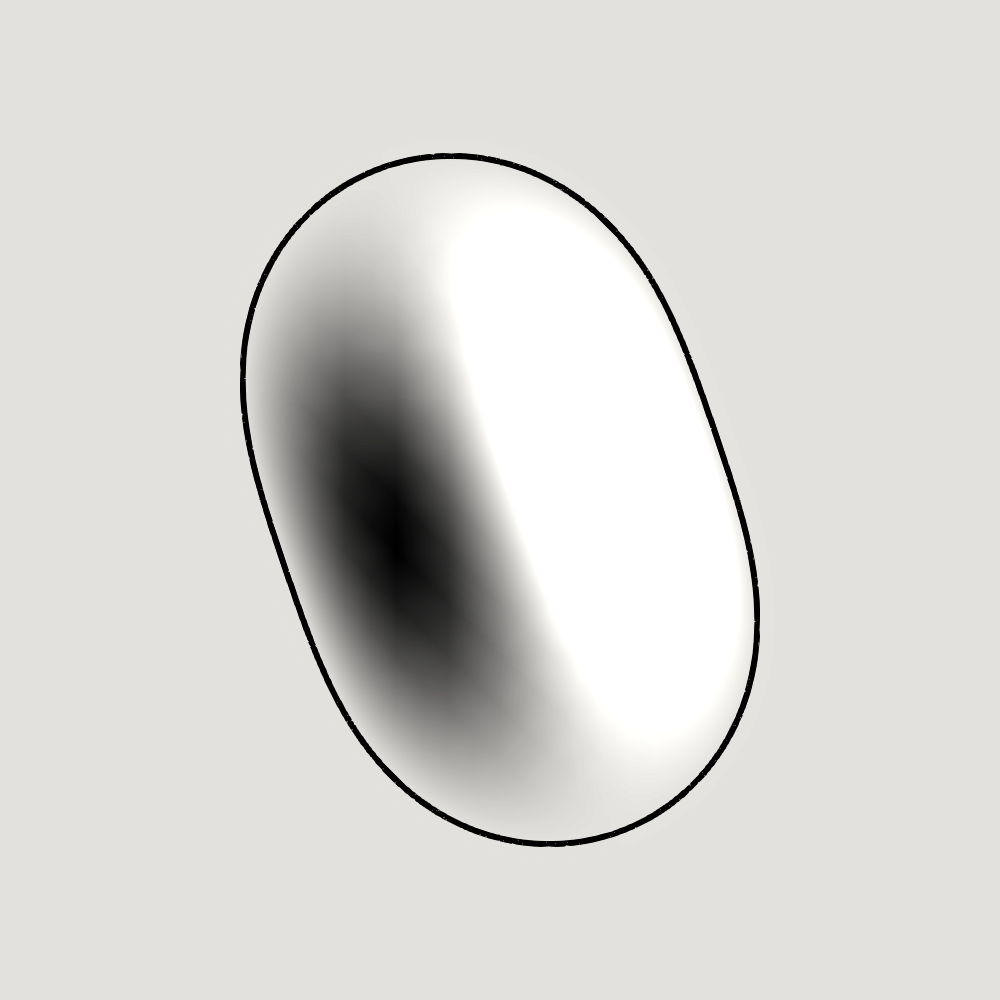} 
    \caption{Optimal shapes for an example of mixed minimization and maximization, 
    namely $J^\lambda= \frac{6\lambda_1 -  \lambda_3}{7} + E^{0.02}(\varphi)$.
    The optimal shape is indicated by the zero-level line of $\varphi^h$ in black, 
    and we show $w^{D,h}_1$, $w^{D,h}_2$, and $w^{D,h}_3$ in gray scale (left to right). Gray indicates the zero level of the eigenfunctions.
    The corresponding eigenvalues are $\lambda_1^h = 81.32$, $\lambda_2^h = 161.01$, and $\lambda_3^h = 255.42$. 
    }
    \label{fig:num:mixed-lambda}
\end{figure}

\begin{figure}
    \centering
    \includegraphics[width=0.25\textwidth]{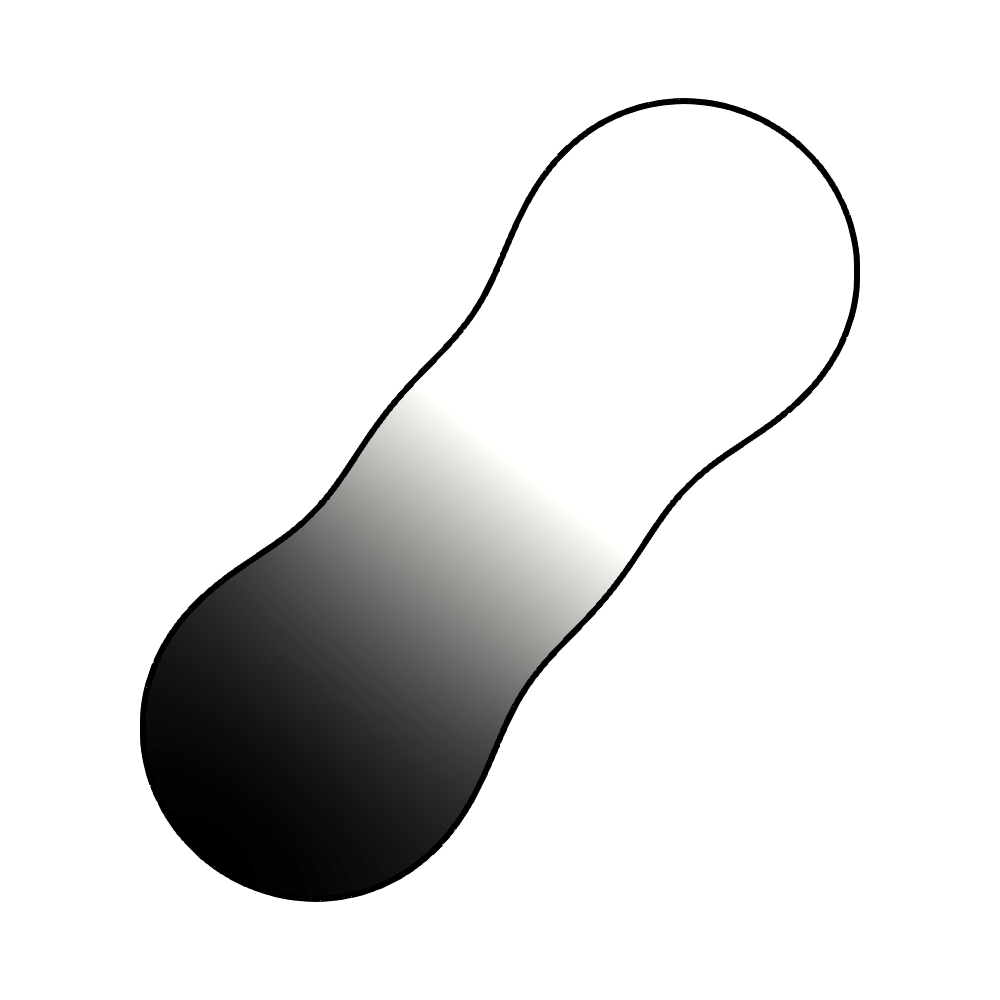} 
    \includegraphics[width=0.25\textwidth]{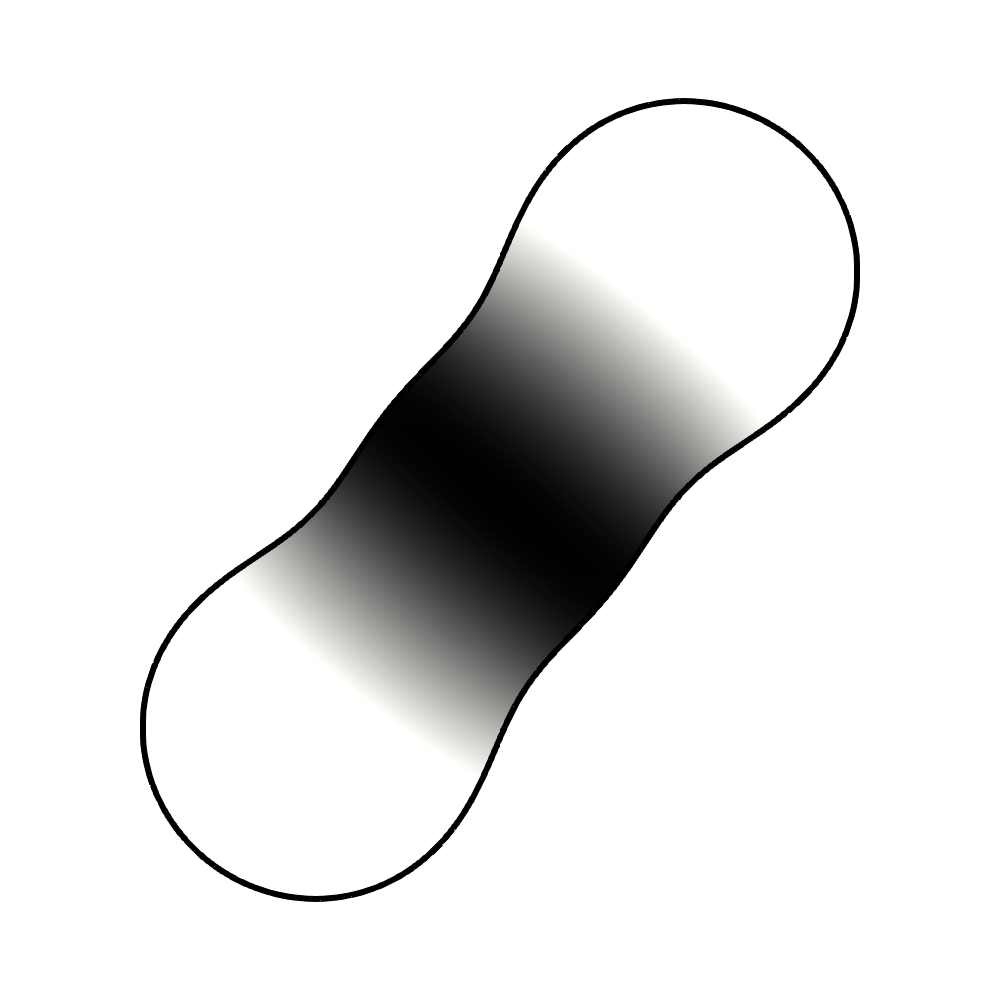} 
    \includegraphics[width=0.25\textwidth]{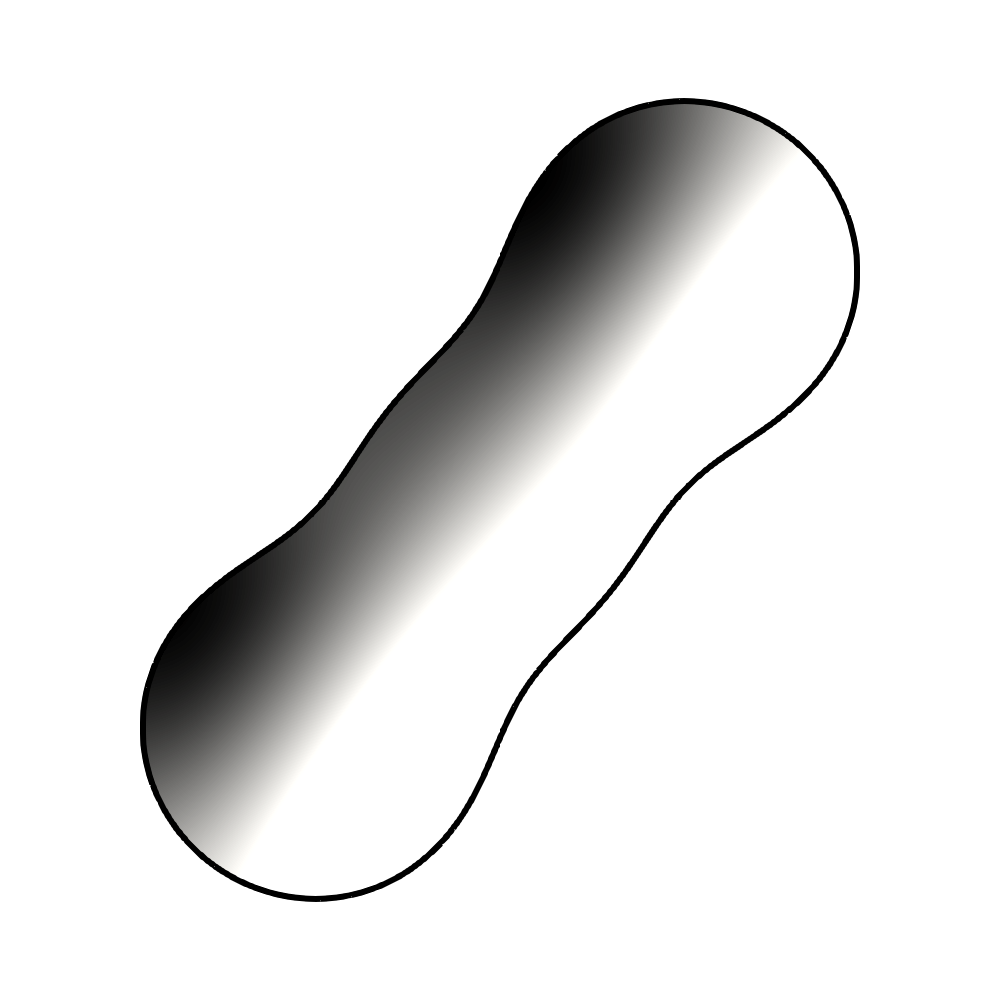} 
    \caption{Optimal shapes for an example of mixed minimization and maximization, 
    namely $J^\mu= \frac{12\mu_1 -\mu_3-\mu_4}{14} + 10 E^{0.02}(\varphi)$.
      The optimal shape is indicated by the zero-level line of $\varphi^h$ in black, 
    and we show $w^{N,h}_1$, $w^{N,h}_2$, and $w^{N,h}_3$ in gray scale (left to right). Gray indicates the zero level of the eigenfunctions. 
    The corresponding eigenvalues are $\mu_1^h = 12.94$, $\mu_2^h = 56.94$, and $\mu_3^h = 115.27$.
    Here, the eigenfunctions $w_i^{N,h}$, $i=1,2,3$, are plotted only on the actual shape.
    }
    \label{fig:num:mixed-mu}
\end{figure}

\begin{figure}
    \centering
    \includegraphics[width=0.25\textwidth]{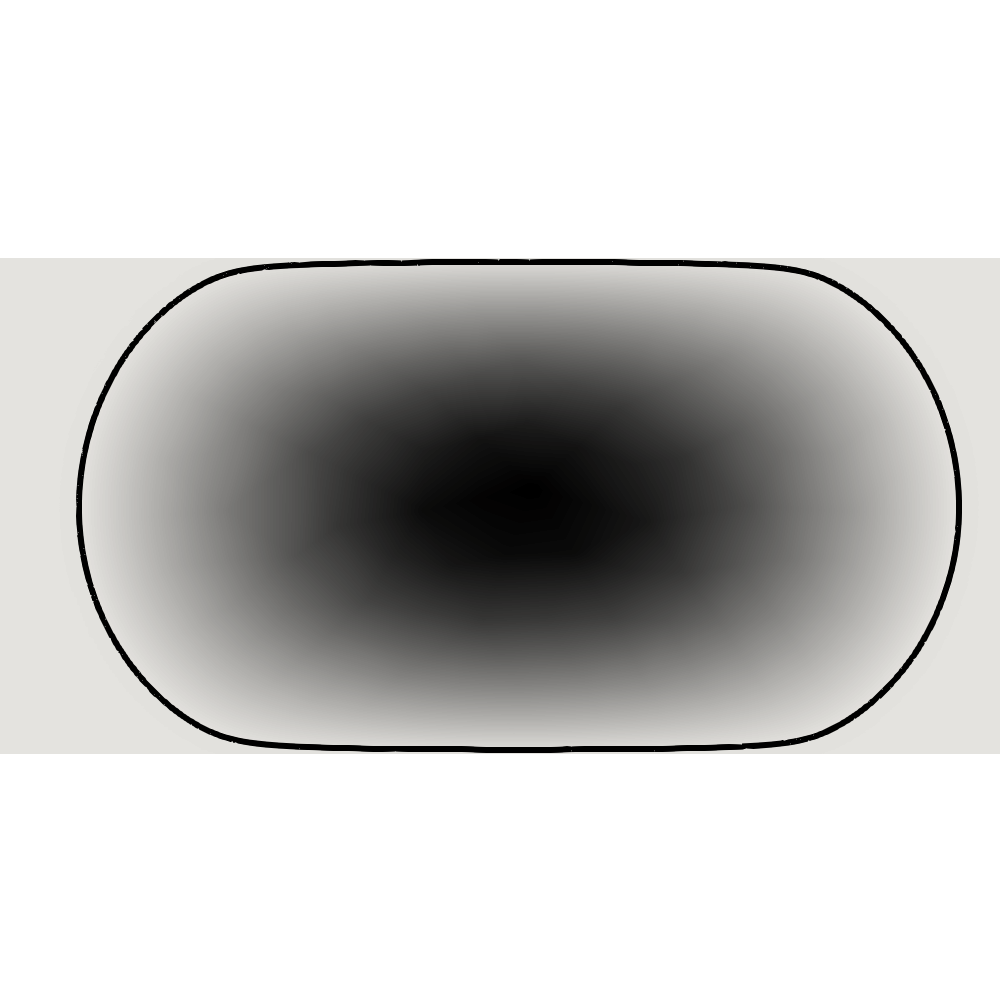}
    \hspace{0.5cm}
    \includegraphics[width=0.25\textwidth]{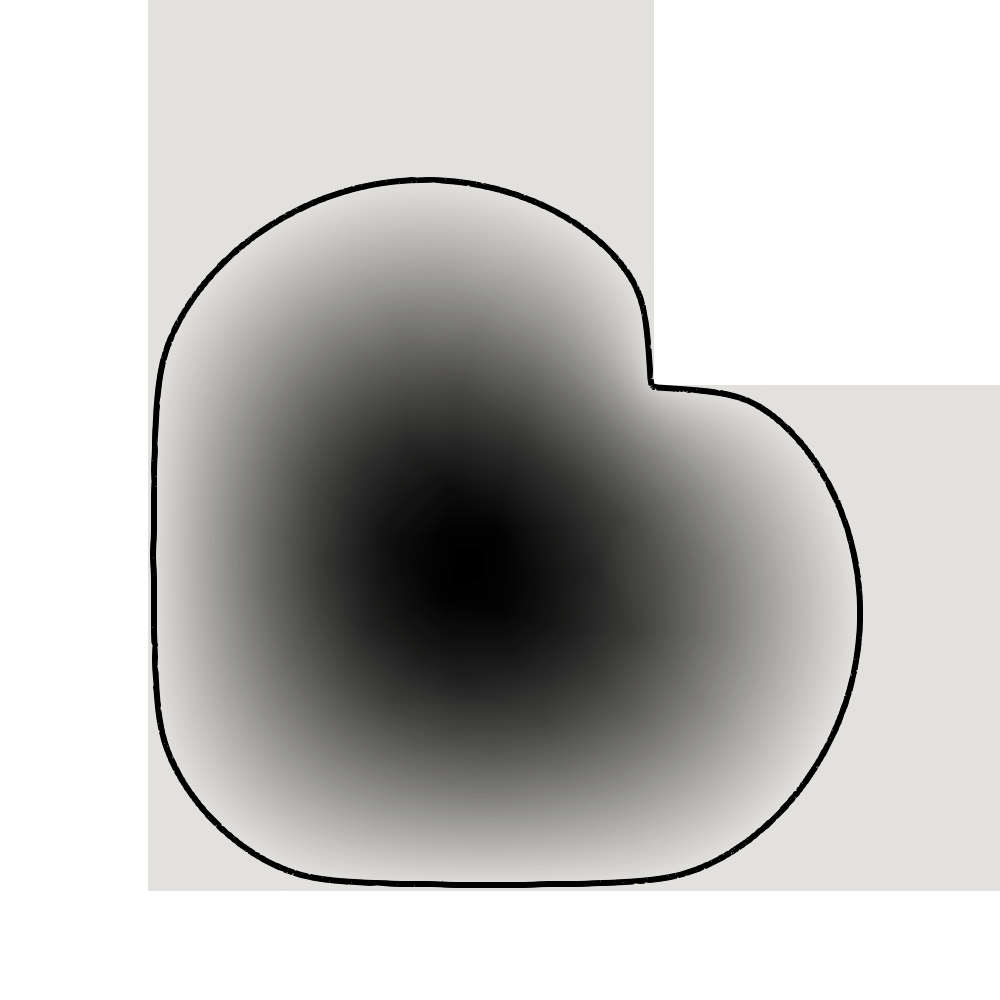}
    \caption{The optimal shapes for minimizing $\lambda_1$ on the rectangular domain $\Omega_1 = (0.0,2.5) \times (0.0,0.4)$ (left)
    and the  L-shaped domain $\Omega_2 = (0,1.45)^2 \backslash (0.4,1.45)^2$ (right).
    The shapes are indicated by the zero level line of $\varphi_\eps^h$ in black and we show the corresponding first 
    eigenfunction $w_1^{D,h}$ in grayscale. Gray indicates the zero level of $w_1^{D,h}$. 
    We show only the relevant part of the computational domain.
    The corresponding eigenvalues are $\lambda_1^h(\Omega_1) = 85.54$ and $\lambda_1^h(\Omega_2) = 77.13$.
    A disc of the same size would lead to $\lambda_1 = 72.68$ as stated in Table~\ref{tab:num:lam123}.}
    \label{fig:num:tight}
\end{figure}

\begin{figure}
    \centering
    \fbox{
    \includegraphics[width=0.3\textwidth]{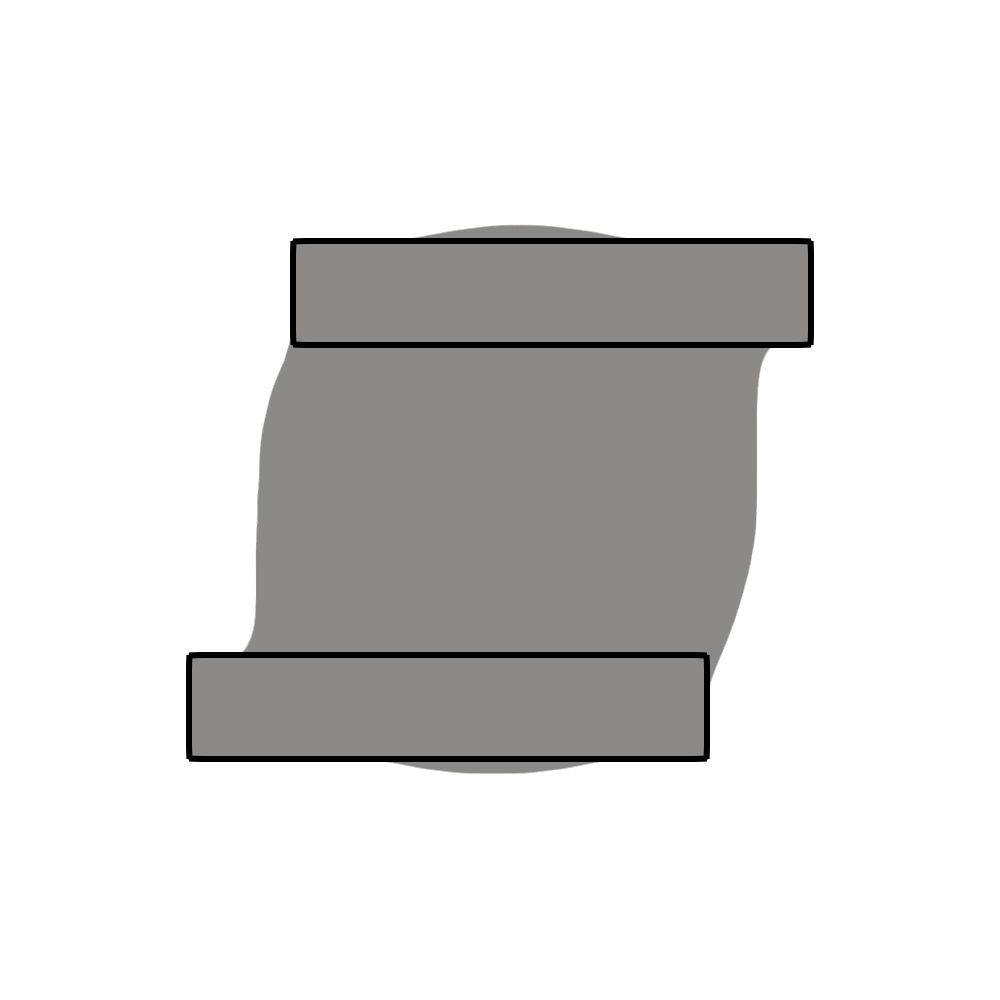}
    } 
    \hspace{10pt}
    \fbox{\includegraphics[width=0.3\textwidth]{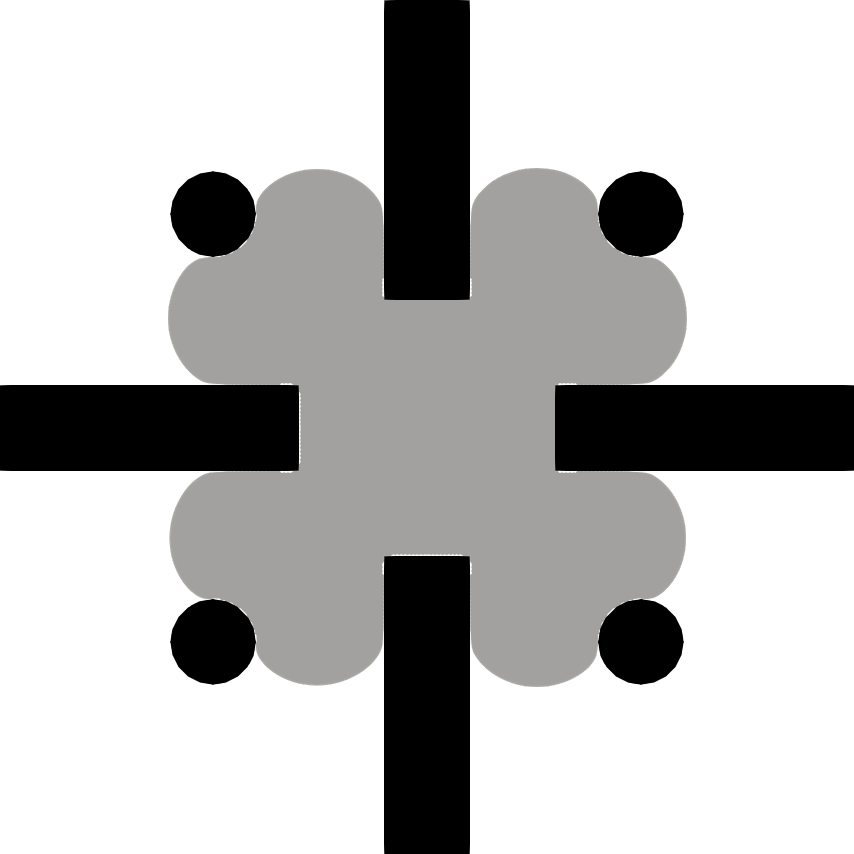}}
    \caption{Numerically obtained optimal shapes in gray for minimizing $\lambda_1$. 
    On the left we fix the domains inside the gray boxes as part of the shape, i.e., $\varphi^h=1$, while on the right we fix
    the black domains to be void, i.e., $\varphi^h = -1$.}
    \label{fig:num:fixSomeStructureVoid}
\end{figure}

\paragraph{Influences from $\Omega$.}
As stated in Theorem~\ref{thm:num:min_lam1}, the minimizer of the first eigenvalue is one single disc of diameter $d = 2\sqrt{{1}/{(4\pi)}} \approx 0.56$. 
Here, we show the optimal shape in the case that this ball does not fit into the computational domain. This leads to an obstacle like problem where the boundary of $\Omega$ acts as an obstacle, see also \cite[Sec.~3.4]{Henrot}.

We consider two cases, namely $\Omega_1 = (0.0,2.5) \times (0.0,0.4)$ which is a rectangular domain of height $0.4 \leq d$, and $\Omega_2 = (0,1.45)^2 \backslash (0.4,1.45)^2$
which is an L-shaped domain. 
In both cases, a disc of diameter $d \approx 0.56$ does not fit into the domain. 
In this example we fix $\varphi^h = -1$ on $\partial\Omega$ to prevent the shape from touching the boundary.
In Figure~\ref{fig:num:tight} we present numerical results for the minimization of $\lambda_1$ in this situation.  

\paragraph{Prescribing parts of the optimal topology.}
Finally, we show another aspect of the flexibility of the proposed approach. We present two examples, in which we a-priori fix certain parts of the
design domain.
In Figure~\ref{fig:num:fixSomeStructureVoid}, 
we present numerical results obtained by either fixing some part of the domain as shape (left) or as void (right).
In both cases, we minimize $\lambda_1$ and we fix $\gamma=0.01$.

\section*{Acknowledgment}
Harald Garcke, Paul Hüttl and Patrik Knopf gratefully acknowledge the support by the Graduiertenkolleg 2339 IntComSin of the Deutsche Forschungsgemeinschaft (DFG, German Research Foundation) – Project-ID 321821685. 
Tim Laux has received funding from Deutsche Forschungsgemeinschaft (DFG, German Research Foundation) under Germany's Excellence Strategy -- EXC-2047/1 -- 390685813. The support is gratefully acknowledged.


\footnotesize\setlength{\parskip}{0cm}

\bibliographystyle{siam}
\bibliography{Laplace_REVISION}


\end{document}